\documentclass[a4paper,11pt]{amsart}
\usepackage{amsmath,amsthm,amsfonts,amssymb,upgreek,multicol,mathrsfs,pifont,amscd,latexsym,cite,bm,geometry,color,fancyhdr,abstract,framed,appendix}
\usepackage{graphicx,algorithm,algorithmic,tabularx,url,booktabs,array,threeparttable,tikz}

\usepackage{enumitem}

\usetikzlibrary{calc}
\usetikzlibrary{angles,quotes}

\usepackage{caption}
\usepackage{subfig}
\usepackage{appendix}

\newcommand{\mt}{\mathbf{t}}
\newcommand{\mn}{\mathbf{n}}
\newcommand{\ms}{\mathbb{S}}
\newcommand{\be}{\begin{eqnarray}}
\newcommand{\ee}{\end{eqnarray}}
\newcommand{\bsi}{\bm{\sigma}}
\newcommand{\bta}{\bm{\tau}}

\newcommand{\bd}{\operatorname{div}}

\newcommand{\ben}{\begin{equation*}}
\newcommand{\een}{\end{equation*}}

\newtheorem{theorem}{Theorem}[section]

\newtheorem{lemma}{Lemma}[section]
\newtheorem{co}{Corollary}[section]
\newtheorem{remark}{Remark}[section]

\numberwithin{equation}{section}
\numberwithin{table}{section}
\numberwithin{figure}{section}

\begin{document}

\title[]{Optimality of adaptive $H(\operatorname{div}\operatorname{div})$ mixed finite element methods for the Kirchhoff-Love plate bending problem}
%%%%%%%%%%%%%%%%%%%%%%%%%%%%%%%%%%%%%%%%%%%%%%%%%%%%%%%

\author[J. Hu]{Jun Hu}
\address{LMAM and School of Mathematical Sciences, Peking University,
	Beijing 100871, P. R. China. }
\address{Chongqing Research Institute of Big Data, Chongqing 401333, P. R. China}
\email{hujun@math.pku.edu.cn}

\author[R. Ma]{Rui Ma}
\address{School of Mathematics and Statistics, Beijing Institute of Technology, 
	Beijing 100081, P. R. China. }
\email{rui.ma@bit.edu.cn}
%\address{Universit\"at Duisburg-Essen, Thea-Leymann-Str. 9, {\rm{45127}} Essen, Germany. }
%\email{ rui.ma@uni-due.de }

\author{Min Zhang}
\address{College of Science, Beijing Forestry University,
	Beijing {\rm100083}, P. R. China.}
\email{zhangminD01@bjfu.edu.cn}

\thanks{The first author was supported by the National Natural Science Foundation of China under grant No. 12288101. The second author was supported by NSFC project No. 12301466 and by NSFC project No. 12422114. The third author was supported by NSFC project No. 1240510 and by grant No. BLX202347.}


%\author[R. Ma]{Rui Ma}
%\address{School of Mathematics and Statistics, Beijing %Institute of Technology, 
%	Beijing 100081, P. R. China. }
%\email{rui.ma@bit.edu.cn}

%\author{Min Zhang}
%\address{College of Science, Beijing Forestry University,
%	Beijing {\rm100083}, P. R. China.}
%\email{zhangminD01@bjfu.edu.cn}

%Address: No.35 Qinghua East Road, Haidian District, Beijing

\maketitle
\begin{abstract}
 This paper presents a reliable and efficient residual-based a posteriori error analysis for the symmetric $H(\operatorname{div}\operatorname{div})$ mixed finite element method for the Kirchhoff-Love plate bending problem with mixed boundary conditions. The key ingredient lies in the construction of boundary-condition-preserving complexes at both continuous and discrete levels.
 Additionally, the discrete symmetric $H(\operatorname{div}\operatorname{div})$ space is extended to ensure nestedness, which leads to optimality for the adaptive algorithm. Numerical examples confirm the effectiveness of the a posteriori error estimator and demonstrate the optimal convergence rate under adaptive refinements. 
\end{abstract}

\section{Introduction}
Fourth-order partial differential equations are fundamental in elasticity theory, fluid mechanics, potential theory, and many other scientific and engineering disciplines. However, despite extensive research, the development of truly practical and efficient finite element methods remains limited. Their
conforming finite element methods require $C^1$-continuity, which demands finite element spaces of high-degree polynomials and extra smoothness at vertices. 
It is not easy to implement such elements, and the expected high convergence rate is rarely manifest under quasi-uniform mesh refinements.

Nonconforming and discontinuous finite element methods are alternative methods for fourth-order problems; see \cite{morley1967triangular,ming2007nonconforming,AdiniNonFem,de1974variational,brenner2005c,suli2007hp,georgoulis2009discontinuous,cockburn2009hybridizable} for instance. Besides, mixed finite element methods have also been attractive for solving fourth-order problems, such as the Ciarlet-Raviart method \cite{ciarlet1974mixed} and the Hellan–Herrmann–Johnson (H-H-J) method \cite{hellan1967analysis,herrmann1967finite,johnson1973convergence}. However, the Ciarlet-Raviart method may fail to maintain equivalence with the original problem  on nonconvex domains \cite{zhang2008invalidity,brenner2015c}. The H-H-J method has not yet been extended to three dimensions.

The adaptive finite element method offers advantages in improving computational efficiency\cite{stevenson2008completion}. A posteriori error estimates for the fourth-order problem have been studied in various settings: for conforming $C^1$ elements\cite{verfurth2013posteriori,gustafsson2018posteriori}; for the nonconforming Morley element \cite{da2007posteriori,hu2009new, beirao2010posteriori,hu2012convergence,carstensen2014discrete}; for penalty and discontinuous Galerkin methods \cite{brenner2010posteriori,georgoulis2011posteriori,xu2014posteriori,da2007family,hansbo2011posteriori}; for the Ciarlet-Raviart mixed method \cite{charbonneau1997residual,gudi2011residual}; and for the H-H-J method \cite{huang2011convergence}. However, studies on optimality of the associated adaptive finite element methods remain relatively scarce. The main difficulty lies in the non-nestedness of the finite element spaces.
In \cite{carstensen2021hierarchical}, a hierarchical space for the Argyris element is proposed to overcome the non-nestedness in two dimensions. For the Morley element, optimality results are available in \cite{hu2012convergence,carstensen2014discrete}. The quasi-optimality of an adaptive hybridizable $C^0$ discontinuous Galerkin method for Kirchhoff plates is analyzed in \cite{sun2018quasi}, and an error estimator along with optimality results for the H-H-J method is given in \cite{huang2011convergence}.

In recent years, mixed finite element methods involving symmetric $H(\operatorname{div}\operatorname{div})$ conforming finite elements become increasingly popular for solving biharmonic equations. Chen and Huang  introduce  symmetric $H(\operatorname{div}\operatorname{div})$ finite elements on triangular meshes \cite{chen2020finite2d} and tetrahedral meshes \cite{ChenHuang20223d}. Subsequently, a family of symmetric $H(\operatorname{div}\operatorname{div})\cap H(\operatorname{div})$ conforming finite elements is proposed in \cite{hu2021family}, incorporating additional constraints 
of $ H(\operatorname{div})$ conforming finite elements in \cite{hu2014family,HuZhang2015} to ensure 
$H(\operatorname{div}\operatorname{div})$-conformity. The construction  is uniform in both two and three dimensions. Based on this idea, extensions to any dimension can be found in \cite{chen2022finite}, which use the $H(\operatorname{div})$ conforming finite elements in higher dimensions \cite{hu2015finite}. Recently,  new $H(\operatorname{div}\operatorname{div})$ conforming finite elements \cite{fuhrer2024mixed} have been proposed in two dimensions which relax the continuity at vertices. Nevertheless, there exists one global constraint at each vertex. An a posteriori error estimate of their elements is also provided under clamped boundary conditions. The extension of such elements for higher dimensions is available in \cite{chen2025new}.

The first objective of this paper is to provide a residual-based a posteriori error estimator for the triangular $H(\operatorname{div}\operatorname{div})$ element proposed in \cite{hu2021family}. The analysis is carried out for the Kirchhoff-Love plate bending problem with mixed boundary conditions.
By appropriately combining degrees of freedom, additional constraints are incorporated into the symmetric bending moment space, enabling the exact enforcement of the mixed boundary conditions. The reliability and efficiency of the estimator are established.
The key ingredient is to characterize the $H^1$ vector-valued space within the divdiv complex for this case. 
In particular, the boundary conditions associated with functions in this space are nontrivial. 
More precisely, it is shown that the divdiv complexes at both continuous and discrete levels are exact if and only if: for each connected component of the free boundary, the trace of the function lies in the Raviart–Thomas (RT) space; and for each edge of the simply supported boundary, the normal trace of the function is constant.

Another objective of this paper is to establish optimality of the adaptive $H(\operatorname{div}\operatorname{div})$ mixed finite element method. The main challenge lies in the non-nestedness of the symmetric $H(\operatorname{div}\operatorname{div})$ element space in \cite{hu2021family}, resulting from the extra $C^0$ vertex continuity of functions in it.
Inspired by the partial relaxation strategy in \cite{hu2021partial,carstensen2021hierarchical}, an extended $H(\operatorname{div}\operatorname{div})$ element space is introduced and shown to be nested. Based on this, the optimal convergence of the adaptive algorithm is established via the unified analysis from \cite{hu2018unified}. Numerical results demonstrate the effectiveness of the a posteriori estimator. 
It is shown that the adaptive $H(\operatorname{div}\operatorname{div})$ mixed method achieves the optimal convergence rate under adaptive mesh refinements. Although the exact solution on nonconvex domains suffers from $\operatorname{div}{\bm \sigma}\notin L^2(\Omega;\mathbb{R}^2)$, the symmetric $H(\operatorname{div}\operatorname{div})\cap H(\operatorname{div})$ conforming finite elements in \cite{hu2021family} do converge. Notably, the basis functions of the $H(\operatorname{div}\operatorname{div})$ element \cite{hu2021family} are explicitly defined and provided in the appendix, which is essential for efficient numerical implementation. 
% 最优收敛性

 %Alternatively, more practical methods have been developed, including nonconforming finite elements\cite{morley1967triangular,wang2013minimal,de1974variational,gao2011note,wang2012high}, discontinuous Galerkin methods\cite{babuvska1973nonconforming,georgoulis2009discontinuous,brenner2005c}, and mixed finite elements\cite{ciarlet1974mixed,cheng2000some,monk1987mixed,zhang2018regular,scapolla1980mixed,hellan1967analysis,herrmann1967finite,johnson1973convergence}. 

The outline of this paper is as follows. Section 2 introduces the notation and 
presents the Kirchhoff-Love plate bending problem with mixed boundary conditions.
Section 3 presents the mixed formulation, its discretization, and establishes exact divdiv complexes with boundary conditions at both the continuous and discrete levels over a contractible domain. Section 4, building on these complexes, develops a residual-based a posteriori error estimator and proves its reliability and efficiency, and further addresses a posteriori error estimates for the postprocessed deflection. Section 5 is devoted to the optimal convergence analysis of the adaptive algorithm. Section 6 reports numerical examples. The appendix provides a detailed proof of the continuous-level inf-sup condition and presents the basis functions of the $H(\operatorname{div}\operatorname{div},\mathbb S)$ conforming space from \cite{hu2021family}.

\section{Preliminary}
This section introduces some notations and presents the Kirchhoff-Love bending model problem with mixed boundary conditions.
\subsection{Notation}
Given a bounded, simply connected  Lipschitz polygonal domain $\Omega\subset\mathbb{R}^{2}$ with boundary $\partial\Omega$, let $\mathbf n=(n_1,n_2)^{\mathrm T}$ denote the unit outward normal and $\mathbf t=(-n_2,n_1)^{\mathrm T}$ is a unit tangential vector. Assume that the clamped boundary $\Gamma_C$ is compact and of positive measure, while the part of simply-supported boundary and free boundary  $\Gamma_S\cup\Gamma_F$ is the relative open complement $\Gamma_S\cup\Gamma_F=\partial \Omega\backslash \Gamma_C$. 

Assume that $G\subset \mathbb{R}^{2}$ is a bounded and topologically trivial domain. The $L^{2}$ scalar product over $G$ is denoted by $(\cdot,\cdot)_{G}$; in particular, if $G$ is an edge, it is written as $\langle\cdot,\cdot\rangle_{G}$. Let $\|\cdot\|_{0,G}$ represent the $L^{2}$ norm over a set $G$, and $\|\cdot\|_{0}$ abbreviates $\|\cdot\|_{0,\Omega}$. For positive integer $m$, let $H^{m}(G;\mathbb X)$ represent the Hilbert space consisting of functions within domain $G$, taking values in space $\mathbb X$, and with all derivatives of order at most $m$ square integrable. The associated norm and semi-norm are denoted as $\|\cdot\|_{m,G}$ and $|\cdot|_{m,G}$, respectively. The space $H_{0}^{m}(G;\mathbb X)$ denotes the closure in $H^{m}(G;\mathbb X)$ of the set of infinitely differentiable functions with compact supports in $\Omega$. 
Similarly, let $C^{m}(G;\mathbb X)$ denote the space of $m$-times continuously differentiable functions, taking values in $\mathbb X$. The range space $\mathbb X$ could be $\mathbb R$, $\mathbb{R}^{2}$,  $\mathbb M$, or $\mathbb{S}$, and $\mathbb X$ is omitted when $\mathbb X$ is $\mathbb R$, where  $\mathbb{M}$ denotes the space of\ $2\times 2$\ real matrices,\ and  $\mathbb{S}$ is the subspace of symmetric matrices.

Let $P_{k}(G;\mathbb X)$ denote the space of polynomials of degree no more than $k$ on $G$,\ taking values in the space $\mathbb X$; or simply $P_k$  when there is no possible confusion.

For $m = 1,2$, define 
\begin{equation*}
    H^m(\Gamma_X):=\{v\in H^m(E)\, \, \text{for any straight edge}\, \,  E\subset \Gamma_X\}.
\end{equation*}
Throughout the paper, $\Gamma_X$ could be any and all of boundaries $\Gamma_C$, $\Gamma_S$, and $\Gamma_F$.

Suppose $\mathcal{T}_h$ is a shape regular subdivision of $\Omega$ consisting of triangles. Let $\mathcal{E}_{h}$ and $\mathcal{V}_{h}$ be the set of all edges and vertices of $\Omega$ regarding to $\mathcal{T}_{h}$, respectively. Let $\mathcal  V_h(\Omega)$ denote the set of interior vertices of $\mathcal T_h$, and $\mathcal V_h(\Gamma_X)$ represent the set of vertices on $\Gamma_X$. Let $\mathcal  E_h(\Omega)$ denote the set of interior edges of $\mathcal T_h$, and $\mathcal E_h(\Gamma_X)$ represent the set of edges on $\Gamma_X$. Given $K\in\mathcal{T}_{h}$, let $\mathcal{E}(K)$ denote the set of all edges of $K$, and $h_{e}$ stands for the diameter of edge $e\in \mathcal{E}_{h}$. 
Let $h_{K}$ be the diameter of $K\in \mathcal{T}_{h}$, and the mesh size of $\mathcal{T}_{h}$ is denoted by $h:=\max\limits_{K\in\mathcal{T}_{h}}h_{K}$.

Besides, the jump of $u$ across an interior edge $e$ shared by two neighboring elements $K_+$ and $K_-$ is defined by
\begin{equation*}
[\![u]\!]_e:=(u|_{K_+})|_{e}-(u|_{K_-})|_{e}.    
\end{equation*}
When it comes to any boundary edge $e\subset\partial \Omega$, the jump $[\![\cdot]\!]_e$ reduces to the trace. Given an edge $e\in \mathcal E_h$, let $\mathbf{n}_e=(n_1,n_2)^{\mathrm T}$ and $\mathbf{t}_e=(-n_2,n_1)^{\mathrm T}$ be a unit normal and unit tangential vector of $e$, respectively. In the case of a boundary edge $e\subset \partial\Omega$, $\mathbf n_e$ coincides with the outward normal vector $\mathbf n$.

Throughout the paper, an inequality $\alpha\lesssim \beta$ replaces $\alpha\leq C \beta$ with some multiplicative mesh-size independent constant $C>0$, which depends on $\Omega$ and boundary conditions only. 
\subsection{Differential operators}
The derivatives $\partial/\partial {x}$ and $\partial/\partial {y}$ are abbreviated as 
$\partial_{x}$ and $\partial_y$, respectively. For a vector field $\bm \phi=(\phi_1,\phi_2)^{{\mathrm T}}$, the gradient and $\operatorname{curl}$ operators apply by row to produce matrix-valued functions, namely, 
\[\nabla \bm \phi=\begin{pmatrix}
    \partial_x \phi_1 &\partial_y\phi_1\\
    \partial_x \phi_2 &\partial_y \phi_2
\end{pmatrix},\quad \operatorname{curl}\bm \phi =\begin{pmatrix}
    \partial_y \phi_1 &-\partial_x \phi_1\\
    \partial_y \phi_2 &-\partial_x \phi_2
\end{pmatrix}.\]

Given a matrix-valued function $\bm \sigma$, define its symmetric part as $\operatorname{sym}\bm \sigma: = \frac{1}{2}(\bm \sigma + {\bm \sigma}^{{\mathrm T}})$. Let the $(i,j)$-th entry of $\bm \sigma$ be denoted as $\sigma_{ij}$. 

For matrix-valued functions, the operators $\operatorname{rot}$ and $\operatorname{div}$ apply by row to produce a vector field: 
\begin{equation*}
    \operatorname{rot}\bm \sigma = \begin{pmatrix}
        \partial_x \sigma_{12}-\partial_y \sigma_{11}\\
        \partial_x \sigma_{22}-\partial_y \sigma_{21}
    \end{pmatrix}, \quad \operatorname{div}\bm \sigma = \begin{pmatrix}
        \partial_x \sigma_{11}+\partial_y \sigma_{12}\\
        \partial_x \sigma_{21}+\partial_y \sigma_{22}
    \end{pmatrix}.
\end{equation*}
The Sobolev space $H(\operatorname{div}\operatorname{div},\Omega;\mathbb S)$ reads 
\begin{equation*}
    H(\operatorname{div}\operatorname{div},\Omega;\mathbb S):= \{\bm \tau\in L^2(\Omega;\mathbb S): \operatorname{div}\operatorname{div}\bm \tau\in L^2(\Omega)\},
\end{equation*}
equipped with the squared norm $\|\bm \tau\|_{H(\operatorname{div}\operatorname{div})}^2: = \|\bm \tau\|_0^2+\|\operatorname{div}\operatorname{div}\bm \tau\|_0^2$.
Define
\begin{equation*}
	H^1(\operatorname{div},\Omega;\mathbb S):=\{\bm{\tau}\in H^1(\Omega;\mathbb S): \operatorname{div} \bm{\tau}\in H^1(\Omega;\mathbb R^2)\}.
\end{equation*}
%%%%%%%%%%%%%%%%%%%

Given any sufficiently smooth vector-valued function $\bm \phi$,
\[\operatorname{sym}\operatorname{curl}\bm \phi = \begin{pmatrix}
    \partial_y \phi_1 &
    \frac{\partial_y \phi_2-\partial_x \phi_1}{2}\\
    \frac{\partial_y \phi_2-\partial_x \phi_1}{2}&-\partial_x \phi_2
\end{pmatrix},\]
which is a symmetric matrix-valued function with the following identities hold. The first one can be found in \cite[Lemma 10]{hu2022conforming}; For completeness, a proof is provided below. 
\begin{lemma}
    For $\bm \phi\in C^2(\Omega;\mathbb R^2)$, there holds 
    \begin{align}
        &\mathbf n^{\mathrm T} (\operatorname{sym}\operatorname{curl}\bm \phi)\mathbf n = \partial_t(\bm \phi\cdot \mathbf n),\, \, \mathbf t^{\mathrm T}( \operatorname{sym}\operatorname{curl}\bm \phi)\mathbf n = \frac{1}{2}\left(\partial_t(\bm \phi\cdot\mathbf t)-\partial_n(\bm \phi \cdot \mathbf n)\right),\label{symcurlv:2}\\
        &\mathbf n^{\mathrm T}\operatorname{div}(\operatorname{sym}\operatorname{curl}\bm \phi)=\frac{1}{2}\partial_t(\operatorname{div}\bm \phi).\label{symcurlv:3}
    \end{align}
\end{lemma}
\begin{proof}
For any scalar function $q\in C^1(\Omega)$, there holds $(\operatorname{curl} q)\cdot \mathbf n = \partial_t q$.
This and $\mathbf n^{\mathrm T}(\operatorname{curl}\bm \phi)=\operatorname{curl}(\bm \phi\cdot\mathbf n)$ shows that $ \mathbf n^{\mathrm T} (\operatorname{curl}\bm \phi)\mathbf n =\partial_t(\bm \phi\cdot \mathbf n)$. This proves the first identity of \eqref{symcurlv:2}. Besides, $\mathbf t^{\mathrm T}\operatorname{curl}\bm \phi= \operatorname{curl}(\bm \phi\cdot \mathbf t)$ gives rise to $ \mathbf t^{\mathrm T} (\operatorname{curl}\bm \phi)\mathbf n =\partial_t(\bm \phi\cdot \mathbf t)$. The relation $(\operatorname{curl} q)\cdot\mathbf t = -\partial_n q$ combined with $\left(\operatorname{curl}\bm \phi\right)^{\mathrm T}\mathbf n = \operatorname{curl}(\bm \phi\cdot \mathbf n)$ leads to $\mathbf t^{\mathrm T} \left(\operatorname{curl}\bm \phi\right)^{\mathrm T}\mathbf n = -\partial_n(\bm \phi\cdot \mathbf n)$. Thus the second identity of \eqref{symcurlv:2} follows. Furthermore, owing to $\operatorname{div}(\operatorname{curl}\bm \phi)=0$ and $\operatorname{div}\left((\operatorname{curl}\bm \phi)^{\mathrm T}\right)=\operatorname{curl}(\operatorname{div}\bm \phi)$, there holds
\begin{equation*}
\mathbf n^{\mathrm T}\operatorname{div}(\operatorname{sym}\operatorname{curl}\bm \phi) = \frac{1}{2}\mathbf n^{\mathrm T}\operatorname{div}\left((\operatorname{curl}\bm \phi)^{\mathrm T}\right) = \frac{1}{2}\mathbf n^{\mathrm T} \operatorname{curl}(\operatorname{div}\bm \phi)= \frac{1}{2}\partial_t\left(\operatorname{div}\bm \phi\right).
\end{equation*}
\end{proof}

\subsection{The Kirchhoff–Love
plate bending model problem}
The Kirchhoff–Love bending
plate is clamped on the part $\Gamma_C\subset \partial\Omega$, simply supported on the part $\Gamma_S\subset \partial\Omega$, and free on the open part $\Gamma_F\subset \partial\Omega$. 
Specifically, given the transversal load $f\in L^2(\Omega)$, the equilibrium equation of the deflection $w$ reads
\begin{equation}\label{equ:0}
    \operatorname{div}\operatorname{div}\mathbb C\nabla^2 w = f \quad\text{in}~~\Omega,
\end{equation}
with the following mixed boundary conditions for the given data $w_{b}\in H^{2}(\Omega)\cap H^2(\Gamma_C\cup\Gamma_S)$, $g_{b}\in H^{1}(\Gamma_C)$, $m_b\in  H^1(\Gamma_C\cup \Gamma_S)$, $h_b\in L^2(\Gamma_F)$, and the point value $p_{\mathbf x}$ for all $\mathbf x\in \mathcal V_F$:
\begin{equation}\label{equ:1}
    \begin{aligned}
        &w = w_{b}, \,\,  \,\, \partial_n w = g_{b} &&\quad\text{on}~~\Gamma_C,\\
        &w = w_{b}, \,\,  \,\, \mathbf n^{\mathrm T}\bm \sigma\mathbf n = m_{b} &&\quad\text{on}~~\Gamma_S,\\
        &\mathbf n^{\mathrm T}\bm \sigma\mathbf n=m_b,\, \, \,\,  \partial_t(
       \mathbf t^{\mathrm T}\bm \sigma\mathbf n)+\mathbf n^{\mathrm T}\operatorname{div}\bm \sigma= h_b &&\quad\text{on}~~\Gamma_F,\\
        &[\![\mathbf t^{\mathrm T} \bm \sigma \mathbf n]\!]_{\mathbf x}=p_{\mathbf x} &&\quad \text{for} \, \,\mathbf x \in \mathcal V_F.
    \end{aligned}
\end{equation}
Here $\mathcal V_F$ denotes the set of all interior corner points on $\Gamma_F$, and for any $\mathbf x \in \mathcal V_F$,  two edges $e_{\pm}$ form the boundary angle at $\mathbf x$. Throughout this and the subsequent discussion, the jump at corner points $\mathbf x$ is defined as follows:
\[[\![\mathbf t^{\mathrm T} \bm \sigma \mathbf n]\!]_{\mathbf x}:=(\mathbf t^{\mathrm T} \bm \sigma(\mathbf x) \mathbf n)|_{e_+}-(\mathbf t^{\mathrm T}\bm \sigma(\mathbf x)\mathbf n)|_{e_-}.\]
The given data satisfies the compatibility condition $[\![g_b\mathbf n+\partial_t w_b\mathbf t]\!]_{\mathbf x}=0$ for $\mathbf x\in \mathcal V_C$ with $\mathcal V_C$ being the set of all interior corner points on $\Gamma_C$.
Besides, $\mathbb C: L^2(\Omega;\mathbb S)\rightarrow L^2(\Omega;\mathbb S)$ is a symmetric positive definite isomorphism, given by 
\begin{equation}\label{MT}
   \mathbb C \nabla^2 w= \mathsf D((1-\nu)\varepsilon(\nabla w)+\nu\operatorname{div}(\nabla w)\bm I),
\end{equation}
where $\varepsilon:=\frac{1}{2}(\nabla+\nabla^{\mathrm T})$, $\bm I$ denotes the identity, and $\mathsf D = \frac{\mathsf E t^3}{12(1-\nu^2)}$
represents the bending rigidity with the Young modulus $\mathsf E$, the plate thickness $t$, and the Poisson ratio $\nu$. The bending moment tensor $\bm \sigma$ is given by $\bm \sigma = \mathbb C \nabla^2 w$.

Note that $\mathbb C^{-1}$ is a symmetric positive definite isomorphism. For $\bm \tau\in L^2(\Omega;\mathbb S)$, the weighted $L^2$-inner product norm is defined by
 \begin{equation*}
    \|\bm \tau\|^2_{\mathbb C^{-1}}:=(\mathbb C^{-1}\bm \tau, \bm \tau).
\end{equation*}
The positive definiteness of $\mathbb C^{-1}$ and $\mathbb C$ gives rise to
  \begin{equation}\label{C-norm:equ}
    \|\mathbb C^{-1}\bm \tau\|^2_0\lesssim \|\bm \tau\|_{\mathbb C^{-1}}^2\lesssim\|\bm \tau\|_0^2\lesssim \|\bm \tau\|_{\mathbb C^{-1}}^2 \lesssim\|\mathbb C^{-1}\bm \tau\|^2_0.
 \end{equation}

Define 
\begin{equation}
    \Lambda = \{v\in H^2(\Omega): v|_{\Gamma_C\cup \Gamma_S} = 0, \, \partial_n v|_{\Gamma_C}=0\}.
\end{equation}
The primal formulation of \eqref{equ:0}--\eqref{equ:1} reads: Find $w\in H^2(\Omega)$ satisfying $w|_{\Gamma_C\cup \Gamma_S}=w_b$ and $\partial_n w|_{\Gamma_S}=g_b$ such that
\begin{equation}
    \label{primal:formulation}
    \begin{aligned}
         (\mathbb C\nabla^2 w, \nabla^2 v) = (f,v)
         +\operatorname{R}_b(v)\quad \text{for all}~~v\in \Lambda.
    \end{aligned}
\end{equation}
Here 
\begin{equation}\label{rb:def}
    \operatorname{R}_{b}(v): = \langle m_b,\partial_n v \rangle_{\Gamma_S\cup \Gamma_F}-\langle h_b, v\rangle_{\Gamma_F}+\sum_{\mathbf x\in\mathcal V_F} p_{\mathbf x}v (\mathbf x).
\end{equation}

\section{Mixed finite element methods and divdiv complexes with boundary conditions}
This section presents the mixed formulation and the discretization for the Kirchhoff-Love plate bending problem with mixed boundary conditions. Besides, this section establishes the divdiv complexes with boundary conditions at both continuous and discrete levels, which play a crucial role in the subsequent a posteriori error estimation and optimality analysis.

\subsection{The mixed finite element method}
Introducing the bending moment $\bm \sigma$, one can reformulate the equilibrium equation of system \eqref{equ:0} into
\begin{equation}\label{equ:0:sigma}
    \begin{aligned}
        &\mathbb C^{-1}\bm \sigma =  \nabla^2 w, \, \, \operatorname{div}\operatorname{div}\bm \sigma = f &&\quad\text{in}~~\Omega.\\
    \end{aligned}
\end{equation}
To present the mixed formulation, the continuous spaces with homogeneous and non-homogeneous boundary conditions are defined as follows:
 \begin{align}
    &\Sigma_0: = \{\bm \tau \in H(\operatorname{div}\operatorname{div},\Omega;\mathbb S): (\bm \tau,\nabla^2 v)-(\operatorname{div}\operatorname{div}\bm \tau, v)=0\, \,  \quad\text{for all}~~v\in \Lambda\},\label{sigma:0:def}\\
    &\Sigma_b: = \{\bm \tau \in H(\operatorname{div}\operatorname{div},\Omega;\mathbb S):(\bm \tau,\nabla^2 v)-(\operatorname{div}\operatorname{div}\bm \tau, v)= \operatorname{R}_b(v) \quad\text{for all}~~v\in \Lambda \}. \label{sigma:b:def}
\end{align}

The mixed formulation for the fourth-order problem \eqref{equ:0:sigma} with the mixed boundary conditions \eqref{equ:1} is to find $(\bm \sigma, w)\in \Sigma_b\times L^2(\Omega)$ such that 
\begin{equation}\label{mixed:bd:con}
 \begin{aligned}
    (\mathbb C^{-1}\bm \sigma, \bm \tau) -(\operatorname{div}\operatorname{div}\bm \tau, w) &= \operatorname{tr}_{b}(u)(\bm \tau)\quad&&\text{for all}~~\bm \tau \in \Sigma_0,\\
    (\operatorname{div}\operatorname{div}\bm \sigma, v) &= (f,v)\quad&& \text{for all}~~v\in L^2(\Omega),
\end{aligned}   
\end{equation}
where 
\begin{equation}\label{trace:gammaD}
  \operatorname{tr}_{b}(u)(\bm \tau):=(\bm \tau, \nabla^2 u) -(\operatorname{div}\operatorname{div}\bm \tau, u)
\end{equation}
with some $u\in H^2(\Omega)$ and $u|_{\Gamma_C\cup \Gamma_S} = w_{b}$, $\partial_n u|_{\Gamma_C}=g_{b}$.
In reality, for  $\bm \tau \in C^1(\overline{\Omega};\mathbb S)\cap \Sigma_0$, with the compatibility condition $[\![g_b\mathbf n+\partial_t w_b\mathbf t]\!]_{\mathbf x}=0$ for all $\mathbf x\in \mathcal V_C$, one can reformulate the expression \eqref{trace:gammaD} into 
\begin{equation}\label{trace:gammaD:C1}
  \operatorname{tr}_{b}(u)(\bm \tau):=\langle\mathbf n^{\mathrm T}\bm \tau\mathbf n, g_b\rangle_{\Gamma_C}+\langle\mathbf t^{\mathrm T}\bm \tau\mathbf n, \partial_t w_b\rangle_{\Gamma_C\cup \Gamma_S} - \langle \mathbf n^{\mathrm T} \operatorname{div}\bm \tau, w_{b}\rangle_{\Gamma_C\cup \Gamma_S}.  
\end{equation}

\begin{theorem}[The inf-sup condition]\label{inf:sup:con:bd}There  holds
   \begin{equation*}
       \sup_{\bm \tau \in \Sigma_0\atop \bm \tau \neq 0}\frac{(\operatorname{div}\operatorname{div}\bm \tau,v)}{\|\bm \tau\|_{H(\operatorname{div}\operatorname{div})}}\gtrsim \|v\|_0\quad \text{for all}~~v\in L^2(\Omega).
   \end{equation*}
\end{theorem}
\noindent The proof of Theorem \ref{inf:sup:con:bd} is provided in Appendix \ref{sec:appendix}.

%The equivalence of the mixed formulation \eqref{mixed:bd:con} and the primal formulation \eqref{primal:formulation} is demonstrated below. In fact, if $w\in H^2(\Omega)$ satisfying $w|_{\Gamma_C\cup \Gamma_S}=w_b$ and $\partial_n w|_{\Gamma_S}=g_b$ solves the primal formulation \eqref{primal:formulation}, then $\bm \sigma=\mathbb C\nabla^2 w\in \Sigma_b$, and $(\bm \sigma, w)$ solves the mixed formulation \eqref{mixed:bd:con}. And, vice versa, if $(\bm \sigma, w)\in \Sigma_b\times L^2(\Omega)$ solves \eqref{mixed:bd:con}, then $w\in H^2(\Omega)$ satisfying $w|_{\Gamma_C\cup \Gamma_S}=w_b$ and $\partial_n w|_{\Gamma_S}=g_b$ solves \eqref{primal:formulation}.

\begin{co}
 Assuming $f\in L^2(\Omega)$, the mixed formulation \eqref{mixed:bd:con} is fully equivalent to the primal formulation \eqref{primal:formulation}. 
\end{co}
    \begin{proof}
       Both problems \eqref{primal:formulation} and \eqref{mixed:bd:con} are uniquely solvable. Therefore, it sufﬁces to show that $(\bm \sigma, w)$ with $\bm \sigma=\mathbb C\nabla^2 w\in \Sigma_b$ solves \eqref{mixed:bd:con}, if $w\in H^2(\Omega)$ satisfying $w|_{\Gamma_C\cup \Gamma_S}=w_b$ and $\partial_n w|_{\Gamma_S}=g_b$ solves \eqref{primal:formulation}. Assume that $w\in H^2(\Omega)$ satisfies $w|_{\Gamma_C\cup \Gamma_S}=w_b$ and $\partial_n w|_{\Gamma_S}=g_b$ and solves \eqref{primal:formulation}. Then $\bm \sigma = \mathbb C\nabla^2 w\in L^2(\Omega;\mathbb S)$ and
       \begin{equation}\label{w:form:1}
           (\bm \sigma, \nabla^2 v)= (f,v)+\operatorname{R}_b(v)\quad \text{for all}\, \, v\in \Lambda.
       \end{equation}
       Since $C_0^\infty(\Omega)\subset \Lambda$, an integration by parts leads to
       \begin{equation*}
           (\operatorname{div}\operatorname{div}\bm \sigma, v) = (f,v)\quad \text{for all}\, \, v\in C_0^\infty(\Omega).
       \end{equation*}
       This shows that $\operatorname{div}\operatorname{div}\bm \sigma = f\in L^2(\Omega)$. Therefore, there holds
       \begin{equation}\label{2nd:row:1}
           (\operatorname{div}\operatorname{div}\bm \sigma, v) = (f, v)\quad \text{for all}\, \, v\in L^2(\Omega),
       \end{equation}
      and the second row of \eqref{mixed:bd:con} follows. In addition, for any $v\in \Lambda$, a subtraction of \eqref{2nd:row:1} from \eqref{w:form:1} results in 
      \begin{equation*}
          (\bm \sigma,\nabla^2 v)-(\operatorname{div}\operatorname{div}\bm \sigma,v) =  \operatorname{R}_b(v).
      \end{equation*}
      This proves $\bm \sigma\in \Sigma_b$. Furthermore, given any function $\bm \tau\in \Sigma_0$, thanks to $(\bm \tau, \nabla^2 w)-(\operatorname{div}\operatorname{div}\bm \tau, w)=\operatorname{tr}_b(w)(\bm \tau)$, a choice of $u=w\in H^2(\Omega)$ with $w|_{\Gamma_C\cup \Gamma_S}= w_b$ and $\partial_n w|_{\Gamma_C} = g_b$ leads to 
      \begin{equation}
          (\mathbb C^{-1}\bm \sigma, \bm \tau)-(\operatorname{div}\operatorname{div}\bm \tau, w) = \operatorname{tr}_b(u)(\bm \tau)\quad \text{for all}\, \, \bm \tau\in \Sigma_0.
      \end{equation}
       This concludes the first row of \eqref{mixed:bd:con}.
    \end{proof}

This paper adopts the mixed finite element method from \cite{hu2021family} to discretize the mixed formulation \eqref{mixed:bd:con} subject to mixed boundary conditions. For $k\geq 3$, let $\Sigma_{h}$ be the $H(\operatorname{div}\operatorname{div}, \Omega; \mathbb{S})$ conforming finite element space proposed in \cite{hu2021family}:
\begin{equation*}
\begin{aligned}
\Sigma_{h}:=\{\bta\in H(\operatorname{div}\operatorname{div}, \Omega; \mathbb{S}): \bta|_{K}\in P_{k}(K;\ms)\, \text{for all}\, K\in \mathcal{T}_{h},\\
\bm \tau (\mathbf x)\, \, \text{is continuous}\, \, \text{at all}\,\mathbf x \in \mathcal V_h,\\
[\![\bta\mathbf{n}_e]\!]_{e}=\mathbf{0}\, \text{and}\, [\![\mathbf{n}_e^{{\mathrm T}}\operatorname{div}\bta]\!]_{e}=0\,\text{for all}\,e\in \mathcal{E}_{h}(\Omega) \}.
\end{aligned}
\end{equation*}

Define a subspace of $\Sigma_0$ in \eqref{sigma:0:def} by
\begin{equation}
\begin{aligned}
     \Sigma_{h,0}: = \{\bm \tau\in \Sigma_h: \, \mathbf n^{\mathrm T}\bm \tau\mathbf n|_{\Gamma_S\cup \Gamma_F} = 0, \mathbf n^{\mathrm T}\operatorname{div}\bm \tau|_{\Gamma_F} = -\partial_t(\mathbf t^{\mathrm T}\bm \tau\mathbf n)|_{\Gamma_F},\\
     [\![\mathbf t^{\mathrm T} \bm \tau \mathbf n]\!]_{\mathbf x}=0 \quad \text{for all} \, \,\mathbf x \in \mathcal V_F\}. 
\end{aligned}  
\end{equation}
 Given any edge $e\in \mathcal E_h(\Gamma_S\cup \Gamma_F)$ with two endpoints $\mathbf x_0$ and $\mathbf x_1$, define $\pi_{e}: C^0(e) \rightarrow P_k(e)$ as follows:
\begin{equation}\label{mb:inter}
    \begin{aligned}
    \pi_{e} m_b(\mathbf x_i) &= m_b(\mathbf x_i)~~&&\text{for}~~i = 0,1,\\
    \langle\pi_{e} m_b,\bm q\rangle_e &= \langle m_b, q\rangle_e~~&&\text{for all}~~q\in P_{k-2}(e).
\end{aligned}
\end{equation}
Let $\pi_{h}m_b|_e = \pi_e m_b$ for all  $e\in \mathcal E_h(\Gamma_S\cup \Gamma_F)$. In addition, let $\mathcal P^{k-1}_e$ represent the $L^2$ projection operator onto $P_{k-1}(e)$, and $\mathcal P^{k-1}_h h_b|_e  =\mathcal P^{k-1}_e h_b$ for all  $e\in \mathcal E_h(\Gamma_F)$. The discrete space with non-homogeneous boundary conditions is defined by: 
\begin{equation}
\begin{aligned}
    \Sigma_{h,b}: = &\{\bm \tau\in \Sigma_h: \, \mathbf n^{\mathrm T}\bm \tau\mathbf n|_{\Gamma_S\cup \Gamma_F} = \pi_h m_b, \\
     &\mathbf n^{\mathrm T}\operatorname{div}\bm \tau|_{\Gamma_F} = -\partial_t(\mathbf t^{\mathrm T}\bm \tau\mathbf n)|_{\Gamma_F}+\mathcal P^{k-1}_h h_b,~~
    [\![\mathbf t^{\mathrm T} \bm \tau\mathbf n]\!]_{\mathbf x}=p_{\mathbf x}~~\text{for all} \, \,\mathbf x \in \mathcal V_F\}.
    \end{aligned}  
\end{equation}
\begin{remark}\label{b:d:sigmahb}
Introduce
\begin{equation}\label{rhb:def}
    R_{h,b}(v):= \langle\pi_h m_b, \partial_n v\rangle_{\Gamma_S\cup \Gamma_F}-\langle\mathcal P_h^{k-1}h_b, v\rangle_{\Gamma_F}+\sum_{\mathbf x\in \mathcal V_F}p_{\mathbf x}v(\mathbf x)\quad \text{for all}~~v\in \Lambda.
\end{equation}
Indeed, for any $\bm \tau \in \Sigma_{h,b}$, the boundary conditions of $\bm \tau$ are equivalent to 
\begin{align*}
    (\bm \tau, \nabla^2 v)-(\operatorname{div}\operatorname{div}\bm \tau,v) = R_{h,b}(v)\, \,  \quad\text{for all}~~v\in \Lambda.
\end{align*}
\end{remark}
Define
\begin{equation*}
U_{h}:=\{v\in L^{2}(\Omega):\, v|_{K}\in P_{k-2}(K)\, \text{for all}\, K\in \mathcal{T}_{h}\}.
\end{equation*}

The mixed finite element method is to find 
$(\bm \sigma_h, w_h)\in \Sigma_{h,b}\times U_h$ such that 
\begin{equation}\label{mixed:bd:dis}
 \begin{aligned}
    (\mathbb C^{-1}\bm \sigma_h, \bm \tau_h) -(\operatorname{div}\operatorname{div}\bm \tau_h, w_h) &= \operatorname{tr}_{b}(u)(\bm \tau_h)\quad&&\text{for all}~~\bm \tau_h \in \Sigma_{h,0},\\
    (\operatorname{div}\operatorname{div}\bm \sigma_h, v_h) &= (f,v_h)\quad &&\text{for all}~~v_h\in U_h.
\end{aligned}   
\end{equation}

For simplicity, let $\mathcal Q_h$ denote the $L^2$ projection onto $U_h$. More generally, let $\mathcal Q_h^{k}$ be the $L^2$ projection onto the space of piecewise $P_{k}$ polynomials. Recall that $\mathcal P_h^k$ denotes the $L^2$ projection onto the space of piecewise $P_k$ polynomials defined on the edges of the mesh. These notations will be frequently used throughout this paper.

\subsection{Continuous divdiv complex with boundary conditions}
This subsection establishes the continuous divdiv complex with boundary conditions. The key aspect is to identify appropriate boundary conditions for the $H^1$ vector-valued space within the divdiv complex. It is shown that a specific and unique boundary condition ensures the exactness of both the continuous and discrete complexes.

For the continuous divdiv complex 
\begin{equation}\label{divdivComp2d}
\begin{aligned}
RT\stackrel{\subset}{\rightarrow}H^1(\Omega;\mathbb{R}^2)  \stackrel{\operatorname{sym}\operatorname{curl}}{\longrightarrow}H(\operatorname{div}\operatorname{div},\Omega;\mathbb{S})\stackrel{\operatorname{div}\operatorname{div}}{\longrightarrow}L^2(\Omega;\mathbb{R})\stackrel{}{\rightarrow}0,
\end{aligned}
\end{equation}
a discrete sub-complex consisting of $\Sigma_h$ and $U_h$ has been presented in \cite{hu2021family}, which reads
\begin{equation}\label{eq:DisComp}
R T \stackrel{\subset}{\longrightarrow} V_{h} \stackrel{\operatorname{sym}\operatorname{curl}}{\longrightarrow} \Sigma_{h} \stackrel{\operatorname{div}\operatorname{div}}{\longrightarrow}U_{h} \rightarrow 0,
\end{equation}
where $$RT:=\{a\mathbf{x}+\bm{b}:\, a\in\mathbb{R}, \bm{b}\in \mathbb{R}^{2}\}$$ with ${\rm{dim}}\, RT=3$ is the shape function space of the lowest-order Raviart-Thomas element \cite{bbf}, and 
\begin{equation}\label{vh:def}
\begin{aligned}
    V_h:= \{\bm \phi \in H^1(\operatorname{div},\Omega;\mathbb R^2): &\bm \phi|_{K}\in P_{k+1}(K;\mathbb R^2),\\ &\nabla \bm \phi (\mathbf x)~~\text{is continuous for all}\, \,\mathbf x\in \mathcal V_h\}
\end{aligned}   
\end{equation}
with $H^1(\operatorname{div},\Omega;\mathbb R^2):=\{\bm \phi \in H^1(\Omega;\mathbb R^2): \operatorname{div}\bm \phi\in  H^1(\Omega)\}$. More specifically, the degrees of freedom of $V_h$ given in \cite[(2.16--2.20)]{hu2021family} are listed below for later use:
\begin{align}
    &\bm \phi(a), \nabla\bm \phi(a)\quad  &&\text{for all}\, \, a\in \mathcal V_h, \label{v:dof:1}\\
     &\langle\bm \phi, \bm q\rangle_e\quad  &&\text{for all}\, \, \bm q\in P_{k-3}(e;\mathbb R^2), e\in\mathcal E_h,\label{v:dof:2}\\
     &\langle\operatorname{div}\bm \phi, q\rangle_e\quad &&\text{for all}\, \, q\in P_{k-2}(e), e\in \mathcal E_h,\label{v:dof:3}\\
     &(\bm \phi, \bm q)_K\quad &&\text{for all}\, \, \bm q\in RT^\perp_{k-4},\label{v:dof:4}
\end{align}
where $RT^\perp_{k-4}$ is the $(k-4)$-order rotated Raviart-Thomas element space \cite{bbf}. The interior degrees of freedom \eqref{v:dof:4} are adapted from those in \cite[(2.19--2.20)]{hu2021family}.
\begin{lemma}[\cite{arnold2021complexes,hu2021family}]\label{exact:divdiv}
The complexes \eqref{divdivComp2d} and \eqref{eq:DisComp} are exact on a contractible domain. 
\end{lemma}

To introduce the divdiv complex with boundary conditions, some notations and conventions are presented. Let $E\subset \partial\Omega$ represent a straight segment of the boundary (with no corner points in its interior). For convenience, assume that each such segment $E$ is subject to at most one type of boundary condition. Let $\Gamma_{F_j}$ denote the $j$-th connected component of $\Gamma_F$ with $1\leq j\leq J$.

The $H^1$ vector-valued function space with particular boundary conditions is given by
\begin{align*}
    V_0:=\{&\bm \phi\in H^1(\Omega;\mathbb R^2): \,  \bm \phi\cdot \mathbf n|_{E}\,\,  \text{is constant for each straight segment}\, \, E\subset \Gamma_S,\\
    &\bm \phi|_{\Gamma_{F_j}}=\bm r_j|_{\Gamma_{F_j}}\, \, \text{and}\,\,   \bm r_j\in RT\, \, \text{for each connected component}\, \,\Gamma_{F_j}\subset\Gamma_F, 2\leq j\leq J,\\
    &\bm \phi|_{\Gamma_{F_1}}=\bm 0\}.
\end{align*}

\begin{lemma}\label{leftarrow}
    For any $\bm \tau\in \Sigma_{0}$ satisfying $\operatorname{div}\operatorname{div}\bm \tau=0$, there exists $\bm \phi\in V_0$ such that $\operatorname{sym}\operatorname{curl}\bm \phi = \bm \tau$ and $|\bm \phi|_1\lesssim \|\bm \tau\|_0$.
\end{lemma}
\begin{proof}
    For any $\bm \tau\in \Sigma_{0}$ satisfying $\operatorname{div}\operatorname{div}\bm \tau=0$, the exactness of the continuous divdiv complex \eqref{divdivComp2d} shows that there exists some $\bm \phi\in H^1(\Omega;\mathbb R^2)$ such that $\operatorname{sym}\operatorname{curl}\bm \phi = \bm \tau$. 
 For any $v\in C^{\infty}(\overline{\Omega})\cap \Lambda$, an integration by parts leads to
\begin{equation}\label{con:ker:divdiv}
     0 = (\bm \tau, \nabla^2 v)=(\operatorname{sym}\operatorname{curl}\bm \phi, \nabla^2 v)= (\operatorname{curl}\bm \phi, \nabla^2 v)=-\langle\bm \phi, \partial_t(\nabla v)\rangle_{\partial\Omega}.
\end{equation}
Decomposing the vector-valued function $\bm \phi$ into its tangential and normal components yields
\begin{equation}\label{n:t:decom}
    0=\langle\bm \phi, \partial_t(\nabla v)\rangle_{\partial\Omega}=\langle\bm \phi\cdot\mathbf t, \partial_{tt}v\rangle_{\partial\Omega}+\langle\bm \phi\cdot\mathbf n, \partial_{tn} v\rangle_{\partial\Omega}.
\end{equation}

For any straight edge $E\subset \Gamma_S$ and any $g\in C^\infty_0(E)$ with compact support $\operatorname{supp} g$ strictly contained in $E$, let $\Omega_1$ and $\Omega_2$ be open neighborhood such that $\operatorname{supp} g\subset \Omega_1\subsetneqq\Omega_2$ and $\Omega_2\cap (\partial\Omega\setminus E)=\emptyset$.
Extend $g$ in the normal direction of $E$ constantly to the neighborhood $\Omega_2$ and still denoted by $g$. Let $v= g\chi_E l_E$ with $l_E=0$ being the equation of $E$ and $\chi_E\in C_0^\infty(\Omega_2)$ a indicator function such that $\chi_E\equiv 1$ on $\Omega_1$. Note that $\partial_n v |_E = g$. Substituting such $v$ into \eqref{n:t:decom} results in 
\begin{equation}
    \langle\bm \phi\cdot\mathbf n, \partial_{t} g\rangle_{E}=0.
\end{equation}
The arbitrariness of $g$ shows that $\bm \phi\cdot\mathbf n|_{E}$ is constant.

Similarly, it can be shown that $\bm \phi\cdot \mathbf n|_E$ is constant for each straight edge $E\subset \Gamma_F$. Setting $v=g\chi_E(1-l_E^2)$ in \eqref{n:t:decom} leads to 
\begin{equation}
    \langle\bm \phi\cdot \mathbf t, \partial_{tt}g\rangle_E=0.
\end{equation}
The arbitrariness of $g$ shows that $\bm \phi\cdot\mathbf t|_E$ is linear. Therefore, there exists $\bm r\in RT$ such that $\bm \phi|_E = \bm r|_E$ for each straight edge $E\subset \Gamma_F$.

For any corner point $\mathbf x\in \mathcal V_F$ shared by two segments $E_1\subset \Gamma_F$ and $E_2\subset \Gamma_F$, since $\bm \phi\in H^1(\Omega;\mathbb R^2)$, i.e. $\bm \phi|_{E_1\cup E_2}\in H^{\frac{1}{2}}(E_1\cup E_2;\mathbb R^2)$, and $\bm \phi$ is linear on $E_i$ with $i = 1,2$, it holds (see e.g., \cite[pp. 49]{bbf})
\begin{equation}\label{con:bm:v}
    \bm \phi|_{E_1}(\mathbf x)=\bm \phi|_{E_2}(\mathbf x).
\end{equation}
 Let $\Omega_{\mathbf x}\subset\Omega$ represent an open neighborhood of the corner point $\mathbf x$.
For any $v\in C_0^\infty(\Omega_{\mathbf x})\cap \Lambda$ in \eqref{n:t:decom}, applying integration by parts together with the fact that $\bm \phi\cdot\mathbf n|_{E_i}$ is constant and $\bm \phi\cdot\mathbf t|_{E_i}$ is linear gives rise to
\begin{align*}
     0&=\langle\bm \phi, \partial_t(\nabla v)\rangle_{\partial\Omega}=[\![\bm \phi\cdot \nabla v]\!]_{\mathbf x}-\sum_{i=1}^2\left(\langle\partial_t(\bm \phi\cdot \mathbf t), \partial_t v\rangle_{E_i}+\langle\partial_t(\bm \phi\cdot \mathbf n), \partial_n v\rangle_{E_i}\right)\\
     &=[\![\bm \phi\cdot \nabla v]\!]_{\mathbf x}-\sum_{i=1}^2\langle\partial_t(\bm \phi\cdot \mathbf t), \partial_t v\rangle_{E_i}=[\![\bm \phi\cdot \nabla v]\!]_{\mathbf x}-[\![\partial_{t}(\bm \phi\cdot \mathbf t)]\!]_{\mathbf x} v(\mathbf x).
\end{align*}
This and \eqref{con:bm:v} yield $[\![\partial_{t}(\bm \phi\cdot \mathbf t) ]\!]_{\mathbf x}=0$ and consequent the existence of some $\bm r\in RT$ such that 
\begin{equation*}
    \bm \phi|_{E_i} = \bm r|_{E_i}\quad \text{i=1, 2}.
\end{equation*}
The same arguments show that $\bm r\in RT$ is uniquely defined on each connected component of $\Gamma_F$. Assume $\bm r=0$ on the first component $\Gamma_{F_1}$ without loss of generality. 
Therefore, $\bm \phi\in V_0$ follows.

Furthermore, the boundary conditions for $\bm \phi$ imply that $\|\operatorname{sym}\operatorname{curl}\bm \phi\|_0$ is a norm. Assume $\bm \phi= (\phi_1, \phi_2)^{\mathrm T}$ and $\bm \phi^\perp = (\phi_2, -\phi_1)^{\mathrm T}$. In light of Korn's inequality, it holds that  
\begin{equation*}
\|\operatorname{sym}\operatorname{curl}\bm \phi\|_0= \|\varepsilon(\bm \phi^\perp)\|_0 \gtrsim \|\nabla (\bm \phi^\perp)\|_0 = \|\nabla \bm \phi\|_0.
\end{equation*}
Thus $|\bm \phi|_1\lesssim \|\operatorname{sym}\operatorname{curl}\bm \phi\|_0=\|\bm \tau\|_0$, which concludes the proof. 
\end{proof}
The continuous divdiv complex with boundary conditions is presented below.
\begin{theorem}\label{con:exact:mbd}
    The divdiv complex
     \begin{equation}\label{eq:conComplex:bd}
0 \stackrel{\subset}{\longrightarrow} V_{0} \stackrel{\operatorname{sym}\operatorname{curl}}{\longrightarrow} \Sigma_{0} \stackrel{\operatorname{div}\operatorname{div}}{\longrightarrow}L^2(\Omega) \rightarrow 0
\end{equation}
    is exact on a contractible domain. 
\end{theorem}
\begin{proof}
    The exactness at $V_0$ is trivial and at $L^2(\Omega)$ results from Theorem \ref{inf:sup:con:bd}. To demonstrate the exactness at $\Sigma_0$, given Lemma \ref{leftarrow}, it suffices to prove  $\operatorname{sym}\operatorname{curl}V_0 \subset  \Sigma_0$. Given any $\bm \phi\in V_0$, for all $v\in\Lambda$, integration by parts leads to
    \begin{align*}
        (\operatorname{sym}\operatorname{curl}\bm \phi, \nabla^2 v)=\langle\operatorname{curl}\bm \phi\mathbf n , \nabla v\rangle_{\partial \Omega} = \langle\partial_t \bm \phi, \nabla v\rangle_{\partial \Omega}.
    \end{align*}
Since $v|_{\Gamma_C\cup \Gamma_S}=0$ and $\partial_n v|_{\Gamma_C}=0$, decomposing $\bm \phi$ into its tangential and normal components plus a further integration by parts gives rise to
\begin{align*}
    &(\operatorname{sym}\operatorname{curl}\bm \phi, \nabla^2 v) = \sum_{E\subset \partial\Omega}\left(\langle\partial_t(\bm \phi\cdot \mathbf n), \partial_n v\rangle_E+\langle\partial_t(\bm \phi\cdot\mathbf t), \partial_t v\rangle_E\right)\\
&=\sum_{E\subset \Gamma_S\cup \Gamma_F}\langle\partial_t(\bm \phi\cdot \mathbf n), \partial_n v\rangle_E+\sum_{\mathbf x\in\mathcal V_F}[\![\partial_{t}(\bm \phi\cdot \mathbf t)]\!]_{\mathbf x}v(\mathbf x)-\sum_{E\subset \Gamma_F}\langle\partial_{tt}(\bm \phi\cdot\mathbf t), v\rangle_E.
\end{align*}
Note that $\bm \phi\cdot\mathbf n|_E$ is constant for each $E\subset \Gamma_S\cup \Gamma_F$, $\bm \phi\cdot \mathbf t$ is a linear function on $E\subset \Gamma_F$ and  $\partial_t(\bm \phi\cdot \mathbf t)$ is continuous at each corner point $\mathbf x\in\mathcal V_F$. Thus $(\operatorname{sym}\operatorname{curl}\bm \phi, \nabla^2 v)=0$, and consequently $\operatorname{sym}\operatorname{curl}\bm \phi\in \Sigma_0$. 
\end{proof}

\subsection{Disctrete divdiv complex with boundary conditions}
The exactness of the discrete divdiv complex composed of finite element spaces with boundary conditions is demonstrated through two lemmas. Let $ V_{h,0}:=V_h\cap V_0$.

\begin{lemma}\label{rightarrow}
There holds
    $$\operatorname{sym}\operatorname{curl}V_{h,0} \subset \Sigma_{h,0}\cap \operatorname{ker}(\operatorname{div}\operatorname{div}). $$
\end{lemma}
\begin{proof}
    Given any $\bm \phi \in V_{h,0}$, the fact $\bm \phi\in V_h$ together with the exactness of the discrete $\operatorname{div}\operatorname{div}$ complex \eqref{eq:DisComp} shows $\operatorname{sym}\operatorname{curl}\bm \phi \in \Sigma_h\cap \operatorname{ker}(\operatorname{div}\operatorname{div})$. To show $\operatorname{sym}\operatorname{curl}\bm \phi \in\Sigma_{h,0}$,
it remains to prove that $\operatorname{sym}\operatorname{curl}\bm \phi$ satisfies:
    \begin{align}
        &\mathbf n^{\mathrm T}(\operatorname{sym}\operatorname{curl}\bm \phi)\mathbf n|_{\Gamma_S\cup \Gamma_F} = 0,\label{symcurlv:in:sigma1}\\
        &\mathbf n^{\mathrm T}\operatorname{div}\left( \operatorname{sym}\operatorname{curl}\bm \phi\right)|_{\Gamma_F} =-\partial_t (\mathbf t^{\mathrm T}\left( \operatorname{sym}\operatorname{curl}\bm \phi\right) \mathbf n)|_{\Gamma_F},\label{symcurlv:in:sigma2}\\
        &[\![\mathbf t^{\mathrm T}(\operatorname{sym}\operatorname{curl}\bm \phi) \mathbf n]\!]_{\mathbf x}=0\quad \text{for all}\, \,\mathbf x \in \mathcal V_F. \label{symcurl:in:sigma3}
    \end{align}
Since $\bm \phi\cdot\mathbf n$ is constant on each straight segment of $\Gamma_S\cup \Gamma_F$, there holds
\begin{equation}\label{symcurlv:in:sigma:a1}
 \mathbf n^{\mathrm T} \operatorname{sym}\operatorname{curl}\bm \phi\mathbf n|_{\Gamma_S\cup \Gamma_F} =\partial_t(\bm \phi\cdot \mathbf n) |_{\Gamma_S\cup \Gamma_F}=0.  
\end{equation}
Besides, $\bm \phi\cdot\mathbf t$ is a linear function on each straight segment of $\Gamma_F$. This combined with the identities \eqref{symcurlv:2}--\eqref{symcurlv:3} yields \eqref{symcurlv:in:sigma2}:
\begin{equation*}
    (\mathbf n^{\mathrm T}\operatorname{div}\left( \operatorname{sym}\operatorname{curl}\bm \phi\right)+\partial_t (\mathbf t^{\mathrm T}\left( \operatorname{sym}\operatorname{curl}\bm \phi\right) \mathbf n))|_{\Gamma_F} = \partial_{tt}(\bm \phi\cdot\mathbf t)|_{\Gamma_F}=0.
\end{equation*}
For any vertex $\mathbf x\in \mathcal V_h$, $\bm \phi(\mathbf x)$ and $\nabla\bm \phi(\mathbf x)$ are continuous. This, $\mathbf t^{\mathrm T}\operatorname{sym}\operatorname{curl}\bm \phi\mathbf n = -\frac{1}{2}(\operatorname{div}\bm \phi-\partial_t(\bm \phi\cdot\mathbf t))$, and $[\![\partial_{t}(\bm \phi\cdot \mathbf t)]\!]_{\mathbf x}=0$ for all $\mathbf x\in \mathcal V_F$ conclude \eqref{symcurl:in:sigma3}.

\end{proof}

%%%%%%%%%%%%%%%%%

\begin{lemma}\label{onto:map}
    For $k\geq 3$, there holds
    \begin{equation*}
     \operatorname{div}\operatorname{div}\Sigma_{h,0} = U_h.
    \end{equation*}
\end{lemma}
\begin{proof}
   The inclusion relationship $ \operatorname{div}\operatorname{div}\Sigma_{h,0} \subset U_h$ immediately follows from the exactness of the discrete divdiv complex \eqref{eq:DisComp}. It suffices to prove $ U_h\subset \operatorname{div}\operatorname{div}\Sigma_{h,0}$. Give any $q_h\in U_h$, the proof of Theorem \ref{inf:sup:con:bd} in Appendix \ref{sec:appendix} ensures the existence of some continuous function $\bm \tau\in \Sigma_0\cap H^1(\operatorname{div},\Omega;\mathbb S)$ with $\mathbf n^{\mathrm T}\bm \tau\mathbf n |_{\Gamma_S\cup \Gamma_F}=0$, $\mathbf t^{\mathrm T}\bm \tau\mathbf n|_{\Gamma_F}=0$, and $\operatorname{div}\bm \tau\cdot \mathbf n|_{\Gamma_F}=0$ such that $\operatorname{div}\operatorname{div}\bm \tau = q_h$. This function, with the quasi-interpolation operator in \cite[Theorem 3.1]{hu2021family} modified on $\Gamma_S$ and $\Gamma_F$ acted upon, gives rise to a $\bm \tau^\ast\in \Sigma_{h,0}$ such that 
   \begin{equation*}
       \operatorname{div}\operatorname{div}\bm \tau^\ast = \mathcal Q_h\operatorname{div}\operatorname{div}\bm \tau = \mathcal Q_h q_h = q_h.
   \end{equation*}
   Thus $ U_h\subset \operatorname{div}\operatorname{div}\Sigma_{h,0}$. 
\end{proof}

\begin{theorem}\label{Dis:exact:mbd}
    The divdiv complex 
    \begin{equation}\label{eq:DisComp:bd}
0 \stackrel{\subset}{\longrightarrow} V_{h,0} \stackrel{\operatorname{sym}\operatorname{curl}}{\longrightarrow} \Sigma_{h,0} \stackrel{\operatorname{div}\operatorname{div}}{\longrightarrow}U_{h} \rightarrow 0
\end{equation}
is exact on a contractible domain. 
\end{theorem}
\begin{proof}
 The exactness at $V_{h,0}$ is trivial and at $U_h$ follows from Lemma \ref{onto:map} immediately. To confirm the exactness at $\Sigma_{h,0}$, it suffices to prove $\operatorname{sym}\operatorname{curl} V_{h,0} = \Sigma_{h,0}\cap \operatorname{ker}(\operatorname{div}\operatorname{div})$. Lemma \ref{rightarrow} leads to $\operatorname{sym}\operatorname{curl} V_{h,0} \subset \Sigma_{h,0}\cap \operatorname{ker}(\operatorname{div}\operatorname{div})$. 
The dimensions of $\Sigma_{h,0}$ and $V_{h,0}$ are counted below to demonstrate the above equality.

Let $N_0$ represent the number of intersection points of the boundary parts $\Gamma_S$ and $\Gamma_F$. Let $\Gamma_{F_j}$ be a connected component of $\Gamma_F$ and $\Gamma_{S_i}$ a connected component of $\Gamma_S$ with $1\leq j\leq J$ and $1\leq i\leq I$. Recall that $\mathcal V_{F_j}$ and $\mathcal V_{S_i}$ are the sets of all interior corner points on $\Gamma_{F_j}$ and $\Gamma_{S_i}$, respectively.
 Let 
$\# \mathcal E_h(\Gamma_X)$ and $\#\mathcal V_h(\Gamma_X)$ respectively denote the number of edges and vertexes on $\Gamma_X$.  
Note that 
\begin{equation*}
  \operatorname{dim}\Sigma_{h,0} = \operatorname{dim}\Sigma_{h} - N_{1,\text{bc}},
\end{equation*}
where $N_{1,\text{bc}}$ counts the constraints from boundary conditions, and $N_{1,\text{bc}}$ reads the summation of 
\begin{align*}
  &\#\left(\mathcal V_h(\Gamma_{F_j})\backslash\mathcal V_{F_j}\right)+(2k-1)\#\mathcal E_h(\Gamma_{F_j})+3\#\mathcal V_{F_j}~~&&\text{on}~~\Gamma_{F_j},\\
  &\#\mathcal V_h(\Gamma_{S_i})+(k-1)\#\mathcal E_h(\Gamma_{S_i})+\#\mathcal V_{S_i}~~&&\text{on}~~\Gamma_{S_i},
\end{align*}
by subtracting $N_0$ of the repeated constraints at the intersection points in $\overline{\Gamma}_F\cap \overline{\Gamma}_S$.
Similarly, 
\begin{equation*}
  \operatorname{dim}V_{h,0} = \operatorname{dim}V_{h} - N_{2,\text{bc}},
\end{equation*}
where $N_{2,\text{bc}}$ counts the constraints from boundary conditions, and $N_{2,\text{bc}}$ reads the summation of 
\begin{align*}
    &4\#\left(\mathcal V_h(\Gamma_{F_1})\backslash\mathcal V_{F_1}\right)+(2k-4)\#\mathcal E_h(\Gamma_{F_1})+6\#\mathcal V_{F_1}~~&&\text{on}~~\Gamma_{F_1},\\
    &4\#\left(\mathcal V_h(\Gamma_{F_j})\backslash\mathcal V_{F_j}\right)+(2k-4)\#\mathcal E_h(\Gamma_{F_j})+6\#\mathcal V_{F_j}-3~~&&\text{on}~~\Gamma_{F_j}, j\neq 1,\\
    &2\#\mathcal V_h(\Gamma_{S_i})+(k-2)\#\mathcal E_h(\Gamma_{S_i})+2\#\mathcal V_{S_i}-(\#\mathcal V_{S_i}+1)~~&&\text{on}~~\Gamma_{S_i},
    \end{align*}
by subtracting $N_0$ of the repeated constraints at the intersection points in $\overline{\Gamma}_F\cap \overline{\Gamma}_S$.
Therefore, 
\begin{align*}
   N_{2,\text{bc}}-  N_{1,\text{bc}}=3\sum_{j}\left(\#\mathcal V_h(\Gamma_{F_j})-\#\mathcal E_h(\Gamma_{F_j})-1\right)+3+\sum_{i}\left(\#\mathcal V_h(\Gamma_{S_i})-\#\mathcal E_h(\Gamma_{S_i})-1\right).
\end{align*}
Since $\#\mathcal V_h(\Gamma_{F_j})-\#\mathcal E_h(\Gamma_{F_j})-1=0$ and $\#\mathcal V_h(\Gamma_{S_i})-\#\mathcal E_h(\Gamma_{S_i})-1=0$ on each connected component, it holds 
\[N_{2,\text{bc}}-N_{1,\text{bc}}=3.\] 

The exactness of the discrete divdiv complex \eqref{eq:DisComp} shows that $\operatorname{dim}\Sigma_h-\operatorname{dim}U_h = \operatorname{dim}V_h-\operatorname{dim}RT$. Accordingly,
\begin{align*}
\operatorname{dim}\left(\operatorname{sym}\operatorname{curl} V_{h,0}\right)= \operatorname{dim} V_{h,0}-0=\operatorname{dim} V_{h}- N_{2,\text{bc}},
\end{align*}
and
\begin{align*}
\operatorname{dim}\left(\Sigma_{h,0}\cap \operatorname{ker}(\operatorname{div}\operatorname{div})\right) &= \operatorname{dim} \Sigma_{h,0}-\operatorname{dim} (\operatorname{div}\operatorname{div}\Sigma_{h,0})= \operatorname{dim} \Sigma_{h,0}-\operatorname{dim} U_h \\
 &= \operatorname{dim}\Sigma_h -\operatorname{dim} U_h-N_{1,\text{bc}}=\operatorname{dim}V_h-\operatorname{dim}RT-N_{1,\text{bc}}\\
 &=\operatorname{dim}V_h-\operatorname{dim}RT-(N_{2,\text{bc}}-3)=\operatorname{dim}V_h-N_{2,\text{bc}}.
\end{align*}
Thus $\operatorname{dim}\left(\operatorname{sym}\operatorname{curl} V_{h,0}\right)=\operatorname{dim}\left(\Sigma_{h,0}\cap \operatorname{ker}(\operatorname{div}\operatorname{div})\right)$ follows. Consequently, $\operatorname{sym}\operatorname{curl} V_{h,0} = \Sigma_{h,0}\cap \operatorname{ker}(\operatorname{div}\operatorname{div})$. This concludes the exactness at $\Sigma_{h,0}$.
\end{proof}

\begin{remark}
   It can be shown from Theorems \ref{con:exact:mbd}-\ref{Dis:exact:mbd} that the particular definition of $V_0$ is not only sufficient to ensure the exactness of the divdiv complexes at both continuous and discrete levels, but in fact characterizes the only admissible boundary conditions under which such exactness can be achieved. 
\end{remark}

\section{A posteriori error estimation for the symmetric bending moment and the postprocessed deflection}
This section presents a posteriori error estimation for the symmetric bending moment, which induces an a posteriori error estimator:
\begin{equation}\label{estimator:global}
    \begin{aligned}
    \eta =\eta_1+\operatorname{osc}.
\end{aligned}
\end{equation}
Here
\begin{equation*}
    \begin{aligned}
    \eta_1^2 &=\sum_{K\in\mathcal T_h}h_K^2\|\operatorname{rot}(\mathbb C^{-1}\bm \sigma_h)\|^2_{0,K}+\sum_{e\in \mathcal E_h(\Omega)}h_e\|[\![\mathbb C^{-1}\bm \sigma_h\mathbf t_e]\!]\|^2_{0,e}\\
   &+\sum_{e\in \mathcal E_h(\Gamma_C)}h_e\|\mathbb C^{-1}\bm \sigma_h\mathbf t- \partial_{tt} w_{b}\mathbf t-\partial_{t}g_b\mathbf n\|^2_{0,e}+\sum_{e\in \mathcal E_h(\Gamma_S)}h_e\|\mathbf t^{\mathrm T}\mathbb C^{-1}\bm \sigma_h\mathbf t-\partial_{tt} w_{b}\|^2_{0,e}.
\end{aligned}
\end{equation*}
Recall that $\pi_h$, as defined in \eqref{mb:inter}, denotes the interpolation operator that maps onto the space of piecewise $P_k$ on $\Gamma_S\cup \Gamma_F$, and $\mathcal P_{h}^{k-1}$ is the $L^2$ projection operator onto piecewise $P_{k-1}$ on $\Gamma_F$.
In \eqref{estimator:global}, $\operatorname{osc}: = \operatorname{osc}(f,\mathcal T_h)+\operatorname{osc}( m_{b},\mathcal E_h(\Gamma_S\cup\Gamma_F))+\operatorname{osc}( h_{b},\mathcal E_h(\Gamma_F))$ with 
\begin{align*}
    &\operatorname{osc}^2(f,\mathcal T_h): =\sum_{K\in \mathcal T_h}h_K^4\|f-\mathcal Q_h f\|_{0,K}^2,\\
    &\operatorname{osc}^2(m_b, \mathcal E_h(\Gamma_S\cup\Gamma_F)): = \sum_{e\in \mathcal E_h(\Gamma_S\cup\Gamma_F)}h_e\|m_{b}-\pi_h m_{b}\|_{0,e}^2,\\
     &\operatorname{osc}^2(h_{b}, \mathcal E_h(\Gamma_F)): = \sum_{e\in \mathcal E_h(\Gamma_F)}h_e^3\|h_{b}-\mathcal P^{k-1}_h h_{b}\|_{0,e}^2.
\end{align*}

\subsection{Reliability}
\begin{theorem}[Upper bound] \label{relia:bd:upper}
Let $(\bm \sigma, w)$ and $(\bm \sigma_h, w_h)$ be the solution to the continuous \eqref{mixed:bd:con} and the discrete \eqref{mixed:bd:dis} mixed formulation with mixed boundary conditions, respectively. It holds
    \begin{align*}
        \|\bm \sigma-\bm\sigma_h\|_{\mathbb C^{-1}}\lesssim \eta.
    \end{align*}
\end{theorem}

To prove Theorem \ref{relia:bd:upper}, an auxiliary problem is introduced: Find $\underline{w}\in H^2(\Omega)$ satisfying $\underline{w}|_{\Gamma_C\cup \Gamma_S}=w_b$ and $\partial_n \underline{w}|_{\Gamma_S}=g_b$ such that 
\begin{equation}\label{auxilary:problem}
    (\mathbb C\nabla^2 \underline{w},\nabla^2 v)=(\bm \sigma_h, \nabla^2 v)\quad \text{for all}~~v\in \Lambda.
\end{equation}
Let $\underline{\bm \sigma}=\mathbb C\nabla^2\underline{w}$. The formula \eqref{auxilary:problem} together with Remark \ref{b:d:sigmahb} immediately yields
\begin{equation}\label{aux:mid}
    \begin{aligned}
    (\underline{\bm \sigma}, \nabla^2 v)=(\bm \sigma_h,\nabla^2 v) = (\operatorname{div}\operatorname{div}\bm \sigma_h, v)+R_{h,b}(v)=(\mathcal Q_h f, v)+R_{h,b}(v).
\end{aligned}
\end{equation}
An estimate concerning the boundary data and $f$ is given below.
\begin{lemma}\label{relia:bd:1}
    There holds \[\|\bm \sigma-\underline{\bm \sigma}\|_{\mathbb C^{-1}}\lesssim \operatorname{osc}.\]
\end{lemma}
\begin{proof}
For $w$ and $\underline{w}$ solving the primal formulation \eqref{primal:formulation} and the auxiliary problem \eqref{auxilary:problem} respectively, let $v:= w-\underline{w}\in \Lambda$ and $|v|_2 = \|\nabla^2 w-\nabla^2\underline{w}\|_0\lesssim\|\bm \sigma-\underline{\bm \sigma}\|_{\mathbb C^{-1}}$. Besides, \eqref{primal:formulation} and \eqref{auxilary:problem}--\eqref{aux:mid} lead to
\begin{equation}\label{osc:estimate:0}
    \begin{aligned}
     \|\bm \sigma - \underline{\bm \sigma}\|_{\mathbb C^{-1}}^2&=(\mathbb C(\nabla^2 w-\nabla^2 \underline{w}), \nabla^2v)=(\mathbb C \nabla^2 w, \nabla^2 v)-(\mathbb C \nabla^2 \underline{w}, \nabla^2 v)\\
     &=(f-Q_hf, v)+R_b(v)-R_{h,b}(v).
\end{aligned}
\end{equation}

Subtracting \eqref{rhb:def} from \eqref{rb:def} yields
\begin{align*}
    R_b(v)-R_{h,b}(v) &= \sum_{e\in \mathcal E_h(\Gamma_S\cup\Gamma_F)}\langle m_b-\pi_h m_b,\partial_n v\rangle_{e}-\sum_{e\in \mathcal E_h(\Gamma_F)}\langle h_b-\mathcal P^{k-1}_h h_b,v\rangle_{e}.
\end{align*}
For $e\in \mathcal E_h(K)\cap \mathcal E_h(\Gamma_S\cup \Gamma_F)$, let $\bm \varphi_{0,e} $ be the $L^2$ projection of $\nabla v$ onto $P_0(K;\mathbb R^2)$. The Cauchy-Schwarz inequality and the trace inequality lead to
\begin{align*}
    \langle m_b-\pi_h m_b,\partial_n v\rangle_{e}&=\langle m_b-\pi_h m_b,\partial_n v-\bm \varphi_{0,e}
\cdot\mathbf n\rangle_{e}\\
    &\lesssim \|m_b-\pi_h m_b\|_{0,e}\|\nabla v-\bm \varphi_{0,e}\|_{0,e}
    \lesssim h_e^{\frac{1}{2}}\|m_b-\pi_h m_b\|_{0,e}|v|_{2,K}.
\end{align*}
Similarly, for $e\in \mathcal E_h(K)\cap \mathcal E_h(\Gamma_F)$, one can get 
\begin{align*}
     \langle h_b-\mathcal P^{k-1}_h h_b,v\rangle_{e}&=\langle h_b-\mathcal P^{k-1}_h h_b,v-v_{1,e}\rangle_{e}\\
     &\lesssim \|h_b-\mathcal P^{k-1}_h h_b\|_{0,e}\|v-v_{1,e}\|_{0,e}
    \lesssim h_e^{\frac{3}{2}}\|h_b-\mathcal P^{k-1}_h h_b\|_{0,e}|v|_{2,K}
\end{align*}
with $v_{1,e}$ being the $L^2$ projection of $v$ onto $P_1(K)$. Therefore, 
\begin{equation}\label{rb-rhb}
    R_b(v)-R_{h,b}(v)\lesssim\left(\operatorname{osc}(m_b,\mathcal E_h(\Gamma_S\cup \Gamma_F))+\operatorname{osc}(h_b,\mathcal E_h(\Gamma_F))\right)|v|_2.
\end{equation}

Furthermore, the Cauchy-Schwarz inequality shows that 
\[(f-\mathcal Q_h f,v)=(f-\mathcal Q_hf,v-\mathcal Q_h^1 v)\lesssim\left(\sum_{K\in \mathcal T_h}h_K^4\|f-\mathcal Q_h f\|_{0,K}^2\right)^{\frac{1}{2}}|v|_2.\]
Substituting this and \eqref{rb-rhb} into \eqref{osc:estimate:0} concludes
\begin{equation}
    \|\bm \sigma-\underline{\bm \sigma}\|_{\mathbb C^{-1}}\lesssim \operatorname{osc}(f,\mathcal T_h)+\operatorname{osc}(m_b,\mathcal E_h(\Gamma_S\cup \Gamma_F))+\operatorname{osc}(h_b,\mathcal E_h(\Gamma_F)).
\end{equation}
\end{proof}

To estimate $\|\underline{\bm \sigma}-\bm \sigma_h\|_{\mathbb C^{-1}}$, a representation formula of $\operatorname{tr}_b(u)(\operatorname{sym}\operatorname{curl}\bm \phi)$ is provided.  

\begin{lemma}\label{trace:lemma}
    For $\bm \phi \in V_0$ satisfying $\bm \phi\cdot\mathbf n|_{\Gamma_S}=0$ and $\bm \phi|_{\Gamma_F}=\bm 0$, there holds
    \begin{align*}
        \operatorname{tr}_{b}(u)(\operatorname{sym}\operatorname{curl}\bm \phi) = -\langle \partial_{tt} w_{b}\mathbf t+\partial_{t}g_b\mathbf n, \bm \phi\rangle_{\Gamma_C}-\langle\partial_{tt} w_{b}, \bm \phi\cdot \mathbf t\rangle_{\Gamma_S}.
    \end{align*}
\end{lemma}
\begin{proof}
For $\bm \phi\in V_0\cap C^\infty(\overline{\Omega})$, since $\operatorname{sym}\operatorname{curl}\bm \phi\in \Sigma_0$, substituting it into $\bm \tau$ in \eqref{trace:gammaD} gives rise to
\begin{equation}\label{trace:all:terms}
       \begin{aligned}
           \operatorname{tr}_{b}(u)(\operatorname{sym}\operatorname{curl}\bm \phi) = (\operatorname{sym}\operatorname{curl}\bm \phi, \nabla^2 u)=  (\operatorname{curl}\bm \phi, \nabla^2 u).
       \end{aligned}
   \end{equation}
An integration by parts combined with the given data $u|_{\Gamma_C\cup \Gamma_S} = w_b$ and $\partial_n u |_{\Gamma_C} = g_b$ leads to
\begin{align*}
    (\operatorname{curl}\bm \phi, \nabla^2 u) &= -\langle\bm \phi, \partial_t(\nabla u)\rangle_{\partial\Omega}=-\langle\bm \phi, \partial_t(\partial_t u)\mathbf t+\partial_t(\partial_n u)\mathbf n\rangle_{\partial\Omega}\\
    &= -\langle\bm \phi, \partial_{tt} w_b\mathbf t+\partial_t g_b\mathbf n\rangle_{\Gamma_C}- \langle\bm \phi\cdot\mathbf t, \partial_{tt} w_b\rangle_{\Gamma_S}-\langle\bm \phi, \partial_t(\nabla u)\rangle_{\Gamma_F}-\langle\bm \phi\cdot\mathbf n, \partial_{tn} u\rangle_{\Gamma_S}\\
    &=-\langle\bm \phi, \partial_{tt} w_b\mathbf t+\partial_t g_b\mathbf n\rangle_{\Gamma_C}- \langle\bm \phi\cdot\mathbf t , \partial_{tt} w_b\rangle_{\Gamma_S}.
\end{align*}
The last equality holds because $\bm \phi\cdot\mathbf n$ and $\bm \phi$ vanish on $\Gamma_S$ and $\Gamma_F$, respectively. Finally, a density argument concludes the proof. 
\end{proof}

%Scott-zhang 
%写一点i=1:N与V_h之间的关系， follow the approach \cite[Section 4]{girault2002hermite}, 
For $\bm \phi \in V_0$, a boundary-preserving interpolation operator $I_h: V_0\rightarrow V_{h,0}$ is introduced, which serves as a tool for the proof of reliability. Denote the degrees of freedom of $V_h$, i.e., \eqref{v:dof:1}--\eqref{v:dof:4}, as $\mathcal N_h=\{\bm \phi\rightarrow D_i\bm \phi: i = 1, \cdots, N\}$, and let $\{\bm \varphi_i: i = 1, \cdots, N\}$ be the corresponding basis of $V_h$. Each of the differential operators involved in $D_i \bm \phi$ has order $|D_i|$ equal to zero or one. Following the approach presented in \cite[Section 4]{girault2002hermite}, define 
\begin{equation}\label{int:v0-vh0}
    I_h \bm \phi = \sum_{k=1}^N \left(\int_{\kappa_i}D_i\bm \phi \, \bm \psi_{i}^{\kappa_i}b_{\kappa_i}d\mu(\kappa_i)\right)\bm \varphi_i, 
\end{equation}
where $\psi_i^{\kappa_i}$ denotes Riesz's representation function corresponding to the $i$-th degree of freedom $D_i\bm \phi$ which is supported on $\kappa_i$, $b_{\kappa_i}$ is a bubble function on $\kappa_i$, and $d\mu(\kappa_i)$ represents the Lebesgue measure on $\kappa_i$.

To perform the integration in \eqref{int:v0-vh0} with respect to the degrees of freedom \eqref{v:dof:1}--\eqref{v:dof:4}, $\kappa_i$ is chosen as either a triangle or an edge, depending on the type of degrees of freedom, as follows:
\begin{enumerate}
    \item For the value and gradient at a vertex $\mathbf x\in \mathcal V_h(\Omega)$ (i.e., \eqref{v:dof:1}), $\kappa_i$ is chosen as a triangle $K\in\mathcal T_h$ sharing the vertex $\mathbf x$;
    \item For the moments of the function on an edge $e\in \mathcal E_h$ (i.e., \eqref{v:dof:2}), $\kappa_i$ is chosen as the edge $e$, while for the moments of its divergence on $e\in\mathcal E_h$ (i.e., \eqref{v:dof:3}), $\kappa_i$ is chosen as $K\in\mathcal T_h$ sharing the edge $e$;
    \item  For the interior moments on a triangle $K\in \mathcal T_h$(i.e., \eqref{v:dof:4}), $\kappa_i$ is chosen as the triangle $K$.
\end{enumerate}
Moreover, to preserve the boundary conditions:
\begin{enumerate}[start = 4]
    \item For the value and the tangential component of the gradient at $\mathbf x\in \mathcal V_h(\Gamma_F)$, $\kappa_i$ is chosen as an edge $e\in\mathcal E_h(\Gamma_F)$ sharing the vertex $\mathbf x$; 
    \item For the normal component and the tangential derivative of the normal component at $\mathbf x\in \mathcal V_h(\Gamma_S)\backslash \mathcal V_h(\Gamma_F)$, $\kappa_i$ is chosen as an edge $e\in \mathcal E_h(\Gamma_S)$ sharing the vertex $\mathbf x$; 
    \item  For all remaining vertex-associated degrees of freedom, $\kappa_i$ is chosen as a triangle $K\in\mathcal T_h$ sharing the vertex; 
\end{enumerate}

%%%%%%%%%

The interpolation operator $I_h$ is well-defined due to the unisolvence of the degrees of freedom of $V_{h}$ established in \cite[Theorem 2.6]{hu2021family}. Furthermore, for any $\bm \phi \in V_0$, $I_h\bm \phi\in V_{h,0}$, and satisfies the following error estimates
\begin{equation}\label{estimate:phi}
    \|\bm \phi-I_h\bm \phi\|_{0,K} \lesssim h_K |\bm \phi|_{1,S(K)}, \quad \|\bm \phi-I_h\bm \phi\|_{0,e} \lesssim h_e^\frac{1}{2}|\bm \phi|_{1,S(\omega_e)},
\end{equation}
where $\omega_e$ denotes the patch of elements containing edge $e$ and $S(X):=\operatorname{int}(\cup\{K^\prime\in\mathcal T_{h}: \, \operatorname{dist}(X,K^\prime)=0\})$.

\begin{lemma}\label{relia:bd:2}There holds
\begin{equation}
\begin{aligned}
     \|\underline{\bm \sigma}-\bm \sigma_h\|_{\mathbb C^{-1}} \lesssim \eta_1.
\end{aligned}
\end{equation}
\end{lemma}
\begin{proof}
The auxiliary problem \eqref{auxilary:problem} shows that $\operatorname{div}\operatorname{div}(\underline{\bm \sigma}-\bm \sigma_{h})=0$. Since $\underline{\bm \sigma}-\bm \sigma_{h}\in \Sigma_{0}$, by Lemma \ref{leftarrow}, there exists a function $\bm \phi\in V_0$ such that $\operatorname{sym}\operatorname{curl}\bm \phi=\underline{\bm \sigma}-\bm \sigma_{h}$ and $|\bm \phi|_{1}\lesssim \|\underline{\bm \sigma}-\bm\sigma_{h}\|_{0}$. The auxiliary problem \eqref{auxilary:problem} and the definition of $\operatorname{tr}_{b}$ in \eqref{trace:gammaD} show
\begin{equation*}
(\mathbb C^{-1}\underline{\bm\sigma}, \operatorname{sym}\operatorname{curl}\bm \phi) =(\nabla^2\underline{w}, \operatorname{sym}\operatorname{curl}\bm \phi)= \operatorname{tr}_{b}(\underline{w})(\operatorname{sym}\operatorname{curl}\bm \phi)=\operatorname{tr}_{b}(u)(\operatorname{sym}\operatorname{curl}\bm \phi).
\end{equation*}
Consequently, 
\begin{equation}\label{sigmawan:minus:sigmah}
   \begin{aligned}
     \|\underline{\bm \sigma}-\bm \sigma_h\|^2_{\mathbb C^{-1}} &= \sum_{K\in \mathcal T_h}(\mathbb C^{-1}(\underline{\bm \sigma}-\bm \sigma_h), \operatorname{sym}\operatorname{curl}\bm \phi)_{K}\\
     &= \operatorname{tr}_{b}(u)(\operatorname{sym}\operatorname{curl}\bm \phi)-\sum_{K\in \mathcal T_h}(\mathbb C^{-1} \bm \sigma_h, \operatorname{sym}\operatorname{curl}\bm \phi)_{K}.
\end{aligned} 
\end{equation}
For $I_h\bm \phi\in V_{h,0}$, Lemma \ref{rightarrow} shows $\operatorname{sym}\operatorname{curl} I_h \bm \phi \in \Sigma_{h,0}$. 
The substitution of $\operatorname{sym}\operatorname{curl} I_h \bm \phi$ into $\bm \tau_h$ in the first equation of the discrete variational formulation \eqref{mixed:bd:dis} leads to
\begin{equation*}
    (\mathbb C^{-1}\bm \sigma_h, \operatorname{sym}\operatorname{curl}I_h\bm \phi) =  \operatorname{tr}_{b}(u)(\operatorname{sym}\operatorname{curl}I_h\bm \phi).
\end{equation*}
Inserting a term $I_h\bm \phi$ into the right hand side of \eqref{sigmawan:minus:sigmah} gives rise to 
\begin{equation}
  \begin{aligned}
    \|\underline{\bm \sigma}-\bm \sigma_h\|^2_{\mathbb C^{-1}}&= -\sum_{K\in \mathcal T_h}(\mathbb C^{-1} \bm \sigma_h, \operatorname{sym}\operatorname{curl}(\bm \phi-I_h \bm \phi))_{K}\\
    &+\operatorname{tr}_{b}(u)(\operatorname{sym}\operatorname{curl}(\bm \phi-I_h\bm \phi)).\label{sigmawan:minus:sigmah:1}
\end{aligned}  
\end{equation}
 Due to $(\bm \phi-I_h\bm \phi) \in V_0$ satisfying $(\bm \phi-I_h\bm \phi)\cdot\mathbf n|_{\Gamma_S}=0$ and $(\bm \phi-I_h\bm \phi)|_{\Gamma_F}=\bm 0$, Lemma \ref{trace:lemma} results in 
\begin{equation}
    \label{bd:sum}
    \begin{aligned}
       &\operatorname{tr}_{b}(u)(\operatorname{sym}\operatorname{curl}(\bm \phi-I_h \bm \phi))\\
       &=-\langle \partial_{tt} w_{b}\mathbf t+\partial_{t}g_b\mathbf n, \bm \phi-I_h \bm \phi\rangle_{\Gamma_C}-\langle \partial_{tt} w_{b}, (\bm \phi-I_h \bm \phi)\cdot \mathbf t\rangle_{\Gamma_S}.
    \end{aligned}
\end{equation}
Furthermore, the symmetry of $\bm \sigma_h$ and an integration by parts lead to
\begin{align*}
 &-\sum_{K\in \mathcal T_h}(\mathbb C^{-1}\bm \sigma_h, \operatorname{sym}\operatorname{curl}(\bm \phi-I_h \bm \phi))_{K}=-\sum_{K\in \mathcal T_h}(\mathbb C^{-1}\bm \sigma_h, \operatorname{curl}(\bm \phi-I_h\bm \phi))_{K}\\
    & =\sum_{K\in \mathcal T_h}\langle\mathbb C^{-1}\bm \sigma_h\mathbf t, \bm \phi-I_h\bm \phi\rangle_{\partial K} - \sum_{K\in \mathcal T_h}(\operatorname{rot}(\mathbb C^{-1}\bm \sigma_h), \bm \phi-I_h\bm \phi)_{K}\\
    & = \sum_{e\in \mathcal E_h(\Omega)}\langle[\![\mathbb C^{-1}\bm \sigma_h\mathbf t_e]\!], \bm \phi-I_h\bm \phi\rangle_e+\sum_{e\in\mathcal E_h(\Gamma_C)}\langle\mathbb C^{-1}\bm \sigma_h\mathbf t, \bm \phi-I_h\bm \phi\rangle_e\\
    &+\sum_{e\in\mathcal E_h(\Gamma_S)}\langle\mathbf t^{\mathrm T}\mathbb C^{-1}\bm \sigma_h\mathbf t, (\bm \phi-I_h\bm \phi)\cdot \mathbf t\rangle_e- \sum_{K\in \mathcal T_h}(\operatorname{rot}(\mathbb C^{-1}\bm \sigma_h), \bm \phi-I_h\bm \phi)_{K}.
\end{align*}
Substituting this and \eqref{bd:sum} back into \eqref{sigmawan:minus:sigmah:1} results in
\begin{align*}
    \|\underline{\bm \sigma}-\bm \sigma_h\|^2_{\mathbb C^{-1}} &=\sum_{e\in \mathcal E_h(\Omega)}\langle[\![\mathbb C^{-1}\bm \sigma_h\mathbf t_e]\!], \bm \phi-I_h\bm \phi\rangle_e-\sum_{K\in \mathcal T_h}(\operatorname{rot}(\mathbb C^{-1}\bm \sigma_h), \bm \phi-I_h\bm \phi)_{K}\\
    &+\sum_{e\in\mathcal E_h(\Gamma_C)}\langle\mathbb C^{-1}\bm \sigma_h\mathbf t-\partial_{tt} w_{b} \mathbf t-\partial_{t}g_b\mathbf n, \bm \phi-I_h\bm \phi\rangle_e\\
    &+\sum_{e\in\mathcal E_h(\Gamma_S)}\langle\mathbf t^{\mathrm T}\mathbb C^{-1}\bm \sigma_h\mathbf t- \partial_{tt} w_{b}, (\bm \phi-I_h\bm \phi)\cdot\mathbf t\rangle_e.
  \end{align*}
The Cauchy-Schwartz inequality and the interpolation error estimates for $I_h$ in \eqref{estimate:phi} show
   \begin{align*}
     \|\underline{\bm \sigma}-\bm \sigma_h\|^2_{\mathbb C^{-1}}& \lesssim \sum_{K\in\mathcal T_h}h_K\|\operatorname{rot}(\mathbb C^{-1}\bm \sigma_h)\|_{0,K}|\bm \phi|_{1,K}+\sum_{e\in \mathcal E_h(\Omega)}h_e^{\frac{1}{2}}\|[\![\mathbb C^{-1}\bm \sigma_h\mathbf t_e]\!]\|_{0,e}|\bm \phi|_{1,S(e)}\\
    &+\sum_{e\in \mathcal E_h(\Gamma_C)}h_e^{\frac{1}{2}}\|\mathbb C^{-1}\bm \sigma_h\mathbf t- \partial_{tt} w_{b}\mathbf t-\partial_{t}g_b\mathbf n\|_{0,e}|\bm \phi|_{1,S(e)}\\
    &+\sum_{e\in \mathcal E_h(\Gamma_S)}h_e^{\frac{1}{2}}\|\mathbf t^{\mathrm T}\mathbb C^{-1}\bm \sigma_h\mathbf t-\partial_{tt} w_{b} \mathbf t\|_{0,e}|\bm \phi|_{1,S(e)}.
\end{align*}
 Since $|\bm \phi|_1\lesssim \|\underline{\bm \sigma}-\bm \sigma_h\|_0\lesssim  \|\underline{\bm \sigma}-\bm \sigma_h\|_{\mathbb C^{-1}}$, canceling out $\|\underline{\bm \sigma}-\bm \sigma_h\|_{\mathbb C^{-1}}$ on both sides of the above inequality concludes the proof.
\end{proof}

Finally, one can obtain the upper bound of $\|\bm \sigma-\bm \sigma_h\|_{\mathbb C^{-1}}$ through the estimation of $\|\bm \sigma-\underline{\bm \sigma}\|_{\mathbb C^{-1}}$ and $\|\underline{\bm \sigma}-\bm \sigma_h\|_{\mathbb C^{-1}}$. 
\begin{proof}[Proof of Theorem \ref{relia:bd:upper}]
 Lemmas \ref{relia:bd:1} and \ref{relia:bd:2} combined with the triangle inequality yield the conclusion immediately.  
\end{proof}

\subsection{Efficiency}This subsection shows the efficiency of the estimator by following the approach from \cite{alonso1996error}. To this end, for $m\geq k+4$, let $\Psi|_{K}\in P_m(K)$ for all $K\in \mathcal T_h$ with
\begin{align*}
    \Psi (\mathbf x) &= 0\quad &&\text{for all}~~a \in \mathcal V_h,\\
\nabla\Psi (\mathbf x) &= 0\quad &&\text{for all}~~a \in \mathcal V_h,\\
\langle\Psi, \bm q \rangle_e &= -\langle h_e[\![\mathbb C^{-1}\bm \sigma_h\mathbf t_e]\!], \bm q \rangle_e\quad &&\text{for all}~~\bm q\in P_{m-4}(e;\mathbb R^2),~~e \in \mathcal E_h(\Omega),\\
\langle\Psi, \bm q \rangle_e &= -h_e\langle\mathbb C^{-1}\bm \sigma_h\mathbf t-\partial_t g_{b}\mathbf n-\partial_{tt} w_{b}\mathbf t, \bm q \rangle_e\quad&&\text{for all}~~\bm q\in P_{m-4}(e;\mathbb R^2),~~e \in \mathcal E_h(\Gamma_C),\\
\langle\Psi\cdot\mathbf t, q \rangle_e &= -h_e\langle\mathbf t^{\mathrm T}\mathbb C^{-1}\bm \sigma_h\mathbf t-\partial_{tt} w_{b}, q \rangle_e\quad&&\text{for all}~~q\in P_{m-4}(e),~~e \in \mathcal E_h(\Gamma_S),\\
\langle\Psi\cdot\mathbf n,  q \rangle_e &= 0\quad&&\text{for all}~ ~q\in P_{m-4}(e),~~e \in \mathcal E_h(\Gamma_S),\\
\langle\Psi, \bm q\rangle_e&=0\quad&&\text{for all}~~\bm q\in P_{m-4}(e;\mathbb R^2), ~~e\in \mathcal E_h(\Gamma_F),\\
\langle\operatorname{div}\Psi, q\rangle_e &= 0\quad &&\text{for all}~~q\in P_{m-3}(e),~~e \in \mathcal E_h,\\
(\Psi,\bm r)_K &= h^2_K(\operatorname{rot}(\mathbb C^{-1}\bm \sigma_h),\bm r)_K, \quad &&\text{for all}~~\bm r\in RT_{m-5}^\perp,~~K\in \mathcal T_h.
\end{align*}

The above definition coincides with the degrees of freedom \eqref{v:dof:1}--\eqref{v:dof:4} for degree $m$. Indeed, $\Psi\in V_0$; more precisely, it lies in the analogous space as $V_{h,0}$, but with local polynomial degree $m$ instead of $k+1$. Furthermore, the definition of $\Psi$ shows 
\begin{equation}\label{psi:estimation}
\begin{aligned}
    \|\Psi\|_{0,K}^2 &\lesssim h_K^4\|\operatorname{rot}(\mathbb C^{-1}\bm \sigma_h)\|_{0,K}^2+\sum_{e\in \mathcal E(K)\cap\mathcal E_h(\Omega)}h^3_e\|[\![\mathbb C^{-1}\bm \sigma_h\mathbf t_e]\!]\|_{0,e}^2\\
    &+\sum_{e\in \mathcal E(K)\cap\mathcal E_h(\Gamma_C)}h^3_e\|\mathbb C^{-1}\bm \sigma_h\mathbf t-\partial_t g_{b}\mathbf n-\partial_{tt} w_{b}\mathbf t\|_{0,e}^2\\
    &+\sum_{e\in \mathcal E(K)\cap\mathcal E_h(\Gamma_S)}h^3_e\|\mathbf t^{\mathrm T}\mathbb C^{-1}\bm \sigma_h\mathbf t-\partial_{tt} w_{b}\|_{0,e}^2. 
\end{aligned}
\end{equation}
For any edge $e$ and $m\geq k+4$, recall that $\mathcal P_e^{m-4}$ denotes the $L^2$ projection onto $P_{m-4}(e)$. Define 
\begin{align*}
  &\operatorname{osc}^2_1(w_b,g_b,\Gamma_C):= \sum_{e\in\mathcal E_h(\Gamma_C)}h_e\|\partial_{tt} w_{b}\mathbf t+\partial_t g_b\mathbf n-\mathcal P_e^{m-4}(\partial_{tt} w_{b}\mathbf t+\partial_t g_b\mathbf n)  \|_{0,e}^2,\\
  &\operatorname{osc}^2_2(w_b,\Gamma_S):= \sum_{e\in\mathcal E_h(\Gamma_S)}h_e\|\partial_{tt} w_{b}-\mathcal P_e^{m-4}\partial_{tt} w_{b}  \|_{0,e}^2.
\end{align*} 

The efficiency is demonstrated below.
\begin{theorem}[Lower bound]There holds
    \begin{equation}\label{efficiency:1}
   \eta_1\lesssim \|\bm \sigma - \bm \sigma_h\|_{\mathbb C^{-1}}+\operatorname{osc}_{1}(w_b,g_b,\Gamma_C)+\operatorname{osc}_{2}(w_b,\Gamma_S).
    \end{equation}
\end{theorem}
\begin{proof}
For simplicity, denote $\bm \xi_b := \partial_{tt} w_{b}\mathbf t+\partial_t g_{b}\mathbf n$. For $e\in\mathcal E_h(\Gamma_C)$, note that
\begin{align*}
    h_e\|\mathbb C^{-1}\bm \sigma_h\mathbf t-\bm \xi_b\|_{0,e}^2&=h_e\langle\mathbb C^{-1}\bm \sigma_h\mathbf t-\bm \xi_b, \mathbb C^{-1}\bm \sigma_h\mathbf t-\mathcal P_e^{m-4}\bm \xi_b+\mathcal P_e^{m-4}\bm \xi_b-\bm \xi_b\rangle_e\\
   &=-\langle\Psi,\mathbb C^{-1}\bm \sigma_h\mathbf t-\mathcal P_e^{m-4}\bm \xi_b\rangle_e+h_e\langle\mathbb C^{-1}\bm \sigma_h\mathbf t-\bm \xi_b, \mathcal P_e^{m-4}\bm \xi_b-\bm \xi_b\rangle_e\\
   &= -\langle\Psi,\mathbb C^{-1}\bm \sigma_h\mathbf t-\bm \xi_b\rangle_e-\rho_{C,e},
\end{align*}
where $\rho_{C,e}:=\langle\Psi,\bm \xi_b-\mathcal P_e^{m-4}\bm \xi_b\rangle_e-h_e\langle\mathbb C^{-1}\bm \sigma_h\mathbf t-\bm \xi_b, \mathcal P_e^{m-4}\bm \xi_b-\bm \xi_b\rangle_e$. 
Similarly, for $e\in\mathcal E_h(\Gamma_S)$,
\begin{align*}
    h_e\|\mathbf t^{\mathrm T}\mathbb C^{-1}\bm \sigma_h\mathbf t-\partial_{tt} w_{b}\|_{0,e}^2 = -\langle\Psi\cdot\mathbf t,\mathbf t^{\mathrm T}\mathbb C^{-1}\bm \sigma_h\mathbf t-\partial_{tt} w_{b}\rangle_e-\rho_{S,e}.
\end{align*}
where $\rho_{S,e}:= \langle\Psi\cdot \mathbf t, \partial_{tt}w_b-\mathcal P_e^{m-4}\partial_{tt}w_b\rangle_e-h_e\langle\mathbf t^{\mathrm T}\mathbb C^{-1}\bm \sigma_h\mathbf t-\partial_{tt} w_{b},\mathcal P_e^{m-4}\partial_{tt} w_{b}-\partial_{tt} w_{b}\rangle_e$.
Therefore, $\eta_1$ can be expressed as
     \begin{align*}
\eta_1^2 &= \sum_{K\in\mathcal T_h}\left((\Psi,\operatorname{rot}(\mathbb C^{-1}\bm \sigma_h))_K- \sum_{e\in\mathcal E(K)\cap \mathcal E_h(\Omega)} \langle\Psi, \mathbb C^{-1}\bm \sigma_h\mathbf t_e\rangle_e\right)\\&-\sum_{e\in\mathcal E_h(\Gamma_C)}\langle\Psi,\mathbb C^{-1}\bm \sigma_h\mathbf t-\bm \xi_b\rangle_e-\sum_{e\in\mathcal E_h(\Gamma_S)}\langle\Psi, \mathbb C^{-1}\bm \sigma_h\mathbf t-\partial_{tt} w_{b}\mathbf t\rangle_e-\rho_C-\rho_S.\\
 \end{align*}
Here $\rho_C= \sum\limits_{e\in\mathcal E_h(\Gamma_C)}\rho_{C,e}$ and $\rho_S= \sum\limits_{e\in\mathcal E_h(\Gamma_S)}\rho_{S,e}$.
Note that $\langle\Psi\cdot\mathbf n, \mathbb C^{-1}\bm \sigma_h\mathbf t\rangle_e=0$ for all $e\in \mathcal E_h(\Gamma_S)$. Rearranging the expression and consolidating like terms result in 
\begin{align*}
    \eta_1^2 &= \sum_{K\in\mathcal T_h}\left((\Psi,\operatorname{rot}(\mathbb C^{-1}\bm \sigma_h))_K- \sum_{e\in\mathcal E(K)} \langle\Psi, \mathbb C^{-1}\bm \sigma_h\mathbf t_e\rangle_e\right)\\
     &+\sum_{e\in\mathcal E_h(\Gamma_C)}\langle\Psi, \bm \xi_b\rangle_e+\sum_{e\in\mathcal E_h(\Gamma_S)}\langle\Psi\cdot \mathbf t, \partial_{tt} w_{b}\rangle_e-\rho_C-\rho_S.
\end{align*}
An integration by parts leads to 
\begin{equation}\label{eta:2:expression}
   \begin{aligned}
   \eta_1^2 & =(\mathbb C^{-1}\bm \sigma_h, \operatorname{sym}\operatorname{curl}\Psi) +\langle\Psi, \bm \xi_b\rangle_{\Gamma_C} +\langle\Psi\cdot\mathbf t, \partial_{tt} w_{b}\rangle_{\Gamma_S}-\rho_C-\rho_S.
\end{aligned} 
\end{equation}
Lemma \ref{rightarrow} shows that $\operatorname{sym}\operatorname{curl}\Psi\in \Sigma_0$. Substituting $\operatorname{sym}\operatorname{curl}\Psi$ into the first equation of the continuous variational formulation \eqref{mixed:bd:con}, one can obtain $(\mathbb C^{-1}\bm \sigma, \operatorname{sym}\operatorname{curl}\Psi) = \operatorname{tr}_{b}(u)(\operatorname{sym}\operatorname{curl}\Psi)$. 
Lemma \ref{trace:lemma} gives rise to
\begin{align*}
    \quad\operatorname{tr}_{b}(u)(\operatorname{sym}\operatorname{curl}\Psi)= -\langle \bm \xi_b, \Psi\rangle_{\Gamma_C}-\langle\partial_{tt} w_{b}, \Psi\cdot \mathbf t\rangle_{\Gamma_S}.
    \end{align*}
Inserting a term $\mathbb C^{-1}\bm \sigma$ into \eqref{eta:2:expression} yields
\begin{align*}
\eta_1^2 &= (\mathbb C^{-1}(\bm \sigma_h-\bm \sigma), \operatorname{sym}\operatorname{curl}\Psi)+\operatorname{tr}_{b}(u)(\operatorname{sym}\operatorname{curl}\Psi)\\
&+\sum_{e\in\mathcal E_h(\Gamma_C)}\langle\Psi,\bm \xi_b\rangle_e +\sum_{e\in\mathcal E_h(\Gamma_S)}\langle\Psi\cdot\mathbf t, \partial_{tt} w_{b}\rangle_e-\rho_C-\rho_S\\
&=(\mathbb C^{-1}(\bm \sigma_h-\bm \sigma), \operatorname{sym}\operatorname{curl}\Psi)-\rho_C-\rho_S.
    \end{align*}
This, \eqref{psi:estimation}, and the inverse estimate yield
\begin{align*}
\eta_1^2 &\lesssim \|\mathbb C^{-1}(\bm \sigma_h-\bm \sigma)\|_0\|\operatorname{sym}\operatorname{curl}\Psi\|_0+|\rho_C|+|\rho_S|\\
& \lesssim\|\mathbb C^{-1}(\bm \sigma_h-\bm \sigma)\|_0(\sum_{K\in\mathcal T_h}h_K^{-2}\|\Psi\|_{0,K}^2)^{\frac{1}{2}}\\
&+\operatorname{osc}_{1}(w_b,g_b,\Gamma_C)\left((\sum_{e\in\mathcal E_h(\Gamma_C)}h_e\|\mathbb C^{-1}\bm \sigma_h\mathbf t-\bm \xi_b\|_{0,e}^2)^{\frac{1}{2}}+(\sum_{K\in\mathcal T_h}h_K^{-2}\|\Psi\|_{0,K}^2)^{\frac{1}{2}}\right)\\
&+\operatorname{osc}_{2}(w_b,\Gamma_S)\left((\sum_{e\in\mathcal E_h(\Gamma_S)}h_e\|\mathbf t^{\mathrm T}\mathbb C^{-1}\bm \sigma_h\mathbf t-\partial_{tt}w_b\|^2_{0,\Gamma_S})^{\frac{1}{2}}+(\sum_{K\in\mathcal T_h}h_K^{-2}\|\Psi\|_{0,K}^2)^{\frac{1}{2}}\right)\\
  &\lesssim \left(\|\bm \sigma - \bm \sigma_h\|_{\mathbb C^{-1}}+\operatorname{osc}_{1}(w_b,g_b,\Gamma_C)+\operatorname{osc}_{2}(w_b,\Gamma_S)\right) \eta_1.
    \end{align*}
    Thus $\eta_1\lesssim \|\bm \sigma - \bm \sigma_h\|_{\mathbb C^{-1}}+\operatorname{osc}_{1}(w_b,g_b,\Gamma_C)+\operatorname{osc}_{2}(w_b,\Gamma_S)$
follows.
\end{proof}

\subsection{A posteriori error estimation for deflection}
This subsection explores the a posterior error estimator for the postprocessed deflection. 

Note that $w_h\in U_h$ solves the discrete mixed variational problem \eqref{mixed:bd:dis} and $ U_h$ consists of piecewise $P_{k-2}$ for $k\geq 3$. Define 
\ben
W_{h}^{\ast}:= \{v\in L^{2}(\Omega): v|_{K}\in P_{k+2}(K)~~\text{for all}~~K\in\mathcal{T}_{h}\}. 
\een
On each $K\in \mathcal{T}_{h}$, let $\mathcal{Q}_{K}=\mathcal{Q}_{h}|_{K}$. Let $w_{h}^{\ast}\in W_{h}^{\ast}$ be a solution to 
\be\label{post:pro:1}
\begin{aligned}
\mathcal{Q}_{K}w_{h}^{\ast} &= w_{h},\\
(\nabla^{2}w_{h}^{\ast},\nabla^{2} q)_{K}&=(\mathbb C^{-1}\bsi_{h}, \nabla^{2}q)_{K}\quad \text{for all}~~q\in (I - \mathcal{Q}_{K})W_{h}^{\ast}|_{K},
\end{aligned}
\ee
where
\ben
(I - \mathcal{Q}_{K})W_{h}^{\ast}|_{K}= \{q \in P_{k+2}(K): (q, v)_{K}=0 \quad \text{for all}~~v\in P_{k-2}(K)\}.
\een
Actually, the projection of $\mathbb C^{-1}\bsi_{h}$ on the Hilbert space $(I - \mathcal{Q}_{K})W_{h}^{\ast}|_{K}$  in $H^{2}$ semi-inner product is computed in \eqref{post:pro:1}, where $w_{h}$ is used to impose the constraint.  Thus the local $H^{2}$ projection is well-defined.

Define a mesh-dependent norm
\begin{align*}
    |v|_{2,h}^2:= \sum_{K\in\mathcal T_h}|v|_{2,K}^2+\sum_{e\in\mathcal E_h(\Omega)\cup \mathcal E_h(\Gamma_C)}(h_e^{-3}\|[\![v]\!]\|_{0,e}^2+h_e^{-1}\|[\![\partial_{n_e} v]\!]\|_{0,e}^2)+\sum_{e\in\mathcal E_h(\Gamma_S)}h_e^{-3}\|[\![v]\!]\|_{0,e}^2.
\end{align*}
Proceeding as the proof in \cite[Lemma 3.4]{hu2021family}, one can obtain the following discrete inf-sup condition: 
\begin{equation}
    \label{discrete:inf:sup:post}
    |v_h|_{2,h}\lesssim \sup_{\bm \tau_h\in \Sigma_{h,0}\atop \bm \tau_h\neq 0}\frac{(\operatorname{div}\operatorname{div}\bm \tau_h,v_h)}{\|\bm \tau_h\|_0} \quad \text{for all}~~v_h\in U_h.
\end{equation}

Besides, the definition of the $L^2$-projection operator $\mathcal Q_h$ shows 
\begin{equation*}
   (v-\mathcal Q_h v,q)=0\quad\text{for all}\quad q\in U_h.
\end{equation*}
This ensures that $\|\nabla_h^2(v-\mathcal Q_hv)\|_0$ is a norm with piecewise-defined hessian operator $\nabla^2_h$. Therefore, 
\begin{equation}
    \label{2:h:norm}
    |v-\mathcal Q_h v|_{2,h}\lesssim \|\nabla^2_h(v-\mathcal Q_h v)\|_0.
\end{equation}

\begin{theorem}
It holds 
\begin{equation}\label{relia:uhstar:th}
    \|\bm \sigma-\bm \sigma_h\|_{\mathbb C^{-1}}+|w-w_h^\ast|_{2,h}+\operatorname{osc}+\operatorname{osc}_1+\operatorname{osc}_2\cong \eta + \|\mathbb C^{-1}\bm \sigma_h-\nabla_h^2 w_h^\ast\|_0+\operatorname{osc}_1+\operatorname{osc}_2\,.
\end{equation}
\end{theorem}
\begin{proof}

The discrete inf-sup condition \eqref{discrete:inf:sup:post} and the definition of $w_h^\ast$ in \eqref{post:pro:1} together with the the variational formulations \eqref{mixed:bd:con} and \eqref{mixed:bd:dis} lead to
\begin{equation}
    \label{relia:uhstar:1}
    \begin{aligned}
      & |\mathcal Q_h(w-w_h^\ast)|_{2,h} \lesssim \sup_{\bm \tau_h\in \Sigma_{h,0}\atop \bm \tau_h\neq 0}\frac{(\operatorname{div}\operatorname{div}\bm \tau_h,\mathcal Q_h(w-w_h^\ast))}{\|\bm \tau_h\|_0}=  \sup_{\bm \tau_h\in \Sigma_{h,0}\atop \bm \tau_h\neq 0}\frac{(\operatorname{div}\operatorname{div}\bm \tau_h,w-w_h^\ast)}{\|\bm \tau_h\|_0}\\
    &= \sup_{\bm \tau_h\in \Sigma_{h,0}\atop \bm \tau_h\neq 0}\frac{(\operatorname{div}\operatorname{div}\bm \tau_h,w-w_h)}{\|\bm \tau_h\|_0} = \sup_{\bm \tau_h\in \Sigma_{h,0}\atop \bm \tau_h\neq 0}\frac{(\mathbb C^{-1}(\bm \sigma-\bm \sigma_h),\bm \tau_h)}{\|\bm \tau_h\|_0} \lesssim \|\bm \sigma-\bm \sigma_h\|_{\mathbb C^{-1}}.
    \end{aligned}    
\end{equation}
The norm equivalence \eqref{2:h:norm}, the triangle inequality, and \eqref{relia:uhstar:1} give rise to 
\begin{equation}
    \label{relia:uhstar:2}
    \begin{aligned}
         |w-w_h^\ast-\mathcal Q_{h}(w-w_h^\ast)|_{2,h}&\lesssim \|\nabla_h^2(w-w_h^\ast-\mathcal Q_{h}(w-w_h^\ast))\|_0\\
         &\lesssim \|\nabla_h^2(w-w_h^\ast)\|_0+|\mathcal Q_{h}(w-w_h^\ast)|_{2,h}\\
         &= \|\mathbb C^{-1}\bm \sigma-\nabla_h^2 w_h^\ast\|_0 +|\mathcal Q_{h}(w-w_h^\ast)|_{2,h}\\
         &\lesssim \|\mathbb C^{-1}(\bm \sigma-\bm \sigma_h)\|_0+\|\mathbb C^{-1}\bm \sigma_h - \nabla_h^2 w_h^\ast\|_0.
    \end{aligned}
\end{equation}
The triangle inequality shows
\begin{equation}\label{2:h:tri}
 \begin{aligned}
        |w-w_h^\ast|_{2,h}\leq |w-w_h^\ast-\mathcal Q_{h}(w-w_h^\ast)|_{2,h}+|\mathcal Q_h(w-w_h^\ast)|_{2,h}.
    \end{aligned}    
\end{equation}
This and \eqref{relia:uhstar:1}--\eqref{relia:uhstar:2} lead to
 \begin{align*}
        |w-w_h^\ast|_{2,h}&\lesssim \|\mathbb C^{-1}(\bm \sigma-\bm \sigma_h)\|_0+\|\mathbb C^{-1}\bm \sigma_h-\nabla^2_h w_h^\ast\|_0\\
        &\lesssim \|\bm \sigma-\bm \sigma_h\|_{\mathbb C^{-1}}+\|\mathbb C^{-1}\bm \sigma_h-\nabla^2_h w_h^\ast\|_0.
    \end{align*}
This plus Theorem \ref{relia:bd:upper} results in $$\|\bm \sigma-\bm \sigma_h\|_{\mathbb C^{-1}}+|w-w_h^\ast|_{2,h}\lesssim \eta + \|\mathbb C^{-1}\bm \sigma_h-\nabla_h^2 w_h^\ast\|_0.$$

On the other hand, by the triangle inequality, one can obtain
\begin{equation}
    \label{effi:uhstar}
    \begin{aligned}
        \|\mathbb C^{-1}\bm \sigma_h-\nabla_h^2 w_h^\ast\|_0&\leq \|\mathbb C^{-1}(\bm \sigma-\bm \sigma_h)\|_0+\|\mathbb C^{-1}\bm \sigma-\nabla^2_h w_h^\ast\|_0\\
        &\lesssim\|\bm \sigma-\bm \sigma_h\|_{\mathbb C^{-1}}+\|\nabla^2 w-\nabla^2_h w_h^\ast\|_0\\
        &\lesssim\|\bm \sigma-\bm \sigma_h\|_{\mathbb C^{-1}}+|w-w_h^\ast|_{2,h}.
    \end{aligned}
\end{equation}
This and the upper bound \eqref{efficiency:1} yield  $$\|\bm \sigma-\bm \sigma_h\|_{\mathbb C^{-1}}+|w-w_h^\ast|_{2,h}+\operatorname{osc}+\operatorname{osc}_1+\operatorname{osc}_2\gtrsim \eta + \|\mathbb C^{-1}\bm \sigma_h-\nabla_h^2 w_h^\ast\|_0.$$
This concludes the proof.
\end{proof}

\section{Optimal convergence analysis}
Let $\mathcal T_0$ be an initial shape-regular triangulation of $\Omega$ into triangles. Let $\mathbb T: = \mathbb T(\mathcal T_0)$ denote the set of all admissible regular triangulations obtained from $\mathcal T_0$ through a finite number of successive newest vertex bisections (NVB)\cite{stevenson2008completion}. Let $\mathcal T_1$, $\mathcal T_2$, ..., $\mathcal T_N$ represent the admissible meshes obtained by successive refinement of $\mathcal T_0$, and 
let $\mathcal T_h$ denote a refined mesh of $\mathcal T_H\in \mathbb T$.

The discrete bending moment space $\Sigma_h$ is non-nested because the space defined on the coarse mesh $\mathcal T_H$ is not necessarily a subspace of the corresponding space on the fine mesh $\mathcal T_h$, due to extra $C^0$ vertex continuity imposed on functions in $\Sigma_h$. This causes essential difficulties in demonstrating some orthogonality or quasi-orthogonality which plays an important role in the optimal convergence analysis of adaptive algorithms. 
 
This section first relaxes $C^0$ vertex continuity of functions in $\Sigma_h$ and introduces an extended bending moment space $\widetilde \Sigma_h$, which is nested. 
Based on these nested spaces, the optimal convergence of the adaptive algorithm is established through the unified analysis from \cite{hu2018unified}. For simplicity, this section focuses on the homogeneous clamped boundary condition $\Gamma_C=\partial\Omega$ with $w_b \equiv 0$ and $g_b\equiv 0$. The corresponding mixed formulation then seeks $(\bm \sigma, w)\in H(\operatorname{div}\operatorname{div},\Omega;\mathbb S)\times L^2(\Omega)$ such that 
\begin{equation}\label{clamped:bd:con}
 \begin{aligned}
    (\mathbb C^{-1}\bm \sigma, \bm \tau) -(\operatorname{div}\operatorname{div}\bm \tau, w) &= 0\quad&&\text{for all}~~\bm \tau\in H(\operatorname{div}\operatorname{div},\Omega;\mathbb S),\\
    (\operatorname{div}\operatorname{div}\bm \sigma, v) &= (f,v)\quad&& \text{for all}~~v\in L^2(\Omega).
\end{aligned}   
\end{equation}
On each element $K\in\mathcal T_h$, the posteriori error estimator $\eta$ from \eqref{estimator:global} reads:
\begin{equation}\label{estimator:on:K}
    \eta^2(\mathcal T_h, K):= h_K^2\|\operatorname{rot}(\mathbb C^{-1}\bm \sigma_h)\|_{0,K}^2 + \sum_{e\in \mathcal E(K)}h_K\|[\![\mathbb C^{-1}\bm \sigma_h\mathbf t_e]\!]\|_{0,e}^2+h_K^4\|f-\mathcal Q_h f\|_{0,K}^2.
\end{equation}
Let 
\begin{align*}
    &\eta^2(\mathcal T_h, \mathcal M): = \sum_{K\in\mathcal M}\eta^2(\mathcal T_h, K)\quad\text{for all}\, \, \mathcal M \subset \mathcal T_h,\\
    &\eta^2(\mathcal T_h): = \sum_{K\in\mathcal T_h}\eta^2(\mathcal T_h, K).
\end{align*}

%%%%%%%%%%%%%%%%%%%%%%%%%%%%%%%%%%%%%%%%%%%%%%%%%%%%%%%%%%%%%%%%%%%%%%
%%%%%%%%%%%%%%%%%%%%%%%%%%%%%%%%%%%%%%%%%%%%%%%%%%%%%%%%%%%%%%%%%%%%%%

 \subsection{Extended bending moment space on adaptive meshes}\label{relax:vertex}To make the bending moment space $\Sigma_h$ hierarchical on an admissible triangulation $\mathcal T_h\in \mathbb T$, an extended bending moment space $\widetilde{\Sigma}_h$ will be defined in this subsection. 

Following the notation for $\mathcal T_h$, where $\mathcal E_h$ and $\mathcal V_h$ denote the sets of all edges and vertices, respectively, the subscript is replaced with $H$ for the coarse mesh $\mathcal T_H$. Accordingly, $\mathcal E_H$ and $\mathcal V_H$ represent the edges and vertices associated with $\mathcal T_H$, respectively. Similarly, let $\mathcal E_h(\Omega)$ and $\mathcal V_h(\Omega)$ denote the sets of all interior edges and vertices of $\mathcal T_h$, and define $\mathcal E_H(\Omega)$ and $\mathcal V_H(\Omega)$ analogously for $\mathcal T_H$.

For $e\in\mathcal E_h$,  $\mathbf n_e$ denotes the unit normal vector and $\mathbf t_e=\mathbf n_e^\perp$ is the unit tangential vector. Let $\mathcal V_0$ denote the set of all vertices on the initial triangulation $\mathcal T_0$. The newest vertex bisection creates each new vertex $\mathbf x_e\in \mathcal V_h\backslash \mathcal V_0$ as a midpoint of some old edge $e$ associated with tangential vector $\mathbf t_e$ and normal vector $\mathbf n_e$. For $\mathbf x_e\in\mathcal V_h\backslash \mathcal V_0$, define two patches $\omega_{\mathbf x_e}^+$ and $\omega_{\mathbf x_e}^-$ by 
\begin{equation}
    \begin{aligned}
        \omega_{\mathbf x_e}^+:=\cup \{K\in{\mathcal T_h}:\mathbf x_e\in K,\, (\operatorname{mid}(K)-\mathbf x_e)\cdot \mathbf n_e>0\},\\
        \omega_{\mathbf x_e}^-:=\cup \{K\in{\mathcal T_h}:\mathbf x_e\in K,\, (\operatorname{mid}(K)-\mathbf x_e)\cdot \mathbf n_e<0\},
    \end{aligned}
\end{equation}
when it comes to any boundary edge $e\subset \partial \Omega$, $\omega_{\mathbf x_e}^+$ is an empty set.

Unlike $\bm \tau\in \Sigma_h$ that all components of $\bm \tau$ at $\mathbf x_e$ are continuous, the continuity of the pure tangential component of $\widetilde{\bm \tau}\in \widetilde{\Sigma}_h$ is relaxed, namely, the values of $\mathbf t_e^{\mathrm T}\widetilde{\bm \tau}\mathbf t_e$ on $\omega_{\mathbf x_e}^+$ and $\omega_{\mathbf x_e}^-$ are not necessarily the same. In particular, the extended bending moment space is defined by
\begin{equation*}
\begin{aligned}
 \widetilde{\Sigma}_h:=\{\bta\in H(\operatorname{div}\operatorname{div}, \Omega; \mathbb{S}): \bta|_{K}\in P_{k}(K;\ms)\, \text{for all}\, K\in\mathcal T_h,\\
 \bm \tau(\mathbf x)\, \, \text{is continuous at all }\mathbf x\in\mathcal V_0,\\
 \mathbf t_e^{\mathrm T}\bm \tau\mathbf t_e(\mathbf x_e)\, \, \text{is continuous}\, \, \text{at}\, \mathbf x_e\, \text{in}\,  \omega_{\mathbf x_e}^+\, \, \text{and at}\, \, \mathbf x_e\, \text{in}\, \omega_{\mathbf x_e}^-\, \, \text{for all}\, \, \mathbf x_e\in \mathcal V_h\backslash \mathcal V_0,\\
 \bm \tau\mathbf n_e(\mathbf x_e)\, \, \text{is continuous at }\mathbf x_e\in \mathcal V_h\backslash \mathcal V_0,\\
[\![\bta\mathbf{n}_e]\!]_{e}=\mathbf{0}\, \text{and}\, [\![\mathbf{n}_e^{{\mathrm T}}\operatorname{div}\bta]\!]_{e}=0\,\text{for all}\,e\in  \mathcal E_h(\Omega)\}.
\end{aligned}
\end{equation*}
The extended bending moment space with respect to $\mathcal T_H$ reads $\widetilde{\Sigma}_H$.

\begin{theorem}[Nestedness]There holds
 \[\widetilde{\Sigma}_H\subset \widetilde{\Sigma}_h.\] 
\end{theorem}
\begin{proof}
    Given any interior edge $e\in \mathcal E_H(\Omega)$ shared by two triangles $K_j\in\mathcal T_H$, $j = 1,2$, the bisection of $e$ induces a refinement of $K_j$, leading to four new global degrees of freedom at the midpoint $\mathbf x_e: = \operatorname{mid}(e)\in \mathcal V_h(\Omega)\backslash \mathcal V_H$. Observe that $\omega_{\mathbf x_e}^+$ and $\omega_{\mathbf x_e}^-$ are contained in $K_1$ and $K_2$, respectively.  
    For $\bm \tau\in \widetilde{\Sigma}_H$, it holds that $\bm \tau|_{K_j}\in P_{k}(K_j;\mathbb S)$, and the normal components $\bm \tau\mathbf n_e(\mathbf x_e)$ remain globally continuous at $\mathbf x_e$. As $\bm \tau$ is continuous within each triangle $K_1$ and $K_2$, the pure tangential component $\mathbf t_e^{\mathrm T}\bm \tau\mathbf t_e$ remains continuous within $\omega_{\mathbf x_e}^+$ and $\omega_{\mathbf x_e}^-$, respectively. Thus $\bm \tau\in \widetilde{\Sigma}_h$ follows. This concludes the proof. 
\end{proof}

The extended mixed problem with the homogeneous clamped boundary condition seeks $(\bm \sigma_h, w_h)\in \widetilde{\Sigma}_h\times U_h$ such that 
\begin{equation}\label{clamped:bd:dis}
 \begin{aligned}
    (\mathbb C^{-1}\bm \sigma_h, \bm \tau_h) -(\operatorname{div}\operatorname{div}\bm \tau_h, w_h) &= 0\quad&&\text{for all}~~\bm \tau_h\in \widetilde{\Sigma}_h,\\
    (\operatorname{div}\operatorname{div}\bm \sigma_h, v_h) &= (f,v_h)\quad&& \text{for all}~~v_h\in U_h.
\end{aligned}   
\end{equation}

%%%%%%%%%%%%%%%%%%%%%%%%%%%%%%%%%%%%%%%%%%%%%%%%%%%%%%%%%%%%%%%%%%%%%%
\subsection{Quasi-orthogonality} 
One of the main task of optimality analysis is to prove the quasi-orthogonality. To this end, the following intermediate problem is introduced: find $(\underline{\bm \sigma}_h, \underline w_h)\in \widetilde{\Sigma}_h\times U_h$ such that 
\begin{equation}
    \label{intermediate:coarse:fine}
    \begin{aligned}
        (\mathbb C^{-1}\underline{\bm \sigma}_h \bm \tau_h)-(\operatorname{div}\operatorname{div}\bm \tau_h, \underline w_h)& = 0 \quad &&\text{for all}\, \, \bm \tau_h \in \widetilde{\Sigma}_h,\\
        (\operatorname{div}\operatorname{div}\underline{\bm \sigma}_h, v_h)& = (\mathcal Q_{H}f, v_h)\quad &&\text{for all}\, \, v_h\in U_h.
    \end{aligned}
\end{equation}
Here $\mathcal Q_{H}$ denotes the $L^2$ projection onto $U_H$. Let 
\[\operatorname{osc}(f,\mathcal T_H\backslash\mathcal T_h): =\left(\sum_{K\in {\mathcal T_H}\backslash \mathcal T_h}h_K^{4}\|f-\mathcal Q_{H}f\|_{0,K}^2\right)^{\frac{1}{2}}. \]
\begin{lemma}\label{quasi-orth:pre}
For $(\bm \sigma_h, w_h)\in \widetilde{\Sigma}_h\times U_h$ solving \eqref{clamped:bd:dis} and $(\underline{\bm \sigma}_h, \underline{w}_h)\in \widetilde{\Sigma}_h\times U_h$ solving \eqref{intermediate:coarse:fine}, it holds
\begin{equation*}
    \|\bm \sigma_h-\underline{\bm \sigma}_h\|_{\mathbb C^{-1}}+|w_h-\underline{w}_h|_{2,h}\lesssim\operatorname{osc}(f, \mathcal T_H\backslash\mathcal T_h).
\end{equation*}
\end{lemma}
\begin{proof}
    Let $\Phi: = \bm \sigma_h-\underline{\bm \sigma}_h$ and $\psi: = w_h-\underline{w}_h$. Note that $(f, v_h) = (\mathcal Q_{h}f,  v_h)$ for all $v_h\in U_h$. Subtracting \eqref{intermediate:coarse:fine} from \eqref{clamped:bd:dis} leads to
    \begin{equation}\label{inter:minus}
        \begin{aligned}
      (\mathbb C^{-1}\Phi,\bm \tau_h)-(\operatorname{div}\operatorname{div}\bm \tau_h, \psi)& = 0 \quad &&\text{for all}\, \, \bm \tau_h \in \widetilde{\Sigma}_h,\\
        (\operatorname{div}\operatorname{div}\Phi, v_h)& =( \mathcal Q_{h}f-\mathcal Q_{H}f, v_h)\quad &&\text{for all}\, \, v_h\in U_h. 
    \end{aligned}
    \end{equation}
    Using the stability result as in \cite[(3.11)]{hu2021family}, one can obtain 
    \begin{equation}\label{inf-sup:dis:norm}
        \begin{aligned}
        \|\Phi\|_{\mathbb C^{-1}}+|\psi|_{2,h}\lesssim\sup_{0\neq\bm \tau_h\in \widetilde{\Sigma}_h\atop 0\neq v_h\in U_h}\frac{(\mathbb C^{-1}\Phi, \bm \tau_h)-(\operatorname{div}\operatorname{div}\bm \tau_h, \psi)+(\operatorname{div}\operatorname{div}\Phi, v_h)}{\|\bm \tau_h\|_{0}+|v_h|_{2,h}}.
    \end{aligned}
    \end{equation}
Given any $v_h\in U_h$, for any $K\in\mathcal T_H$, by \cite[Chapter 10.6]{brennerscott}, there holds 
\begin{equation*}
    \|v_h-\mathcal Q_H v_h\|_{0,K}\lesssim h_K^2\left(|v_h|^2_{2,K}+\sum_{e\in\mathcal E_h, e\subset K}\left(h_e^{-3}\|[\![v_h]\!]\|^2_{0,e}+h_e^{-1}\|[\![\partial_{n_e}v_h]\!]\|^2_{0,e}\right)\right)^\frac{1}{2}.
\end{equation*}
This, \eqref{inter:minus}--\eqref{inf-sup:dis:norm}, and the Cauchy-Schwarz inequality conclude
\begin{align*}
    \|\Phi\|_{\mathbb C^{-1}}+|\psi|_{2,h}&\lesssim\sup_{0\neq\bm \tau_h\in \widetilde{\Sigma}_h\atop 0\neq v_h\in U_h}\frac{(\mathcal Q_{h}f-\mathcal Q_{H}f, v_h)}{\|\bm \tau_h\|_{0}+|v_h|_{2,h}}\\&=\sup_{0\neq\bm \tau_h\in \widetilde{\Sigma}_h\atop 0\neq v_h\in U_h}\frac{(f-\mathcal Q_{H}f, v_h-\mathcal Q_H v_h)}{\|\bm \tau_h\|_{0}+|v_h|_{2,h}}\lesssim\operatorname{osc}(f,\mathcal T_H\backslash\mathcal T_h).
\end{align*}   
\end{proof}

\begin{theorem}[Quasi-orthogonality]\label{quasi:orthogonality}
    For any $0<\delta<1$, there exists some constant $\mathcal C>0$ such that 
    \begin{equation*}
        (1-\delta)\|\bm \sigma-\bm \sigma_h\|_{\mathbb C^{-1}}^2\leq \|\bm \sigma-\bm \sigma_H\|_{\mathbb C^{-1}}^2-\|\bm \sigma_h-\bm \sigma_H\|_{\mathbb C^{-1}}^{2}+\frac{\mathcal C}{\delta}\operatorname{osc}^2(f, \mathcal T_H\backslash {\mathcal T_h}).
    \end{equation*}
\end{theorem}

\begin{proof}
The second equation in \eqref{intermediate:coarse:fine} shows that $ \operatorname{div}\operatorname{div}(\underline{\bm \sigma}_h-\bm \sigma_H)=0$. 
Accordingly, assigning $\underline{\bm \sigma}_h-\bm \sigma_H$ to $\bm \tau_h$ in \eqref{mixed:bd:con} and \eqref{clamped:bd:dis} respectively leads to $(\mathbb C^{-1}\bm \sigma, \underline{\bm \sigma}_h-\bm \sigma_H)=0$ and $(\mathbb C^{-1}\bm \sigma_h, \underline{\bm \sigma}_h-\bm \sigma_H)=0$. Thus 
\begin{align*}
    (\mathbb C^{-1}(\bm \sigma-\bm \sigma_h), {\bm \sigma}_h-\bm \sigma_H)&= (\mathbb C^{-1}(\bm \sigma-\bm \sigma_h), {\bm \sigma}_h-\underline{\bm \sigma}_h+\underline{\bm \sigma}_h-\bm \sigma_H)\\
    &=(\mathbb C^{-1}(\bm \sigma-\bm \sigma_h), {\bm \sigma}_h-\underline{\bm \sigma}_h).
\end{align*}
This and Lemma \ref{quasi-orth:pre} give rise to 
\begin{equation*}
    (\mathbb C^{-1}(\bm \sigma-\bm \sigma_h), {\bm \sigma}_h-\bm \sigma_H)\leq \sqrt{\mathcal C}\|\bm \sigma - \bm \sigma_h\|_{\mathbb C^{-1}}\operatorname{osc}(f, \mathcal T_H\backslash\mathcal T_h).
\end{equation*}
This combined with Young's inequality concludes the proof. 
\end{proof}
%%%%%%%%%%%%%%%%%%%%%%%%%%%%%%%%%%%%%%%%%%%%%%%%%%%%%%%%%%%%%%%%%
\subsection{Discrete reliability} 
Thanks to the following extended $H^1$-conforming finite element space, the discrete reliability can be established, which is another main task of optimality analysis. Recall the subdomains $\omega_{\mathbf x_e}^+$ and $\omega_{\mathbf x_e}^-$. The extended $H^1$-conforming space reads
\begin{align*}
    \widetilde{V}_h=\left\{  \bm \phi \in H^1(\operatorname{div},\Omega;\mathbb R^2): \,  \bm \phi|_K\in P_{k+1}(K;\mathbb R^2)\, \, \text{for all}\, \, K\in \mathcal T_h,\right.\\
     \bm \phi(\mathbf x) \, \, \text{is continuous at all vertices }\mathbf x\in \mathcal V_h,\\
    \nabla  \bm \phi(\mathbf x) \, \, \text{is continuous at each initial vertex}\, \, \mathbf x\in \mathcal V_0, \\
\text{and each boundary vertex}\, \, \mathbf x\in\mathcal V_h\backslash\mathcal V_h(\Omega), \\
\mathbf n_e^{\mathrm T}\nabla  \bm \phi \mathbf  n_e,  \mathbf n_e^{\mathrm T}\nabla  \bm \phi\mathbf t_e, \, \, \text{and}\, \, \mathbf t_e^{\mathrm T}\nabla  \bm \phi \mathbf t_e\, \, \text{are continuous at each internal vertex}\, \, \mathbf x_e\in \mathcal V_h(\Omega)\backslash \mathcal V_0, \\
\left.\mathbf t_e^{\mathrm T}\nabla  \bm \phi \mathbf n_e\, \, \text{is continuous at}\, \, \mathbf x_e\, \text{in}\, \, \omega_{\mathbf x_e}^+ \, \, \text{and at}\, \,\mathbf x_e\, \text{in}\, \, \omega_{\mathbf x_e}^- \, \, \text{for all}\, \, \mathbf x_e\in \mathcal V_h(\Omega)\backslash\mathcal V_0\right\}.
\end{align*}
\begin{lemma}\label{dim:sigma:V:extended}
    The dimensions of $\widetilde{\Sigma}_h$ and $\widetilde{V}_h$ are
    \begin{equation*}
        \operatorname{dim}\widetilde{\Sigma}_h = \operatorname{dim}\Sigma_h+\#(\mathcal V_h(\Omega)\backslash\mathcal V_0), \quad \operatorname{dim}\widetilde{V}_h= \operatorname{dim}V_h+\# (\mathcal V_h(\Omega)\backslash\mathcal V_0).
    \end{equation*}
\end{lemma}
\begin{proof}
Compared to the space ${\Sigma}_h$, where each vertex $\mathbf x_e\in \mathcal V_h(\Omega)\backslash\mathcal V_0$ has three global degrees of freedom, the extended space $\widetilde{\Sigma}_h$ increases this number to four. Assuming the basis functions associated to $\mathbf x_e$ in ${\Sigma}_h$ are $\bm \tau_{\mathbf x_e,1}, \bm \tau_{\mathbf x_e,2}$, and $\bm \tau_{\mathbf x_e,3}$ with $\mathbf t_e^{\mathrm T}\bm \tau_{\mathbf x_e,3}(\mathbf x_e)\mathbf t_e = 1$ and $\bm \tau_{\mathbf x_e,3}(\mathbf x_e)\mathbf n_e=0$. The corresponding basis functions of $\mathbf x_e$ in the extended space $\widetilde{\Sigma}_h$ are given by 
\begin{equation*}
    \bm \tau_{\mathbf x_e,1}(\mathbf x), \, \, \bm \tau_{\mathbf x_e,2}(\mathbf x), \, \, \bm \tau_{\mathbf x_e,3}^+(\mathbf x), \, \, \bm \tau_{\mathbf x_e,3}^-(\mathbf x),
\end{equation*}
where $\bm \tau_{\mathbf x_e,3}^+(\mathbf x)=\bm \tau_{\mathbf x_e,3}(\mathbf x)$ for $\mathbf x\in \omega_{\mathbf x}^+$ and otherwise vanishes, and $\bm \tau_{\mathbf x_e,3}^-(\mathbf x)$ is similarly defined for $\omega_{\mathbf x}^-$. Therefore, \[\widetilde{\Sigma}_h= \Sigma_h + \mathop{\operatorname{span}}\limits_{\mathbf x_e\in \mathcal V_h(\Omega)\backslash\mathcal V_0}\{\bm \tau_{\mathbf x_e,3}^+, \bm \tau_{\mathbf x_e,3}^-\},\] and thus $\operatorname{dim}\widetilde{\Sigma}_h = \operatorname{dim}\Sigma_h+\#(\mathcal V_h(\Omega)\backslash\mathcal V_0)$ follows. 

Similarly, the global degrees of freedom with respect to $\nabla \bm \phi$ at each internal vertex $\mathbf x_e\in \mathcal V_h(\Omega)\backslash\mathcal V_0$ increase from four in $V_h$ to five in $\widetilde{V}_h$. Therefore,  $\operatorname{dim}\widetilde{V}_h= \operatorname{dim}V_h+\# (\mathcal V_h(\Omega)\backslash\mathcal V_0)$.
\end{proof}
With Lemma \ref{dim:sigma:V:extended}, the following lemma establishes their relationship. 
\begin{lemma}\label{extend:hemholtz}
    It holds that  $\operatorname{sym}\operatorname{curl}\widetilde{V}_h = \widetilde{\Sigma}_h\cap  \operatorname{ker}(\operatorname{div}\operatorname{div})$.
 \end{lemma}
\begin{proof}
 For $\bm \phi_h\in \widetilde{V}_h$, the identities $\mathbf n_e^{\mathrm T}(\operatorname{sym}\operatorname{curl}\bm \phi_h)\mathbf n_e=\mathbf n_e^{\mathrm T} \nabla\bm \phi_h \mathbf t_e$,
  $\mathbf t_e^{\mathrm T}(\operatorname{sym}\operatorname{curl}\bm \phi_h)\mathbf n_e = \mathbf t_e^{\mathrm T} \nabla\bm \phi_h \mathbf t_e-\mathbf n_e^{\mathrm T} \nabla\bm \phi_h \mathbf n_e$, and $\mathbf t_e^{\mathrm T}(\operatorname{sym}\operatorname{curl}\bm \phi_h)\mathbf t_e=\mathbf t_e^{\mathrm T} \nabla\bm \phi_h \mathbf n_e$ show that 
 $(\operatorname{sym}\operatorname{curl}\bm \phi_h)\mathbf n_e$ is continuous at internal vertices $\mathbf x_e\in \mathcal V_h(\Omega)\backslash\mathcal V_0$ and $ \mathbf t_e^{\mathrm T}(\operatorname{sym}\operatorname{curl}\bm \phi_h)\mathbf t_e$ is continuous at internal vertices $\mathbf x_e$ in $\omega_{\mathbf x_e}^+$ and at $\mathbf x_e$ in $\omega_{\mathbf x_e}^-$ for all $\mathbf x_e\in \mathcal V_h(\Omega)\backslash\mathcal V_0$. This plus the exactness established in Lemma \ref{exact:divdiv} yields $\operatorname{sym}\operatorname{curl}\widetilde{V}_h\subset \widetilde{\Sigma}_h\cap  \operatorname{ker}(\operatorname{div}\operatorname{div})$.

On the other hand, counting dimensions with Lemma \ref{dim:sigma:V:extended} gives rise to
\begin{align*}
    \operatorname{dim} \widetilde{\Sigma}_h\cap \operatorname{ker}(\operatorname{div}\operatorname{div}) = \operatorname{dim} \widetilde{\Sigma}_h-\operatorname{dim}U_h=\operatorname{dim} {\Sigma}_h+\# (\mathcal V_h(\Omega)\backslash\mathcal V_0)-\operatorname{dim}U_h,
\end{align*}
and 
\begin{align*}
\operatorname{dim}\operatorname{sym}\operatorname{curl}\widetilde{V}_h = \operatorname{dim} \widetilde{V}_h-\operatorname{dim}RT=
\operatorname{dim} {V}_h+\# (\mathcal V_h(\Omega)\backslash\mathcal V_0)-4.
\end{align*}
Using this and $\operatorname{dim}\Sigma_h-\operatorname{dim}U_h= \operatorname{dim}V_h-4$ derived from the exactness of the discrete divdiv complex in Lemma \ref{exact:divdiv}, one can obtain $\operatorname{dim} \widetilde{\Sigma}_h\cap \operatorname{ker}(\operatorname{div}\operatorname{div})=\operatorname{dim}\operatorname{sym}\operatorname{curl}\widetilde{V}_h$. Thus $\operatorname{sym}\operatorname{curl}\widetilde{V}_h=\widetilde{\Sigma}_h\cap \operatorname{ker}(\operatorname{div}\operatorname{div})$ follows. 
\end{proof}

Analogous to the interpolation operator in \eqref{int:v0-vh0}, a quasi-interpolation operator $\widetilde{I}_H:  \widetilde{V}_h\rightarrow \widetilde{V}_H$ for the extended $H^1$ conforming space is given to estimate $\|\bm \sigma_H-\underline{\bm \sigma}_h\|_{\mathbb C^{-1}}$ as in \cite[Section 3]{carstensen2021hierarchical}.  In particular, the $\kappa_i$ in \eqref{int:v0-vh0} corresponding to the degrees of freedom associated with any vertex or edge contained in some unrefined triangle or edge is chosen as that triangle or edge of the coarse mesh. Besides, for the degree of freedom $\mathbf t_e^{\mathrm T}\nabla \bm \phi\mathbf n_e|_{\omega_{\mathbf x_e}^+}$ at each vertex $\mathbf x_e\in \mathcal V_H(\Omega)\backslash \mathcal V_0$, the $\kappa_i$ is chosen as one triangle $K^+\in \omega_{\mathbf x_e}^+$, and similar for $\omega_{\mathbf x_e}^-$. Note that $\widetilde{I}_H\bm \phi_h|_{K} = \bm \phi_h|_K$ for any $K\in \mathcal T_H\cap \mathcal T_h$. Similar to the estimates for $I_h$ in \eqref{estimate:phi}, the following estimates hold:
\begin{equation}\label{estimate:widetilde:Ih}
    \|\bm \phi_h-\widetilde{I}_H\bm \phi_h\|_{0,K} \lesssim h_K |\bm \phi_h|_{1,S(K)}, \quad \|\bm \phi_h-\widetilde{I}_H\bm \phi_h\|_{0,e} \lesssim h_e^\frac{1}{2}|\bm \phi_h|_{1,S(\omega_e)}.
\end{equation}

With this quasi-interpolation operator $\widetilde{I}_H$, the following lemma can be demonstrated. 
\begin{lemma}\label{sigmaH:minus:sigmahath}
    Let $(\bm \sigma_H, w_H)\in \widetilde{\Sigma}_H\times U_H$ and $(\underline{\bm \sigma}_h, w_h)\in \widetilde{\Sigma}_h\times U_h$ solve \eqref{clamped:bd:dis} and \eqref{intermediate:coarse:fine} respectively. Then it holds 
    \begin{equation*}
        \|\bm \sigma_H-\underline{\bm \sigma}_h\|_{\mathbb C^{-1}}\lesssim \sum_{K\in\mathcal T_H\backslash\mathcal T_h}\left(h_K\|\operatorname{rot}(\mathbb C^{-1}\bm \sigma_H)\|_{0,K}+\sum_{e\in\mathcal E(K)}h_K^{\frac{1}{2}}\|[\![\mathbb C^{-1}\bm \sigma_H\mathbf t_e]\!]\|_{0,e}\right).
    \end{equation*}
\end{lemma}
\begin{proof}
    Let $\Phi: = \bm \sigma_H-\underline{\bm \sigma}_h$. Note that $\Phi\in \widetilde{\Sigma}_h$ and $\operatorname{div}\operatorname{div}\Phi = 0$. Lemma \ref{extend:hemholtz} shows that there exists some $\bm \phi_h\in \widetilde{V}_h$ such that $\Phi= \operatorname{sym}\operatorname{curl}\bm \phi_h$ and $|\bm \phi_h|_{1,K}\lesssim \|\Phi\|_0$. Let $\Theta: = \bm \phi_h-\widetilde{I}_H \bm \phi_h$. Then $\Theta\in \widetilde{V}_h$, and $\Theta|_K = 0$ for all $K\in \mathcal T_H\cap \mathcal T_h$. 
    
    Assigning $\Phi$ to $\bm \tau_h$ in \eqref{intermediate:coarse:fine} results in $(\mathbb C^{-1}\underline{\bm \sigma}_h, \Phi) = 0$. The choice of $\bm \tau_H = \operatorname{sym}\operatorname{curl}(\widetilde{I}_H\bm \phi_h)$ in \eqref{clamped:bd:dis} leads to $(\mathbb C^{-1}\bm \sigma_H, \operatorname{sym}\operatorname{curl} \widetilde{I}_H\bm \phi_h)=0$. Therefore, 
    \begin{align*}
        \|\Phi\|_{\mathbb C^{-1}}^2  = (\mathbb C^{-1}(\bm \sigma_H-\underline{\bm \sigma}_h), \Phi)=  (\mathbb C^{-1}\bm \sigma_H, \Phi)=(\mathbb C^{-1}\bm \sigma_H, \operatorname{sym}\operatorname{curl}\bm \phi_h)\\
        =(\mathbb C^{-1}\bm \sigma_H, \operatorname{sym}\operatorname{curl}\Theta)=\sum_{K\in \mathcal T_H\backslash\mathcal T_h}(\mathbb C^{-1}\bm \sigma_H, \operatorname{sym}\operatorname{curl}\Theta)_K.
    \end{align*}
    An integration by parts, \eqref{estimate:widetilde:Ih}, and the Cauchy-Schwarz inequality lead to 
    \begin{align*}
        \|\Phi\|_{\mathbb C^{-1}}^2 &= \sum_{K\in \mathcal T_H\backslash\mathcal T_h}(\mathbb C^{-1}\bm \sigma_H, \operatorname{sym}\operatorname{curl}\Theta)_K=\sum_{K\in \mathcal T_H\backslash\mathcal T_h}(\mathbb C^{-1}\bm \sigma_H, \operatorname{curl}\Theta)_K\\
        &=\sum_{K\in \mathcal T_H\backslash\mathcal T_h}\left((\operatorname{rot}(\mathbb C^{-1}\bm \sigma_H), \Theta)_K-\langle\mathbb C^{-1}\bm \sigma_H\mathbf t, \Theta\rangle_{\partial K}\right)\\
        &\lesssim \sum_{K\in \mathcal T_H\backslash\mathcal T_h}\left(h_K\|\operatorname{rot}(\mathbb C^{-1}\bm \sigma_H)\|_{0,K}|\bm \phi_h|_{1,S(K)}+\sum_{e\subset \partial K}h_K^{\frac{1}{2}}\|[\![\mathbb C^{-1}\bm \sigma_H\mathbf t_e]\!]\|_{0,e}|\bm \phi_h|_{1, S(\omega_e)}\right)\\
        &\lesssim  \sum_{K\in\mathcal T_H\backslash\mathcal T_h}\left(h_K\|\operatorname{rot}(\mathbb C^{-1}\bm \sigma_H)\|_{0,K}+\sum_{e\in\mathcal E(K)}h_K^{\frac{1}{2}}\|[\![\mathbb C^{-1}\bm \sigma_H\mathbf t_e]\!]\|_{0,e}\right)\|\Phi\|_{\mathbb C^{-1}}.
    \end{align*}
Removing $\|\Phi\|_{\mathbb C^{-1}}$ from both sides of this inequality concludes the proof. 
\end{proof}

\begin{theorem}[Discrete Reliability]\label{discrete:reliability}
    There holds
    \begin{equation*}
     \|\bm \sigma_h- \bm \sigma_H\|^2_{\mathbb C^{-1}} \lesssim \eta^2(\mathcal T_H, \mathcal T_H\backslash \mathcal T_h).
    \end{equation*}
\end{theorem}
\begin{proof}
Lemmas \ref{quasi-orth:pre} and \ref{sigmaH:minus:sigmahath} together with a triangle inequality yields the result. 
\end{proof}
%%%%%%%%%%%%%%%%%%%%%%%%%%%%%%%%%%%%%%%%%%%%%%%%%%%%%%%%%%%%%%%%%%%%%%

To distinguish the specific level in the successive refinements, let $\widetilde{\Sigma}(\mathcal T_l)$ and $U(\mathcal T_l)$ denote the finite element spaces on $\mathcal T_l$, and let $(\bm \sigma(\mathcal T_l), w(\mathcal T_l))\in \widetilde{\Sigma}(\mathcal T_l)\times U(\mathcal T_l)$ represent the corresponding discrete solutions. 
\begin{lemma}[Estimator reduction]\label{estimator:reduction}
  Given any positive constant $\epsilon$, there exists $\lambda: = 1-2^{-1/2}<1$ and $\mathcal C_{\epsilon}>0$ such that 
    \begin{equation*}
        \eta^2(\mathcal T_l) \leq (1+\epsilon)\left(\eta^2(\mathcal T_{l-1})-\lambda \eta^2(\mathcal T_{l-1}, \mathcal M_{l-1})\right)+\mathcal C_{\epsilon}\|\bm \sigma(\mathcal T_l)-\bm \sigma(\mathcal T_{l-1})\|_{\mathbb C^{-1}}^2
    \end{equation*}
    and 
    \begin{equation*}
        \operatorname{osc}^2(f, \mathcal T_l) \leq \operatorname{osc}^2(f, \mathcal T_{l-1}) -\lambda \operatorname{osc}^2(f, \mathcal T_{l-1}\backslash\mathcal T_l).
    \end{equation*}
\end{lemma}
\begin{proof}
   Proceeding as in \cite[Corollary 3.4]{cascon2008quasi} yields the results. 
\end{proof}
 Given $f\in L^2(\Omega)$, for $(\bm \sigma, w)$ being the exact solution of \eqref{clamped:bd:con}, define 
 \begin{align*}
        E_{l} = \|\bm \sigma-\bm \sigma(\mathcal T_l)\|_{\mathbb C^{-1}}^2+\gamma\eta^2(\mathcal T_l)+(\beta+\gamma)\operatorname{osc}^2(f, \mathcal T_l).
    \end{align*}

\begin{theorem}
   There exist positive constants $0<\alpha<1$, $\beta>0$, and $\gamma>0$ such that 
    \begin{equation*}
        E_l\leq \alpha E_{l-1}.
    \end{equation*}
\end{theorem}
\begin{proof}
    In light of Lemma \ref{estimator:reduction}, the reliability established in Theorem \ref{relia:bd:upper}, and the quasi-orthogonality in Theorem \ref{quasi:orthogonality}, the proof can be concluded as in \cite{carstensen2014axioms,hu2018unified,huang2012convergence}.
\end{proof}
For $s>0$, define the approximation class $\mathbb A_s$ as
\begin{align*}
    \mathbb A_s=\{(\bm \sigma, f): |\bm \sigma, f|_s<\infty\},
\end{align*}
where 
\[|\bm \sigma, f|_s: = \sup_{N>0}\left(N^s \inf_{|\mathcal T_l|-|\mathcal T_0|\leq N}\, \inf_{\bm \tau_h\in \widetilde{\Sigma}(\mathcal T_l)}\|\bm \sigma-\bm \tau_h\|_{\mathbb C^{-1}}^2+\operatorname{osc}(f, \mathcal T_l)\right).\]
Let $\mathcal M_l$ represent a set of marked elements with minimal cardinality, $(\bm \sigma, w)$ the exact solution of \eqref{clamped:bd:con} and $(\mathcal T_l, \widetilde{\Sigma}(\mathcal T_l)\times U(\mathcal T_l), \bm \sigma(\mathcal T_l), w(\mathcal T_l))$ any level of the sequence of triangulations, finite element spaces and discrete solutions produced by the adaptive finite element methods with the marking parameter $\theta$. Then the optimality is presented below. 
\begin{theorem}[Optimality]It holds that 
\begin{equation*}
    \|\bm \sigma-\bm \sigma(\mathcal T_l)\|^2_{\mathbb C^{-1}}+\operatorname{osc}^2(f,\mathcal T_l)\lesssim |\bm \sigma, f|_s(|\mathcal T_l|-|\mathcal T_0|)^{-s}\quad\text{for}\, \, (\bm \sigma,f)\in\mathbb A_s.
\end{equation*}
\end{theorem}
\begin{proof}
    As the estimator reduction in Lemma \ref{estimator:reduction}, the quasi-orthogonality in Theorem \ref{quasi:orthogonality}, the discrete reliability in Theorem \ref{discrete:reliability}, and the efficiency \eqref{efficiency:1} have been established, the optimality follows from \cite{carstensen2014axioms,hu2018unified,huang2012convergence}. 
\end{proof}

%%%%%%%%%%%%%%%%%%%%%%%%%%%%%%%%%%%%%%%%%%%%%%%%%%%%%%%%%%%%%%%%%%%%%%

\section{Numerical examples}
 Some numerical examples are performed in this section to validate the theoretical results.

\subsection{L-shaped domain with a clamped corner}
The computation domain is $\Omega = (-1,1)\times (-1,1)\slash ([-1,0]\times [-1,0])$. The exact solution in polar coordinates is given as follows:
\begin{align*}
	&w(r,\phi) = r^{1+\alpha}\left(\cos[(\alpha+1)(\phi-\pi/4)]+C \cos[(\alpha-1)(\phi-\pi/4)]\right),\\
	 &\alpha  \approx  0.544483736782464, \quad C \approx 1.8414.
\end{align*}
Take $\mathbb C$ in \eqref{equ:0} as an identity matrix. The load function $f = \Delta^2 w = 0$. Let $\bm \sigma = \nabla^2 w$. The clamped part of the boundary reads 
\begin{align*}
	&w = \partial_{n} w  = 0 \quad  \text{on}\quad (-1,0]\times \{0\}\cup \{0\}\times (-1,0],
\end{align*}
while the remaining boundary is free boundary. This is a singular case with $\bm \sigma \in H^{\alpha-\epsilon}(\Omega;\mathbb{S})$ for all $\epsilon>0$, that is, $\bm \sigma\notin H^1(\Omega;\mathbb S)$ and with shear force $\operatorname{div}\bm \sigma \notin L^2(\Omega;\mathbb R^2)$. 

Figure \ref{LshapeMesh} illustrates the L-shaped domain with the adaptively refined meshes corresponding to the a posteriori error estimator $\eta$. Strong refinements are concentrated near the origin. The errors are plotted against the number of degrees of freedom $N$ in Figure \ref{fig:convergence:Lshaped}. 
Figure \ref{fig:convergence:Lshaped} shows that under uniform mesh-refinement, the suboptimal convergence rates of $\mathcal O(N^{-0.27})$ for $\bm \sigma$ and $\mathcal O(N^{-0.55})$ for $w$, corresponding to the expected rates $h^\alpha$ and $h^{2\alpha}$, respectively, in terms of the maximal mesh-size $h$. Despite the singularity, the adaptive algorithm recovers the optimal convergence rates $\mathcal O(N^{-1})$ for displacement errors $\|w-w_h\|_0$ and $\mathcal O(N^{-2})$ for $\|\bm \sigma-\bm \sigma_h\|_0$, $|w-w_h^\ast|_{2,h}$, and $\eta$ under adaptive refinements.

%%%%%%%%%%%%%%%%%%%%%%%%%%
\begin{figure}[htbp]
\includegraphics[width=0.65\textwidth]{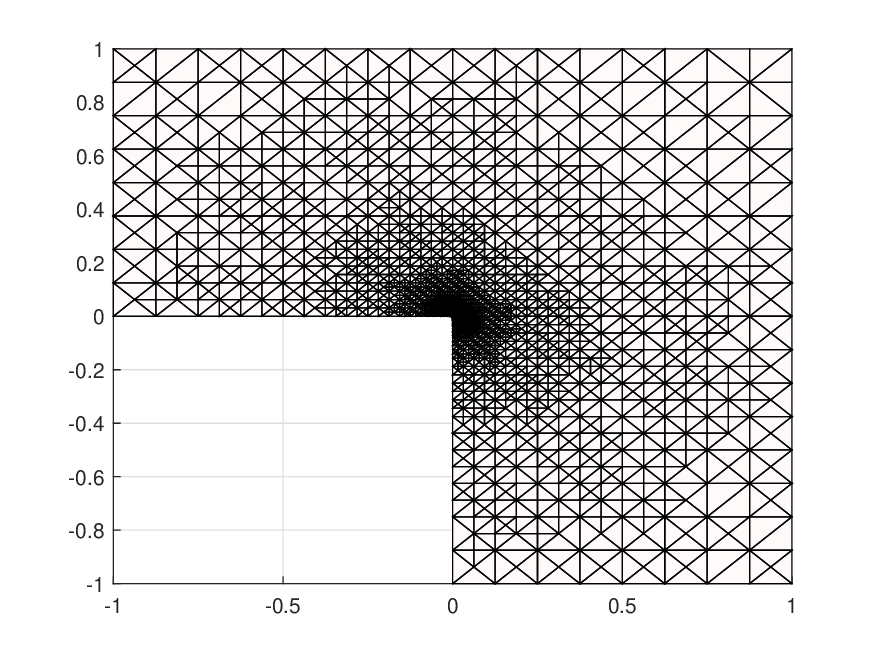}
 \caption{Adaptive meshes for L-shaped domain with a clamped corner}
 \label{LshapeMesh}
\end{figure}

\begin{figure}[htbp]
\includegraphics[width=0.75\textwidth]{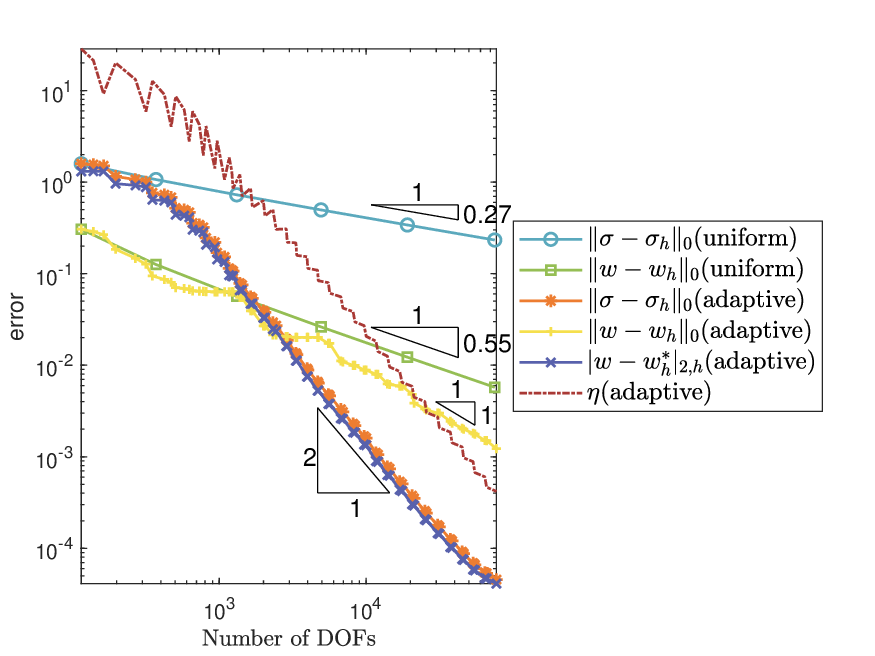}
 \caption{Clamped L-corner: The convergence history of the errors and the estimator}
 \label{fig:convergence:Lshaped}
\end{figure}
%%%%%%%%%%%%%%%%%%%%%%%%%%

%Let $\bm q = \operatorname{div}\bm \sigma$.Below, we provide an illustration of the exact solution, as depicted in Figure. Additionally, the numerical solutions for $h = \frac{1}{32}$ are presented and can be observed in Figure.

\subsection{L-shaped domain with a simply supported corner}
This case considers a solution with lower regularity on an L-shaped domain $\Omega = (-1,1)\times (-1,1)\slash ([0,1]\times [-1,0])$, where simply supported boundary conditions are imposed on the two boundary edges forming the reentrant corner. More precisely, the exact solution in polar coordinates is given by:
\begin{align*}
	&w(r,\phi) = r^{\frac{4}{3}}\sin(\frac{2}{3}\theta).
\end{align*}
Take $\mathbb C$ in \eqref{equ:0} as an identity matrix. The load function $f = \Delta^2 w = 0$. Let $\bm \sigma = \nabla^2 w$. The simply supported part of the boundary is 
\begin{align*}
	&w =\Delta w= 0 \quad  \text{on}\quad (0,1]\times \{0\}\cup \{0\}\times (-1,0],
\end{align*}
while the remaining boundary is clamped.

Figure \ref{LshapeMesh120} shows the L-shaped domain and the adaptively refined mesh based on the a posteriori error estimator 
$\eta$. Figure \ref{fig:convergence:Lshaped:m} plots the errors against $N$. Uniform refinements yield suboptimal rates, while adaptive refinements recover optimal convergence despite the singularity.

%%%%%%%%%%%%%%%%%%%%%%%%%%
\begin{figure}[htbp]
\includegraphics[width=0.6\textwidth]{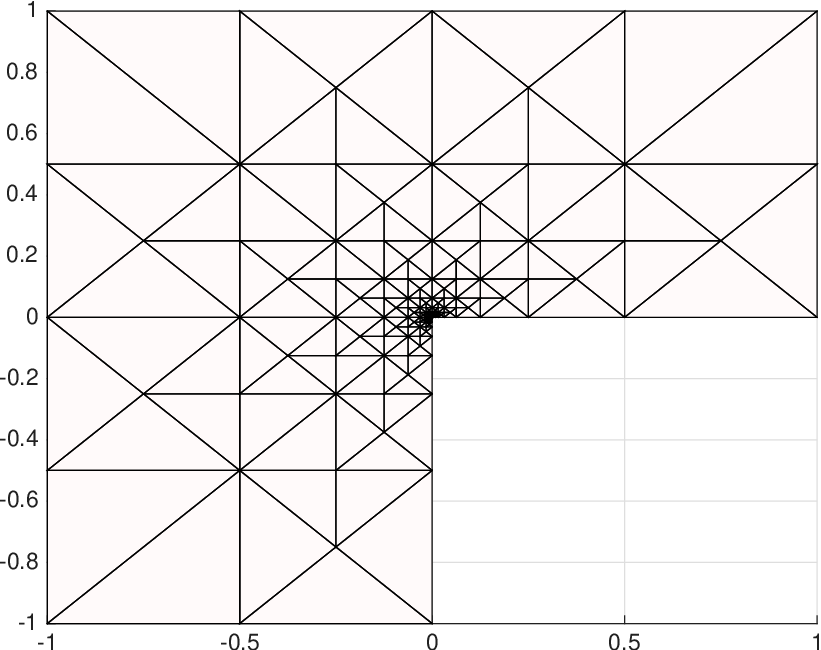}
 \caption{Adaptive meshes for L-shaped domain with a simply supported corner}
 \label{LshapeMesh120}
\end{figure}

\begin{figure}[htbp]
\includegraphics[width=0.7\textwidth]{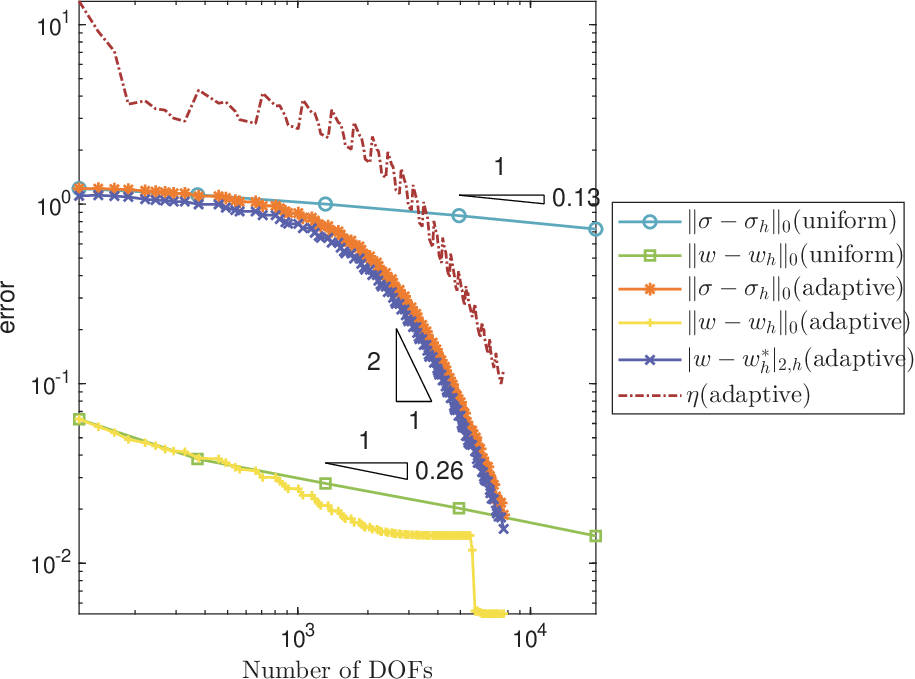}
 \caption{Simply supported L-corner: The convergence history of the errors and the estimator}
 \label{fig:convergence:Lshaped:m}
\end{figure}

%At the reentrant corner where the bending moment $\bm \sigma$ exhibits a singularity, the continuity of $\Sigma_h$ can be relaxed by enforcing only normal continuity, in a manner similar to the treatment of $H(\operatorname{div},\mathbb S)$ conforming elements in \cite[Section 4.1]{hu2021partial}.

\begin{remark}
The bending moment $\bm \sigma$ exhibits a point singularity at the reentrant corner. Actually, the continuity of $\Sigma_h$ at the reentrant corner can be relaxed by enforcing only normal continuity, in a manner similar to the treatment of $H(\operatorname{div},\mathbb S)$ conforming elements in \cite[Section 4.1]{hu2021partial}.
\end{remark}
%由于奇异性，sigma在顶点奇异，事实上顶点连续性是可以去掉，然后根据那篇文章
%%%%%%%%%%%%%%%%%%%%%%%%%%%%%%%%%%%%%%%%%%%%%%%%%%%%
%%%%%%%%%%%%%%%%%%%%%%%%%%%%%%%%%%%%%%%%%%%%%%%%%%%%

\subsection{The equation with jump coefficients}
The computational domain is $\Omega=(0,1)\times (0,1)$. Take $\mathbb C:=\alpha(\mathbf{x})\mathbf I$ in \eqref{equ:0} with
	\ben
	\alpha(\mathbf{x})=\left\{\begin{array}{ll}
		1~~~~~~~~&\text{in}~~(0,0.5)\times (0, 1),\\
		0.2~~~~~~~~&\text{in}~~(0.5,1)\times (0, 1).
	\end{array}\right.
	\een
    Here $\mathbf I$ is an identity matrix. 
Let the exact solution be
	\ben
	w=\left\{
	\begin{array}{ll}
		(2.88x^2-4.48x+1.92)x^2y^4(y-1)^4(0.4-1)^{-1}, ~~~~&(0,0.5)\times (0, 1),\\
		(-4.8x^2+6.4x-1.6)(x-1)^2y^4(y-1)^4(0.4-1)^{-1}, ~~~~&(0.5,1)\times (0, 1).
	\end{array}
	\right. 
	\een
Note that $\alpha(\mathbf x)$ varies in the left-hand side and right-hand side of the interface $x=0.5$. For all vertices on the interface $x=0.5$, the tangential-tangential continuity of the discrete symmetric bending moment function is relaxed as in Subsection \ref{relax:vertex}, leading to an extended $H(\operatorname{div}\operatorname{div})$ element space for solving this problem with jump coefficients. A corresponding postprocessed displacement $w_h^\ast\in W_h^\ast$ is defined by \eqref{post:pro:1} with $k=3$.

Table \ref{Ta1} displays the errors and convergence rates on uniform triangular meshes. It can be observed that the optimal order of convergence for both $\bm \sigma$ and $w$ is achieved. Errors $|w-w_{h}^{\ast}|_{2,h}$ convergence to $0$ with order $\mathcal O(h^4)$. Besides, errors $\|\mathcal Q_{h}w-w_{h}\|_{0}$ and $|\mathcal Q_{h}w-w_{h}|_{2,h}$ exhibit superconvergence of order $\mathcal O(h^4)$.

	\begin{table}[h!]\footnotesize\tabcolsep 10pt
		\begin{center}
			\caption{The error and the order of convergence on uniform meshes }\label{Ta1}\vspace{-2mm}
		\end{center}
		\begin{center}
			\begin{tabular}{c| c c c c c c }
				\toprule
				&$\|\bm \sigma-\bm \sigma_{h}\|_{0}$    &$h^{n}$     & $\|\operatorname{div}\operatorname{div}(\bm \sigma-\bm \sigma_{h})\|_{0}$   &$h^{n}$  &$\|w-w_{h}\|_{0}$  &$h^{n}$\\
				\midrule
				1 &4.7037e-03     &$-$     &5.1719e-01  &$-$ &1.3716e-04 &$-$ \\
				2  &2.7415e-03  &0.78     &6.5941e-01 & -0.35 &6.2782e-04 &1.13  \\
				3  &4.2808e-04   &2.68    &3.1020e-01   &1.09  &1.9384e-05 &1.70 \\
				4  &3.0212e-05   &3.82    &8.5000e-02   &1.87 &5.4184e-05  &1.84 \\
				5  &1.9634e-06     &3.94   &2.1742e-02  &1.98 &1.4091e-06  &1.94 \\
				\toprule
				&$\|\mathcal Q_{h}w-w_{h}\|_{0}$    &$h^{n}$        &$|\mathcal Q_{h}w-w_{h}|_{2,h}$ &$h^{n}$  &$|w-w_{h}^{\ast}|_{2,h}$  &$h^{n}$\\
				\midrule
				1&1.3010e-04&$-$  &1.1463e-03 &$-$  &7.7396e-03&$-$\\
				2 &2.8492e-05  &2.19 &1.0054e-03   &0.19 &4.2551e-03 &0.86\\
				3&7.8629e-06   &1.86 &3.1250e-04  &1.69 &6.9398e-04 &2.62\\
				4&6.1317e-07   &3.68 &2.8884e-05   &3.44  &5.1392e-05 &3.76\\
				5 &4.0323e-08   &3.93 &2.0723e-06   &3.80 &3.3972e-06 &3.92\\
				\bottomrule
			\end{tabular}
		\end{center}
	\end{table}

\bibliographystyle{plain}      % mathematics and physical sciences
\bibliography{bibfile}

\begin{thebibliography}{10}

\bibitem{AdiniNonFem}
Ari Adini and Ray~W. Clough.
\newblock Analysis of plate bending by the finite element method.
\newblock page 7337. University of California., 1960.

\bibitem{alonso1996error}
Ana Alonso.
\newblock Error estimators for a mixed method.
\newblock {\em Numerische Mathematik}, 74(4):385--395, 1996.

\bibitem{amrouche2006characterizations}
Cherif Amrouche, Philippe~G Ciarlet, Liliana Gratie, and Srinivasan Kesavan.
\newblock On the characterizations of matrix fields as linearized strain tensor
  fields.
\newblock {\em Journal de math{\'e}matiques pures et appliqu{\'e}es},
  86(2):116--132, 2006.

\bibitem{arnold2021complexes}
Douglas~N Arnold and Kaibo Hu.
\newblock Complexes from complexes.
\newblock {\em Foundations of Computational Mathematics}, 21(6):1739--1774,
  2021.

\bibitem{beirao2010posteriori}
L~Beir{\~a}o~da Veiga, Jarkko Niiranen, and Rolf Stenberg.
\newblock A posteriori error analysis for the morley plate element with general
  boundary conditions.
\newblock {\em International journal for numerical methods in engineering},
  83(1):1--26, 2010.

\bibitem{bbf}
Daniele Boffi, Franco Brezzi, and Michel Fortin.
\newblock {\em Mixed finite element methods and applications}.
\newblock Springer, Heidelberg, 2013.

\bibitem{brenner2010posteriori}
Susanne~C Brenner, Thirupathi Gudi, and Li-yeng Sung.
\newblock An a posteriori error estimator for a quadratic ${C}^0$-interior
  penalty method for the biharmonic problem.
\newblock {\em IMA Journal of Numerical Analysis}, 30(3):777--798, 2010.

\bibitem{brenner2015c}
Susanne~C Brenner, Peter Monk, and Jiguang Sun.
\newblock ${C}^0$interior penalty galerkin method for biharmonic eigenvalue
  problems.
\newblock In {\em Spectral and High Order Methods for Partial Differential
  Equations ICOSAHOM 2014: Selected papers from the ICOSAHOM conference}, pages
  3--15. Springer, 2015.

\bibitem{brennerscott}
Susanne~C. Brenner and L.~Ridgway Scott.
\newblock {\em The mathematical theory of finite element methods}.
\newblock Springer, New York, 2008.

\bibitem{brenner2005c}
Susanne~C Brenner and Li-Yeng Sung.
\newblock C 0 interior penalty methods for fourth order elliptic boundary value
  problems on polygonal domains.
\newblock {\em Journal of Scientific Computing}, 22(1):83--118, 2005.

\bibitem{carstensen2014axioms}
Carsten Carstensen, Michael Feischl, Marcus Page, and Dirk Praetorius.
\newblock Axioms of adaptivity.
\newblock {\em Computers \& Mathematics with Applications}, 67(6):1195--1253,
  2014.

\bibitem{carstensen2014discrete}
Carsten Carstensen, Dietmar Gallistl, and Jun Hu.
\newblock A discrete helmholtz decomposition with morley finite element
  functions and the optimality of adaptive finite element schemes.
\newblock {\em Computers \& Mathematics with Applications}, 68(12):2167--2181,
  2014.

\bibitem{carstensen2021hierarchical}
Carsten Carstensen and Jun Hu.
\newblock Hierarchical argyris finite element method for adaptive and multigrid
  algorithms.
\newblock {\em Computational Methods in Applied Mathematics}, 21(3):529--556,
  2021.

\bibitem{cascon2008quasi}
J~Manuel Cascon, Christian Kreuzer, Ricardo~H Nochetto, and Kunibert~G Siebert.
\newblock Quasi-optimal convergence rate for an adaptive finite element method.
\newblock {\em SIAM Journal on Numerical Analysis}, 46(5):2524--2550, 2008.

\bibitem{charbonneau1997residual}
Alain Charbonneau, Kokou Dossou, and Roger Pierre.
\newblock A residual-based a posteriori error estimator for the ciarlet-raviart
  formulation of the first biharmonic problem.
\newblock {\em Numerical Methods for Partial Differential Equations: An
  International Journal}, 13(1):93--111, 1997.

\bibitem{chen2020finite2d}
Long Chen and Xuehai Huang.
\newblock Finite elements for divdiv-conforming symmetric tensors.
\newblock {\em arXiv preprint arXiv:2005.01271}, 2021.

\bibitem{chen2022finite}
Long Chen and Xuehai Huang.
\newblock Finite elements for div-and divdiv-conforming symmetric tensors in
  arbitrary dimension.
\newblock {\em SIAM Journal on Numerical Analysis}, 60(4):1932--1961, 2022.

\bibitem{ChenHuang20223d}
Long Chen and Xuehai Huang.
\newblock Finite elements for {${\rm div\,div}$} conforming symmetric tensors
  in three dimensions.
\newblock {\em Math. Comp.}, 91(335):1107--1142, 2022.

\bibitem{chen2025new}
Long Chen and Xuehai Huang.
\newblock A new div-div-conforming symmetric tensor finite element space with
  applications to the biharmonic equation.
\newblock {\em Mathematics of Computation}, 94(351):33--72, 2025.

\bibitem{ciarlet1974mixed}
Philippe~G Ciarlet and Pierre-Arnaud Raviart.
\newblock A mixed finite element method for the biharmonic equation.
\newblock In {\em Mathematical aspects of finite elements in partial
  differential equations}, pages 125--145. Elsevier, 1974.

\bibitem{cockburn2009hybridizable}
Bernardo Cockburn, Bo~Dong, and Johnny Guzm{\'a}n.
\newblock A hybridizable and superconvergent discontinuous galerkin method for
  biharmonic problems.
\newblock {\em Journal of Scientific Computing}, 40(1):141--187, 2009.

\bibitem{da2007family}
L~Beir{\~a}o Da~Veiga, Jarkko Niiranen, and Rolf Stenberg.
\newblock A family of ${C}^{0}$ finite elements for kirchhoff plates i: Error
  analysis.
\newblock {\em SIAM Journal on Numerical Analysis}, 45(5):2047--2071, 2007.

\bibitem{da2007posteriori}
L~Beirao da~Veiga, Jarkko Niiranen, and Rolf Stenberg.
\newblock A posteriori error estimates for the morley plate bending element.
\newblock {\em Numerische Mathematik}, 106:165--179, 2007.

\bibitem{de1974variational}
B~Fraeijs De~Veubeke.
\newblock Variational principles and the patch test.
\newblock {\em International Journal for Numerical Methods in Engineering},
  8(4):783--801, 1974.

\bibitem{fuhrer2024mixed}
Thomas F{\"u}hrer and Norbert Heuer.
\newblock Mixed finite elements for {K}irchhoff--{L}ove plate bending.
\newblock {\em Math. Comp.}, 94:1065--1099, 2025.

\bibitem{georgoulis2009discontinuous}
Emmanuil~H Georgoulis and Paul Houston.
\newblock Discontinuous galerkin methods for the biharmonic problem.
\newblock {\em IMA journal of numerical analysis}, 29(3):573--594, 2009.

\bibitem{georgoulis2011posteriori}
Emmanuil~H Georgoulis, Paul Houston, and Juha Virtanen.
\newblock An a posteriori error indicator for discontinuous galerkin
  approximations of fourth-order elliptic problems.
\newblock {\em IMA journal of numerical analysis}, 31(1):281--298, 2011.

\bibitem{girault2002hermite}
Vivette Girault and L~Scott.
\newblock Hermite interpolation of nonsmooth functions preserving boundary
  conditions.
\newblock {\em Mathematics of Computation}, 71(239):1043--1074, 2002.

\bibitem{gudi2011residual}
Thirupathi Gudi.
\newblock Residual-based a posteriori error estimator for the mixed finite
  element approximation of the biharmonic equation.
\newblock {\em Numerical Methods for Partial Differential Equations},
  27(2):315--328, 2011.

\bibitem{gustafsson2018posteriori}
Tom Gustafsson, Rolf Stenberg, and Juha Videman.
\newblock A posteriori estimates for conforming kirchhoff plate elements.
\newblock {\em SIAM Journal on Scientific Computing}, 40(3):A1386--A1407, 2018.

\bibitem{hansbo2011posteriori}
Peter Hansbo and Mats~G Larson.
\newblock A posteriori error estimates for continuous/discontinuous galerkin
  approximations of the kirchhoff-love plate.
\newblock {\em Computer methods in applied mechanics and engineering},
  200(47-48):3289--3295, 2011.

\bibitem{hellan1967analysis}
K{\aa}re Hellan.
\newblock Analysis of elastic plates in flexure by a simplified finite element
  method.
\newblock {\em Acta polytechnica Scandinavica. Civil engineering and building
  construction series}, (46):1, 1967.

\bibitem{herrmann1967finite}
Leonard~R Herrmann.
\newblock Finite-element bending analysis for plates.
\newblock {\em Journal of the Engineering Mechanics Division}, 93(5):13--26,
  1967.

\bibitem{hu2015finite}
Jun Hu.
\newblock Finite element approximations of symmetric tensors on simplicial
  grids in $\mathbb{R}^n$: The higher order case.
\newblock {\em Journal of Computational Mathematics}, pages 283--296, 2015.

\bibitem{hu2022conforming}
Jun Hu, Yizhou Liang, and Rui Ma.
\newblock Conforming finite element divdiv complexes and the application for
  the linearized einstein--bianchi system.
\newblock {\em SIAM Journal on Numerical Analysis}, 60(3):1307--1330, 2022.

\bibitem{hu2021partial}
Jun Hu and Rui Ma.
\newblock Partial relaxation of ${C}^0$ vertex continuity of stresses of
  conforming mixed finite elements for the elasticity problem.
\newblock {\em Computational Methods in Applied Mathematics}, 21(1):89--108,
  2021.

\bibitem{hu2021family}
Jun Hu, Rui Ma, and Min Zhang.
\newblock A family of mixed finite elements for the biharmonic equations on
  triangular and tetrahedral grids.
\newblock {\em Science China Mathematics}, 64(12):2793--2816, 2021.

\bibitem{hu2009new}
Jun Hu and Zhongci Shi.
\newblock A new a posteriori error estimate for the morley element.
\newblock {\em Numerische Mathematik}, 112:25--40, 2009.

\bibitem{hu2012convergence}
Jun Hu, Zhongci Shi, and Jinchao Xu.
\newblock Convergence and optimality of the adaptive morley element method.
\newblock {\em Numerische Mathematik}, 121(4):731--752, 2012.

\bibitem{hu2018unified}
Jun Hu and Guozhu Yu.
\newblock A unified analysis of quasi-optimal convergence for adaptive mixed
  finite element methods.
\newblock {\em SIAM Journal on Numerical Analysis}, 56(1):296--316, 2018.

\bibitem{hu2014family}
Jun Hu and Shangyou Zhang.
\newblock A family of conforming mixed finite elements for linear elasticity on
  triangular grids.
\newblock {\em arXiv:1406.7457}, 2014.

\bibitem{HuZhang2015}
Jun Hu and Shangyou Zhang.
\newblock A family of symmetric mixed finite elements for linear elasticity on
  tetrahedral grids.
\newblock {\em Sci. China Math.}, 58(2):297--307, 2015.

\bibitem{huang2011convergence}
Jianguo Huang, Xuehai Huang, and Yifeng Xu.
\newblock Convergence of an adaptive mixed finite element method for kirchhoff
  plate bending problems.
\newblock {\em SIAM journal on numerical analysis}, 49(2):574--607, 2011.

\bibitem{huang2012convergence}
JianGuo Huang and YiFeng Xu.
\newblock Convergence and complexity of arbitrary order adaptive mixed element
  methods for the poisson equation.
\newblock {\em Science China Mathematics}, 55:1083--1098, 2012.

\bibitem{johnson1973convergence}
Claes Johnson.
\newblock On the convergence of a mixed finite-element method for plate bending
  problems.
\newblock {\em Numerische Mathematik}, 21(1):43--62, 1973.

\bibitem{ming2007nonconforming}
Wang Ming and Jinchao Xu.
\newblock Nonconforming tetrahedral finite elements for fourth order elliptic
  equations.
\newblock {\em Mathematics of computation}, 76(257):1--18, 2007.

\bibitem{morley1967triangular}
LSD Morley.
\newblock A triangular equilibrium element with linearly varying bending
  moments for plate bending problems.
\newblock {\em The Aeronautical Journal}, 71(682):715--719, 1967.

\bibitem{stevenson2008completion}
Rob Stevenson.
\newblock The completion of locally refined simplicial partitions created by
  bisection.
\newblock {\em Mathematics of computation}, 77(261):227--241, 2008.

\bibitem{suli2007hp}
Endre S{\"u}li and Igor Mozolevski.
\newblock hp-version interior penalty dgfems for the biharmonic equation.
\newblock {\em Computer methods in applied mechanics and engineering},
  196(13-16):1851--1863, 2007.

\bibitem{sun2018quasi}
Pengtao Sun and Xuehai Huang.
\newblock Quasi-optimal convergence rate for an adaptive hybridizable c 0
  discontinuous galerkin method for kirchhoff plates.
\newblock {\em Numerische Mathematik}, 139(4):795--829, 2018.

\bibitem{verfurth2013posteriori}
R{\"u}diger Verf{\"u}rth.
\newblock {\em A posteriori error estimation techniques for finite element
  methods}.
\newblock OUP Oxford, 2013.

\bibitem{xu2014posteriori}
Yifeng Xu, Jianguo Huang, and Xuehai Huang.
\newblock A posteriori error estimates for local $c^0$ discontinuous galerkin
  methods for kirchhoff plate bending problems.
\newblock {\em Journal of Computational Mathematics}, pages 665--686, 2014.

\bibitem{zhang2008invalidity}
Sheng Zhang and Zhimin Zhang.
\newblock Invalidity of decoupling a biharmonic equation to two poisson
  equations on non-convex polygons.
\newblock {\em International Journal of Numerical Analysis and Modeling},
  5(1):73--76, 2008.

\end{thebibliography}

\begin{appendices}
\section{The proof of the inf-sup condition}\label{sec:appendix}
To obtain the well-posedness of the problem \eqref{mixed:bd:con}, the inf-sup condition is established through the following lemma. 
\begin{lemma}\label{step2:infsup}
    For $\bm \phi \in H^1(\Omega;\mathbb R^2)$ satisfying $\bm \phi\cdot \mathbf n |_{\Gamma_F}=0$, there exists some $\bm \tau^\ast\in H^1(\Omega;\mathbb S)$ with $\mathbf n^{\mathrm T}\bm \tau^\ast \mathbf n |_{\Gamma_S\cup \Gamma_F}=0$ and $\mathbf t^{\mathrm T}\bm \tau^\ast\mathbf n|_{\Gamma_F}=0$ such that $\operatorname{div}\bm \tau^\ast=\bm \phi$ and $\|\bm \tau^\ast\|_1\lesssim\|\bm \phi\|_0$.
\end{lemma}
\begin{proof}
    This shall be proved by a contradiction argument. In fact, suppose that such $\bm \tau^\ast$ does not exist, then there would exist a sequence of $\{\bm \phi_n\}$ in $H^1(\Omega;\mathbb R^2)$ satisfying $\bm \phi_n\cdot\mathbf n|_{\Gamma_F}=0$ and $\|\bm \phi_n\|_0=1$ such that
\begin{equation}\label{assump:1}
   \lim_{n\rightarrow \infty}\sup_{
   \substack{\bm \tau^\ast\in H^1(\Omega;\mathbb S)\\
   \mathbf n^{\mathrm T}\bm \tau^\ast\mathbf n|_{\Gamma_S\cup\Gamma_F}=0\\
   \mathbf t^{\mathrm T}\bm \tau^\ast\mathbf n|_{\Gamma_F}=0}}\frac{(\operatorname{div}\bm \tau^\ast, \bm \phi_n)}{\|\bm \tau^\ast\|_1}= 0.
\end{equation}
Owning to
\begin{align*}
   \|\varepsilon(\bm \phi_n)\|_{H^{-1}} =\sup_{\bm \tau\in H^1(\Omega;\mathbb S)\atop \bm \tau\mathbf n|_{\partial \Omega}=0}\frac{(\varepsilon(\bm \phi_n), \bm \tau)}{\|\bm \tau\|_1}\leq \sup_{
   \substack{\bm \tau^\ast\in H^1(\Omega;\mathbb S)\\
   \mathbf n^{\mathrm T}\bm \tau^\ast\mathbf n|_{\Gamma_S\cup\Gamma_F}=0\\
   \mathbf t^{\mathrm T}\bm \tau^\ast\mathbf n|_{\Gamma_F}=0}}\frac{(\bm \phi_n, \operatorname{div}\bm \tau^\ast)}{\|\bm \tau^\ast\|_1}, 
\end{align*}
the assumption \eqref{assump:1} leads to 
\begin{equation}\label{qn:H:-1}
   \lim_{n\rightarrow\infty}\|\varepsilon(\bm \phi_n)\|_{H^{-1}} =0.
\end{equation}
In addition, since $L^2(\Omega;\mathbb R^2)$ is compactly embedded in $H^{-1}(\Omega;\mathbb R^2)$ and $\|\bm \phi_n\|_0=1$, some subsequence, say $\{\bm \phi_n\}$ for simplicity, converges in $H^{-1}(\Omega;\mathbb R^2)$ to some limit $\bm \phi^\ast$:
\begin{equation}\label{qn:convergent:to:p:in:H}
    \lim_{n\rightarrow \infty}\|\bm \phi_n-\bm \phi^\ast\|_{H^{-1}} = 0.
\end{equation}
The Korn inequality in $L^2$ presented in \cite[Theorem 3.2]{amrouche2006characterizations} gives rise to
\begin{equation}
    \|\bm \phi_n-\bm \phi_m\|_0\lesssim \|\bm \phi_n-\bm \phi_m\|_{H^{-1}}+\|\varepsilon(\bm \phi_n-\bm \phi_m)\|_{H^{-1}}.
\end{equation}
According to \eqref{qn:convergent:to:p:in:H} and \eqref{qn:H:-1}, $\{\bm \phi_n\}$ is a Cauchy sequence in $L^2(\Omega;\mathbb R^2)$. Thus 
\begin{equation}\label{qn:L2:p}
    \lim_{n\rightarrow \infty}\|\bm \phi_n-\bm \phi^\ast\|_0=0.
\end{equation}
A triangle inequality leads to
\begin{equation}\label{p:H:-1:0}
\|\varepsilon(\bm \phi^\ast)\|_{H^{-1}} \leq \|\varepsilon(\bm \phi^\ast-\bm \phi_n)\|_{H^{-1}}+\|\varepsilon(\bm \phi_n)\|_{H^{-1}}.
\end{equation}
The definition of the $H^{-1}$-norm and the Cauchy inequality result in
\begin{equation}
    \lim_{n\rightarrow 0} \|\varepsilon(\bm \phi^\ast-\bm \phi_n)\|_{H^{-1}} \lesssim \lim_{n\rightarrow 0}\|\bm \phi^\ast-\bm \phi_n\|_{0}=0.
\end{equation}
Using this and \eqref{qn:H:-1}, passing the right hand side of the inequality \eqref{p:H:-1:0} to the limit, one can obtain $\|\varepsilon(\bm \phi^\ast)\|_{H^{-1}} =0$, Therefore, 
\[\bm \phi^\ast=\bm a +b\begin{pmatrix}
  y\\
  -x
\end{pmatrix}\quad\text{with} \, \, \bm a\in\mathbb R^2\, \, \text{and}\, \, b\in \mathbb R.\]
Since $\|\bm \phi^\ast\|_0=1$ and $\bm \phi^\ast|_{\Gamma_C}\not \equiv 0$, an integration by parts gives rise to 
\begin{equation*}
    (\operatorname{div}\bm \tau^\ast, \bm \phi^\ast) = \langle\bm \tau^\ast\mathbf n, \bm \phi^\ast\rangle_{\partial\Omega}=\langle\mathbf n^{\mathrm T}\bm \tau^\ast\mathbf n, \bm \phi^\ast\cdot\mathbf n\rangle_{\Gamma_C}+\langle\mathbf t^{\mathrm T}\bm \tau^\ast\mathbf n, \bm \phi^\ast\cdot\mathbf t\rangle_{\Gamma_C\cup\Gamma_S}.
\end{equation*}
Consequently, there exists some $\bm \tau^
\ast\in H^1(\Omega;\mathbb S)$ satisfying $\mathbf n^{\mathrm T}\bm \tau^\ast\mathbf n|_{\Gamma_S\cup\Gamma_F}=0$ and $\mathbf t^{\mathrm T}\bm \tau^\ast\mathbf n|_{\Gamma_F}=0$ such that 
$(\operatorname{div}\bm \tau^\ast, \bm \phi^\ast)\neq 0$. This contradicts \eqref{assump:1}.
\end{proof}

\emph{Proof of Theorem \ref{inf:sup:con:bd}}
For any $v\in L^2(\Omega)$, there exists some $\bm \phi \in H^1(\Omega;\mathbb R^2)$ satisfying $\bm \phi\cdot\mathbf n|_{\Gamma_F} = 0$, such that $\operatorname{div}\bm \phi = v$ and $\|\bm \phi\|_1\lesssim\|v\|_0$. In light of Lemma \ref{step2:infsup}, for such $\bm \phi$, there exists $\bm \tau^\ast\in H^1(\Omega;\mathbb S)$ satisfying $\mathbf n^{\mathrm T}\bm \tau^\ast \mathbf n |_{\Gamma_S\cup \Gamma_F}=0$ and $\mathbf t^{\mathrm T}\bm \tau^\ast\mathbf n|_{\Gamma_F}=0$ such that $\operatorname{div}\bm \tau^\ast=\bm \phi$ and $\|\bm \tau^\ast\|_1\lesssim \|\bm \phi\|_0$. For any $u\in \Lambda$, an integration by parts shows that 
\begin{align*}
    &(\bm \tau^\ast, \nabla^2 u)-(\operatorname{div}\operatorname{div}\bm \tau^\ast, u)=\langle\bm \tau^\ast\mathbf n, \nabla u\rangle_{\partial\Omega}-(\bm \phi, \nabla u)-(\operatorname{div}\bm \phi, u)\\
    & = \langle\mathbf n^{\mathrm T}\bm \tau^\ast\mathbf n, \partial_n u\rangle_{\partial\Omega}+\langle\mathbf t^{\mathrm T}\bm \tau^\ast\mathbf n, \partial_t u\rangle_{\partial\Omega}-\langle\bm \phi\cdot\mathbf n, u\rangle_{\partial\Omega}\\
    &= \langle\mathbf n^{\mathrm T}\bm \tau^\ast\mathbf n, \partial_n u\rangle_{\Gamma_C}+\langle\mathbf t^{\mathrm T}\bm \tau^\ast\mathbf n, \partial_t u\rangle_{\Gamma_C\cup\Gamma_S}-\langle\bm \phi\cdot\mathbf n, u\rangle_{\Gamma_C\cup\Gamma_S}=0.
\end{align*}
Thus $\bm \tau^\ast\in \Sigma_0$ follows. This combined with $\operatorname{div}\operatorname{div}\bm \tau^\ast=v$ and $\|\bm \tau^\ast\|_{H(\operatorname{div}\operatorname{div})}\lesssim\|v\|_0$ results in
\begin{align*}
 \sup_{\bm \tau \in \Sigma_0\atop \bm \tau \neq 0}\frac{(\operatorname{div}\operatorname{div}\bm \tau,v)}{\|\bm \tau\|_{H(\operatorname{div}\operatorname{div})}}\geq    \frac{(\operatorname{div}\operatorname{div}\bm \tau^\ast,v)}{\|\bm \tau^\ast\|_{H(\operatorname{div}\operatorname{div})}}=\frac{\|v\|_0^2}{\|\bm \tau\|_{H(\operatorname{div}\operatorname{div})}}\geq C\|v\|_0.
\end{align*}
\qed

\section{Basis functions in 2D}
This section provides the basis of the $H(\operatorname{div}\operatorname{div},\Omega;\mathbb S)$ conforming finite elements $\Sigma_{h}$, which greatly facilitates the implementation.
Furthermore, small modifications of these basis functions yield the basis functions for $H(\operatorname{div}\operatorname{div},\Omega;\mathbb S)$ elements in \cite{chen2020finite2d}.

Given a triangle $K\in\mathcal{T}_h$, let $\mn$ denote the unit outnormal of $\partial K$. Suppose $e_{i}$ denotes the edge opposite to vertex $a_i$ of $K$ with $1\leq i\leq 3$. Let $\mn_i=(n_1,n_2)^T$ denote the unit outnormal of $e_i$ and $\mt_i=(-n_2,n_1)^T$ denote the tangential vector of $e_i$. Let $h_i$ denote the height of $e_i$. Let $\lambda_i$ denote the barycentric coordinates with respect to $a_i$. Let $\{\bm{S}_j\}_{j=1}^3$ denote the canonical basis of $\mathbb{S}$ in two dimensions. Let $\bm T_{i,j}:=\mathbf t_i\mathbf t_j^{\mathrm T}+\mathbf t_j\mathbf t_i^{\mathrm T}$.
Define $\bm T_i: = \mathbf t_i\mathbf t_i^{\mathrm T}$, and its orthogonal complement matrices $\bm T_i^{\perp,j}\in \mathbb S$ satisfying $\bm T_i^{\perp,j}:\bm T_i=0$.
For instance, one can select 
\begin{align*}
\bm T_i^{\perp,1}=\mn_i\mn_i^T,\quad\bm T_i^{\perp,2}=\mn_i\mt_i^T+\mt_i\mn_i^T.
\end{align*}
Notably, $\bm  T_i$ and $\bm  T_{i}^{\perp,j}$ also form a set of basis for $\mathbb S$ in two dimensions, denoted by $\{\widetilde{\bm  S}_j\}_{j=1}^3$.

The following relation can significantly facilitate the calculations. 
\begin{remark}
    For any $\bta=\phi \bm  T$ with a scalar function $\phi$ and a constant matrix $\bm  T$, $\bd\bta\cdot\mn=\bm  T\nabla\phi$ and that $\nabla\lambda_i=-h_i^{-1}\mn_i$ with $1\leq i\leq 3$.
\end{remark}

Let $P_k(e)$ denote the space of polynomials of degree $k$ on edge $e$, where $\lambda_i,\lambda_j$ are the barycentric coordinates associated with respect to the two vertices of $e$. Any function $\phi=\sum_{m=0}^kc_m\big(\lambda_i|_e\big)^m\big(\lambda_j|_e\big)^{k-m}\in P_k(e)$ can be extended to $P_k(K)$ via the natural extension of the barycentric coordinates to element $K$. Let
\[c_{i} : = \frac{-h_{i}}{2(\mathbf n_{i}^{\mathrm T}\mathbf t_{i+1})(\mathbf n_{i}^{\mathrm T}\mathbf t_{i+2})}.\]

Define a bubble function space $\Sigma_{\partial K,b,k}=\{\bta\in P_{k}(K;\mathbb{S}):\bta\mn|_{\partial K}=0,\,\bd\bta\cdot\mn|_{\partial K}=0\}$ of $\Sigma_h$. Note that 
\begin{equation*}
    \operatorname{dim}\Sigma_{\partial K,b,k} = \frac{3}{2}k(k-3).
\end{equation*}

The basis functions of $\Sigma_{h}$ are give by:
\begin{enumerate}
    \item (Bubble function basis) The basis function of the bubble function space $\Sigma_{\partial K,b,k}$ on $K$: for $1\leq i\leq 3$, 
\begin{equation*}
\bm \tau^{(1)}_{i } = \left\{\begin{array}{ll}
      (\lambda_{i+1}\lambda_{i+2})^2\phi\bm T_i& \text{for }\, \phi\in P_{k-4}(e_i),\\
      \lambda_i^2\lambda_{i+1}\lambda_{i+2}\phi\bm T_{i+1,i+2}&\text{for }\, \phi\in P_{k-4}(K); \\
    
\end{array}\right.   
\end{equation*}
\item (Edge basis for $\bd\bta\cdot\mn$) The basis function with respect to $(\bd\bta\cdot\mn_i)|_{e_i}$for edge $e_i$ with $1\leq i\leq 3$:
\begin{equation*}
    \bm \tau^{(2)}_{i,\ell }=\left\{\begin{array}{ll}
        \frac{-h_i}{(\mn_i^{\mathrm T}\mt_{i+2})^2}\lambda_i\lambda_{i+1}^2\bm T_{i+2}  &\text{for }\ell=1,\\
          \frac{-h_i}{(\mn_i^{\mathrm T}\mt_{i+1})^2}\lambda_i\lambda_{i+2}^2\bm T_{i+1}  &\text{for }\ell=2,\\
 c_i\lambda_1\lambda_2\lambda_3\phi\bm T_{i+1,i+2}& \text{with }\,\phi\in P_{k-3}(e_i)\text{ for }\ell>2;
    \end{array}\right.
\end{equation*}
\item (Edge basis for $\bta\mn$) The basis function with respect to $\bm \tau \mathbf n_i|_{e_i}$ for edge $e_i$ with $1\leq i\leq 3$ and $1\leq j \leq 2$: 
\begin{align*}
\bm \tau^{(3)}_{i,j} =\left\{\begin{array}{lll} \lambda_{i+1}^2\lambda_{i+2}\bm T_i^{\perp,j}+\frac{\mathbf n_{i+2}^{\mathrm T}\bm T_i^{\perp, j}\mathbf n_{i+2}}{h_{i+2}} \bm\tau^{(2)}_{i+2,2}+\frac{\mathbf n_{i+2}^{\mathrm T}\bm T_i^{\perp, j}\mathbf n_{i}}{h_{i+2}} \bm \tau^{(2)}_{i,1} +\frac{2c_i\mathbf n_{i+1}^{\mathrm T}\bm T_i^{\perp, j}\mathbf n_{i}}{h_{i+1}}  \lambda_1\lambda_2\lambda_3\bm T_{i+1,i+2} ,\\
\lambda_{i+2}^2\lambda_{i+1}\bm T_i^{\perp,j}+\frac{\mathbf n_{i+1}^{\mathrm T}\bm T_i^{\perp, j}\mathbf n_{i+1}}{h_{i+1}} \bm\tau^{(2)}_{i+1,1}+\frac{\mathbf n_{i+1}^{\mathrm T}\bm T_i^{\perp, j}\mathbf n_{i}}{h_{i+1}} \bm \tau^{(2)}_{i,2} +\frac{2c_i\mathbf n_{i+2}^{\mathrm T}\bm T_i^{\perp, j}\mathbf n_{i}}{h_{i+2}}  \lambda_1\lambda_2\lambda_3\bm T_{i+1,i+2},\\
    \phi \bm T_i^{\perp,j}- c_{i} (\nabla\phi )^T \bm T_i^{\perp,j}\mn_{i} \lambda_i\mathbb T_{i+1,i+2} \quad \text{for }\phi =(\lambda_{i+1}\lambda_{i+2})^2P_{k-4}(e_i);
    \end{array}\right.
\end{align*}

\item (Vertex basis) The basis function with respect to $\bm \tau(a_i)$ for vertex $a_i$ with $1\leq i\leq 3$:
\begin{align*}
    \bm \tau^{(4)}_{i,j}=\lambda_i^3\bm{S}_j+\frac{3\mn_{i}^{\mathrm T}\widetilde{\bm S}_j\mathbf n_{i+1}}{h_i}  \bm \tau^{(2)}_{i+1,2}+\frac{3\mn_{i}^{\mathrm T}\widetilde{\bm S}_j\mathbf n_{i+2}}{h_i}  \bm  \tau^{(2)}_{i+2,1} \, \quad\text{for}\, \, 1\leq j\leq 3.
\end{align*}
\end{enumerate}

%%%%%%%%%%%%%%%%%%%%%%%%%%%
               \iffalse
               For $1\leq l\leq 2$,
\begin{align*}
    \bm \tau^{(3)}_{i,l} = \lambda_{i+1}^2\lambda_{i+2}\bm T_i^{\perp,j}+\nu_{2,i,j}(\mathbf n_{i+2}\bm \tau^{(2)}_{i+2,2}+\mathbf n_{i}\bm \tau^{(2)}_{i,1})+\nu_{1,i,j}\mathbf n_i\bm \tau_i,
\end{align*}
for  $3\leq l\leq 4$,
\begin{align*}
 \bm \tau^{(3)}_{i,l}=\lambda_{i+2}^2\lambda_{i+1}\bm T_i^{\perp,j}+\nu_{1,i,j}(\mathbf n_{i+1}\bm \tau^{(2)}_{i+1,1}+\mathbf n_{i}\bm \tau^{(2)}_{i,2})+\nu_{2,i,j}\mathbf n_i\bm \tau_i,
\end{align*}
and for $5\leq l\leq k-3$ with $\varphi_{i,l-4}:=(\lambda_{i+1}\lambda_{i+2})^2\zeta^{k-4}_{i,l-4}$:
\begin{align*}
\tau^{(3)}_{i,l}=\varphi_{i,l-4}\bm T_i^{\perp,j}-c_i\lambda_i(\nabla\varphi_{i,l-4})^T \bm T_i^{\perp,j}\mn_{i}\bm T_{i+1,i+2};
\end{align*}
               \fi
%%%%%%%%%%%%%%%%%%%%%%%%%%%

\begin{remark}
The sign of the the outnormal needs to be taken into consideration when forming global basis functions.
\end{remark}

\begin{remark}
   Simple calculations yield
    \begin{equation}\label{trigger:1}
    \operatorname{div}(\bm \tau^{(2)}_{i,\ell})\cdot\mathbf n_i|_{e_i} = \left\{\begin{array}{ll}
         \lambda_{i+\ell}^2&\text{for} \, \,1\leq \ell\leq 2,  \\
         \lambda_{i+1}\lambda_{i+2}\phi&\text{for}\,\, \ell>2.
    \end{array}\right.
\end{equation} 
This serves as an efficient tool in representing basis functions (3) explicitly. In reality, once $\bm \tau^{(2)}_{i,l}$ is obtained, the basis functions in (3) can be calculated by modifying the $H(\operatorname{div},\Omega;\mathbb S)$ conforming basis functions from \cite{hu2014family}. More precisely, let $\psi_{i,l}\in P_{k-1}(e_i)$ denote the dual basis of $\bd\bta^{(2)}_{i,l}\cdot\mn|_{e_i}$ in the $L^2(e_i)$ sense. For any $\bta\in P_k(K;\mathbb{S})$, the modified function
\begin{align}\label{eq:basisN2D}
\bta^{(3)}=\bta-\sum_{i=1}^3\sum_{l=1}^k\langle \bd\bta\cdot\mn,\psi_{i,l}\rangle_{e_i}\bta^{(2)}_{i,l}
\end{align}
satisfies
$\bd\bta^{(3)}\cdot\mn|_{\partial K}=0$ and $\bta^{(3)}\mn|_{\partial K}=\bta\mn|_{\partial K}$.
Thanks to \eqref{trigger:1}, the explicit representation of $\bm \tau^{(3)}_{i,j,l}$ is presented in (3).  
\end{remark}

\begin{remark}
    A simple modification of basis functions (1)--(4) gives rise to the basis of the $H(\bd\bd,\mathbb{S})$ conforming elements in \cite{chen2020finite2d}. Since any function $\bta$ in (1)--(2) satisfies $\bta\mn|_{\partial K}=0$, this shows $\big(\partial_t(\mt^T\bta\mn)+\bd\bta\cdot\mn\big)|_{\partial K}=\bd\bta\cdot\mn|_{\partial K}$. This means that functions in (1) are bubble functions of \cite{chen2020finite2d} and functions in (2) are basis functions with respect to the degrees of freedom $\big(\partial_t(\mt^T\bta\mn)+\bd\bta\cdot\mn\big)|_{e_i}$. In order to satisfy $\big(\partial_t(\mt^T\bta\mn)+\bd\bta\cdot\mn\big)|_{\partial K}=0$, the functions in (4) with respect to $a_i$ shall be replaced by
    \begin{align*}
        \bm \tau^{(4)}_{i,j}+\frac{3(\mathbf n_i^{\mathrm T}\mathbf t_{i+1})\mathbf t_{i+1}^{\mathrm T}\widetilde{\bm S}_{j}}{h_i}\mathbf n_{i+1}\bm \tau_{i+1,2}^{(2)}+\frac{3(\mathbf n_i^{\mathrm T}\mathbf t_{i+2})\mathbf t_{i+2}^{\mathrm T}\widetilde{\bm S}_{j}}{h_i}\mathbf n_{i+2}\bm \tau_{i+2,1}^{(2)}\quad \text{for}\, 1\leq j\leq 3.
    \end{align*}
Given edge $e_i$,
for any $\bta$ in (3) with $j=1$, $\mt^T\bta\mn|_{\partial K}=0$ shows  $\big(\partial_t(\mt^T\bta\mn)+\bd\bta\cdot\mn\big)|_{\partial K}=\bd\bta\cdot\mn|_{\partial K}=0$. This implies that the functions in (3) with $j=1$ are basis functions of \cite{chen2020finite2d} with respect to $\mn^T\bta\mn|_{e_i}$ in the interior of $e_i$. The remaining bubble functions of \cite{chen2020finite2d} are obtained by modifying
the functions in (3) with $j=2$ as follows
\begin{align*}
    &\bm \tau_{i,2}^{(3)}+\mathbf t_i^{\mathrm T}\bm T_i^{\perp,2}\mathbf n_i\left(\frac{\mathbf n_{i+2}^{\mathrm T}\mathbf t_i}{h_{i+2}}\bm \tau_{i,1}^{(2)}+\frac{2c_i\mathbf n_{i+1}^{\mathrm T}\mathbf t_i}{h_{i+1}}\lambda_1\lambda_2\lambda_3\bm T_{i+1,i+2}\right),\\
   &\bm \tau_{i,2}^{(3)}+\mathbf t_i^{\mathrm T}\bm T_i^{\perp,2}\mathbf n_i\left(\frac{\mathbf n_{i+1}^{\mathrm T}\mathbf t_i}{h_{i+1}}\bm \tau_{i,2}^{(2)}+\frac{2c_i\mathbf n_{i+2}^{\mathrm T}\mathbf t_i}{h_{i+2}}\lambda_1\lambda_2\lambda_3\bm T_{i+1,i+2}\right),\\
    &\bm \tau_{i,2}^{(3)}-c_i\lambda_i\left((\nabla \phi)^{\mathrm T}\mathbf t_i\right) (\mathbf t_i^{\mathrm T}\bm T_{i}^{\perp,2}\mathbf n_i)\bm T_{i+1,i+2}.
\end{align*}

\end{remark}

\end{appendices}

\end{document}